\documentclass[11pt,a4paper,oneside]{article}
\normalbaselines
\usepackage{amsmath}
\usepackage{amssymb}
\usepackage{appendix}
\usepackage{mathtools}
\usepackage{mathcomp}
\usepackage{textcomp}
\usepackage{amsthm}
\usepackage{enumerate}
\usepackage[auth-lg]{authblk}
\usepackage{cite}
\usepackage{tikz}
\usepackage{caption}
\usepackage{subcaption}
\usepackage{bm}
\usetikzlibrary{arrows}
\usepackage{multicol}
\usepackage{tikz-qtree}
\usepackage{rotating}
\usepackage{url}
\usepackage{array}
\newcommand{\tcolblue}{\textcolor{blue}}

\usepackage{blkarray}

\newtheorem{thm}{Theorem}[section]\theoremstyle{plain}
\newtheorem{theorem}[thm]{Theorem}\theoremstyle{plain}
\newtheorem{proposition}[thm]{Proposition}\theoremstyle{plain}
\newtheorem{lemma}[thm]{Lemma}\theoremstyle{plain}
\theoremstyle{plain}
\theoremstyle{plain}
\newtheorem{claim}[thm]{Claim}\theoremstyle{plain}
\theoremstyle{plain}
\theoremstyle{plain}
\newtheorem{conj}[thm]{Conjecture}\theoremstyle{plain}
\theoremstyle{plain}
\theoremstyle{plain}
\theoremstyle{plain}
\newtheorem{question}[thm]{Question}\theoremstyle{plain}

\theoremstyle{definition}
\theoremstyle{plain}

\theoremstyle{definition}
\theoremstyle{plain}

\DeclareMathOperator{\rank}{rank}

\DeclareMathOperator{\cl}{cl}

\newcommand{\sm}{\setminus}

\newcommand{\B}{{\cal B}}

\newcommand{\R}{{\mathbb R}}

\newcommand{\sym}{{\rm Sym}}
\newcommand{\lk}{{\rm lk}}

\newcommand{\bill}[1]{{\color{red} Bill: #1}}
\newcommand{\billn}[1]{{\color{purple} Billn: #1}}

\title{Symmetric Powers of Matroids}
\author{Bill Jackson\thanks{School of Mathematical Sciences, Queen Mary
University of London, Mile End Road, London E1 4NS, United Kingdom.
E-mail: b.jackson@qmul.ac.uk} and Shin-ichi Tanigawa\thanks{Department of Business Economics, School of Management, Tokyo University of Science, 
1-11-2 Fujimi, Chiyoda-ku,
Tokyo 102-0071,  Japan. Email: {tanigawa@rs.tus.ac.jp}}}

\date{\today}

\begin{document}

\maketitle

\begin{abstract}
The study of matroid products has become an active area of research, owing to their connections with tropical ideals and linear representability. In this paper, we study matroidal abstractions of the multilinearity of symmetric powers of vector spaces, using a duality between symmetric powers of matroids and abstract rigidity. 
These observations allow us to solve Mason’s conjecture concerning the equivalence of two definitions of a symmetric power of a matroid. 
We show that Mason’s conjecture holds for second symmetric powers of matroids whereas it fails for third symmetric powers. 
\end{abstract}

\section{Introduction}
Lov{\'a}sz~\cite{L}, Mason~\cite{M} and Las Vergnas~\cite{LV81} initiated the study of products of matroids over fifty years ago by abstracting the basic properties of tensor, symmetric, or skew-symmetric products of vector spaces. As noted in these classical references, there are no universal definitions which extend all the fundamental properties of these products
%tensor/symmetric/skew-symmetrics product of 
from vector spaces to matroids. In particular, matroid products may not be unique, and may not even exist when the matroid is not representable.

The study of matroid products has to a large extent been 
dormant since this early work 
%by Lov{\'a}sz, Mason, and Las Vergnas  
until recent results of  Draisma and Rinc{\'o}n~\cite{DR21},  and Anderson\cite{And}, which give connections between matroid products and the tropical ideals of Maclagan and Rinc{\'o}n~\cite{MR18},
%Tropical ideals are combinatorial analogues of ideals in the tropical polynomial semi-ring.
%Maclagan and Ric{\'o}n and Anderson proved a striking connection of tropical ideals to symmetric powers of matroids. (WILL MAKE MORE EXPLICIT.)
and B{\'e}rczi et al.~\cite{BGILMS25,BGILPS26}, which open a new research direction in matroid representability based on matroid products.
Despite these recent important applications, matroid products are still not well understand, even at a very basic level. Indeed, 
%the recent article of 
Anderson~\cite{And,AndThesis} recently pointed out that an old conjecture of Mason  (given below) concerning the equivalence of two definitions of a symmetric power of a matroid is still open. 

In this paper, we will study  matroid abstractions of the multilinearity property of symmetric powers of a vector space and use our results to solve Mason's conjecture.

\subsection{Symmetric Powers of Matroids and Mason's Conjecture}
Throughout this paper, $V$ will always denote a finite set.
%\bill{
For  an integer $k\geq 0$,
let $\sym_k(V)$ be the set of all unordered $k$-words over $V$. In particular, $\sym_0(V)=\{\emptyset\}$, where $\emptyset$ denotes the empty word, and  $\sym_1(V)=V$.
%Throughout the paper,  we distinguish between unordered words and multisets,
%and we identify $\sym_1(V)$ and $V$.
%We let $\sym_0(V)=\{\emptyset\}$. 
For $\alpha\in A\subseteq  \sym_i(V)$ and $\beta\in B\subseteq \sym_j(V)$,
let $\alpha\cdot \beta\in \sym_{i+j}(V)$ be the concatenation of $\alpha$ and $\beta$, and put 
$A\cdot B=\{\alpha\cdot \beta:\alpha\in A,\beta\in B \}$. When $A=\{\alpha\}$ we will often write 
$\alpha\cdot B$ for $A\cdot B$.
For example, 
when $V=\{1,2,3,4\}$ and $A=\{12\},B=\{12,13,22,24\}\subset\sym_2(V)$, we have 
$$A\cdot B=12\cdot B=\{1122, 1123, 1222, 1224\}\subset\sym_4(V).$$ 
%}
%For  an integer $k\geq 0$,
%let $\sym_k(V)$ be the set of all unordered $k$-tuples (equivalently, multisets of size $k$) which can be constructed using the elements of $V$. 
%%We let $\sym_0(V)=\{\emptyset\}$. 
%For $\alpha\in  \sym_i(V)$ and $\beta\in \sym_j(V)$,
%let $\alpha\cdot \beta\in \sym_{i+j}(V)$ be the union of $\alpha$ and $\beta$ as a multiset, %and let $\alpha\cdot \sym_j(V)=\{\alpha\cdot \beta:\beta\in \sym_j(V)\}$.

Suppose $M$ is  a matroid on $V$ with loop set $L_M$ and rank function $r_M$. 
%Let  $L_M$ be the set of loops of $M$. 
Following Anderson~\cite{And}, we define
a {\em $k$-th symmetric quasi-power} of $M$ to be a matroid $N$ on $\sym_k(V)$ which satisfies the following abstraction of multilinearity. 
%\bill{
\begin{description}
\item[Symmetric Q-power:] For all $1\leq i\leq k-1$ and all $\alpha_1,\alpha_2\in \sym_i(V\setminus L_M)$,
\[
\text{$N|_{\alpha_1\cdot \sym_{k-i}(V)}$ is isomorphic to $N|_{\alpha_2\cdot \sym_{k-i}(V)}$,}
\]
via the bijection $\alpha_1\cdot \beta\mapsto \alpha_2\cdot \beta$ for each $\beta\in \sym_{k-i}(V)$.
In addition, $N|_{\alpha\cdot V}$ is isomorphic to $M$ for all $\alpha\in \sym_{k-1}(V\setminus L_M)$, and $\alpha\cdot v$ is a loop in $N$ for all $\alpha\in \sym_{k-1}(V)$ and $v\in L_M$.
\end{description}
%}

%sequence of matroids $(N_1, N_2, \dots, N_k)$ with the following properties.
%\begin{enumerate}
%\item[S1:] $N_i$ is a matroid on $\sym_i(V)$ for $i=1,\dots, k$.
%\item[S2:] {$N_1=M$.}
%%$N_1$ is isomorphic to $M$ via the bijection $\{v\}\mapsto v$ between $\sym_1(V)$ and $V$.
%\item[S3:] for all $j\in \{1,\dots, k-1\}$  and all $\alpha\in V$,
% $N_{j+1}|_{\alpha \cdot \sym_j(V)}$ is isomorphic to $N_j$ via the bijection $\alpha \cdot \beta \mapsto \beta$ between $\alpha\cdot \sym_j(V)$ and $\sym_j(V)$ when $\alpha$ is not a loop in $M$, and
%otherwise $N_{j+1}|_{\alpha \cdot \sym_j(V)}$ is the rank zero matroid.
%\end{enumerate}
We say that a {$k$-th symmetric quasi-power} 
%$(N_1, N_2, \dots, N_k)$ $N$ 
is a {\em $k$-th symmetric power} 
of $M$ if it satisfies the following additional condition.

\begin{description}
\item[SP-rank Property:] The rank of $N$ is equal to $\binom{\rank M +k-1}{k}$.
\end{description}

Results of Mason~\cite{M} and Anderson~\cite{And} imply that the rank of any $k$-th symmetric quasi-power of $M$ is at most $\binom{\rank M +k-1}{k}$, and hence a symmetric power of $M$ is a symmetric quasi-power of maximum possible rank.

Lov\'asz~\cite{L} and Mason~\cite{M} consider slightly different families in their papers on matroid products, both of which use the following property of the 
matroid $N$.
%sequence $(N_1, N_2, \dots, N_k)$.

\begin{description}
\item[Flat Property:] For every flat $F$ of $M$, $F\cdot \sym_{k-1}(V)$
%$\{\alpha\in \sym_k(V): \text{$\alpha$ contains an element in $F$}\}$
is a flat of $N$.
\end{description}
%\bill{I THINK IT WOULD BE A GOOD IDEA TO REWRITE THIS AS $F\cdot \sym_{k-1}(V)$ AS IT WILL FIT BETTER WITH LATER SECTIONS.}

Lov\'asz~\cite{L} considered matroids $N$ on $\sym_k(V)$ satisfying the SP-Rank and Flat Properties, whereas Mason~\cite{M} considered  matroids satisfying Symmetric Q-power and Flat Properties. We have chosen to use Aderson's definition of symmetric power since we believe it gives a better abstraction of vector space symmetric powers to matroids. Anderson's definition  is more restrictive than the definitions of both Lov\'asz and Mason since \cite[Proposition 2.21]{And} implies that any matroid satisfying the Symmetric Q-power and SP-rank Properties will also satisfy the Flat Property.

%We will refer to any $k$-th product obtained in this way as a {\em geometric $k$-th tensor product} of $M$.

We may construct an example of a $k$-th symmetric power  of any representable matroid $M$ since, if $M$ is the column matroid of a matrix $A$, then we can take $N$ to be the column matroid of the $k$-th Macauly matrix of $A$, see \cite{And}. On the other hand, symmetric powers may not be unique, and may not even exist in the case when $M$ is not representable, see \cite{M,And}.
%However, Las Vergnas (for tensors) and Draisma and Rinc{\'o}n (for symmetric powers)
%showed that $k$-th symmetric powers may not exist for non-representable matroids.
%It is also well known that a given matroid may have several $k$-th symmetric powers. 

Mason~\cite{M} was very much aware of Lov\'asz's SP-rank property: he says on~\cite[Page 552]{M} that he suspects that SP-rank property will follow from Symmetric Q-power and Flat properties, and gives this as the first of two outstanding open questions for symmetric powers on~\cite[Page 554]{M}. This question is given as a formal conjecture by Anderson \cite[Conjecture 2.22]{And}.

\begin{conj}[Mason's Rank Conjecture] \label{conj:And}
Let $N$ be a $k$-th symmetric quasi-power of a matroid $M$.
Then the rank of $N$ is $\binom{\rank M +1}{k}$ if and only if 
%\[
%\{\alpha\in \sym_k(V): \text{$\alpha$ contains an element of $F$}\}
%\]
$F\cdot \sym_{k-1}(V)$ is a flat of $N_k$ for every flat $F$ of $M$.
\end{conj}

    We will verify Conjecture \ref{conj:And} for all matroids when $k=2$, and also for all uniform matroids other than the free matroid when $k=3$. In addition, we will show that Conjecture \ref{conj:And} does not hold for all matroids when  $k\geq 3$, and does not hold for uniform matroids when $k=4$. Our approach is to apply techniques from rigidity theory to the dual formulation of Conjecture \ref{conj:And}. It is inspired by a recent paper of Brakensiek et al.~\cite{brakensiek} which shows that the restriction of the geometric symmetric second power of a generic realisation of the uniform matroid $U_n^k$ in $\R^k$ to the edge set $E(K_n)$ of the complete graph $K_n$ is dual to the generic $(n-k-1)$-dimensional rigidity matroid on $E(K_n)$, and our paper \cite{JTfields} in which we consider the interplay between the family of symmetric second powers of $U_n^k$ restricted to $E(K_n)$ and the dual family of abstract $(n-k-1)$-rigidity matroids on $E(K_n)$.

%The rest of the paper is organized as follows.
The aim of this paper is not only to solve Mason's conjecture but also to clarify relationships between different abstractions of multilinearity.
We will describe these abstractions in Section~\ref{sec:Multilinearity}.
In Section~\ref{sec:rigidity}, we will describe the connection to rigidity theory.
In Section~\ref{sec:counterexample}, we will describe our counterexamples which show that Mason's Rank Conjecture does not hold in general.
In Section~\ref{sec:lemmas} and Section~\ref{sec:positive}, we will give the proofs of our positive results for $k=2,3$.

\subsection{Notation}
We will use the following notation throughout the paper.
Recall that, for a finite set $V$ and an integer $k\geq 0$,
$\sym_k(V)$ denotes the set of all unordered $k$-words over $V$ and that, for $\alpha
%=\{u_1,\dots, u_i\}
\in  \sym_i(V)$ and $\beta
%=\{v_1,\dots, v_j\}
\in \sym_j(V)$,
$\alpha\cdot \beta\in \sym_{i+j}(V)$ is the concatenation of $\alpha$ and $\beta$.

We will often 
refer to the  elements of $V$ as 
{\em letters} and the elements of $\sym_k(V)$ as {\em words}.
For $\alpha\in \sym_k(V)$, ${\rm supp}(\alpha)$ denotes the set  of distinct letters contained in $\alpha$.
%For $E\subseteq \sym_k(V)$, we will use $V(E)$ to denote the set of distinct elements appearing in $E$, i.e., $V(E)=\bigcup_{\alpha\in E}{\rm supp}(\alpha)$. \st{I guess this V( ) notation is not used.}

For $E\subseteq \sym_k(V)$ and $\alpha\in \sym_i(V)$, the {\em link of $\alpha$ in $E$} is $$\lk(\alpha,E)=\{\beta\in \sym_{k-i}(V): \alpha\cdot \beta\in E\}.$$
%, where $\{u,v\}$ is considered as a multiset.
For example, if $E=\{12,13,22,24\}$,
then $\lk(2,E)=\{1,2,4\}$.

%\bill{We will often simplify notation by associating the elements and subsets of $V$ with the corresponding elements and subsets of $\sym_1(V)$. For example, for $v\in V$, $X\subseteq V$ and $B\subseteq \sym_k(V)$ we will write $v\cdot B$ and $X\cdot B$ to mean $\{v\}\cdot B$ and $\sym_1(X)\cdot B$. }
%Similarly for $A\subset \sym_1(V)$ we will write $\cl_M(A)$ to mean 
%\cl_M(\{a:\{a\}\in A\})$.

%\bill{We will often simplify notation by associating the elements and subsets of $V$ with the corresponding elements and subsets of $\sym_1(V)$. For example, for $v\in V$, $X\subseteq V$ and $B\subseteq \sym_k(V)$ we will write $v\cdot B$ and $X\cdot B$ to mean $\{v\}\cdot B$ and $\sym_1(X)\cdot B$. }
%Similarly for $A\subset \sym_1(V)$ we will write $\cl_M(A)$ to mean 
%\cl_M(\{a:\{a\}\in A\})$.

We refer the reader to \cite{Oxley11} for  standard terminology and results on matroids which are not explicitly given in the text. 
For a matroid $M$ on $V$, we will use $r_M$ and ${\rm cl}_M$ to denote the rank function and the closure operator of $M$, respectively. 
%We also use $r_M$ to denote the rank of $M$ if it is clear from the context.
The dual matroid of $M$ is denoted by $M^*$.
The sets of loops and coloops in $M$ are denoted by $L_M$ and  $CL_M$, respectively. 
For $X\subseteq V$, we will use $M|_X$ to denote the restriction of $M$ to $X$ and we will say that  $X$ is {\em cyclic} in $M$ if $M|_X$ has no coloops. 
%In addition we will sometimes use our association between the elements of $V$ and $\sym_1(V)$ in statements about $M$. For example, for $A\subseteq \sym_1(V)$, we will write $\cl_M(A)$ to mean $\cl_M(\{a:\{a\}\in A\})$.

We next define two order relations on the set of all matroids with groundset $V$. The {\em weak order} $\preceq_{\rm w}$ is defined such that, 
for two matroids $M_1, M_2$ on  $V$, $M_1\preceq_{\rm w} M_2$ 
 if $r_{M_1}(X)\leq r_{M_2}(X)$ for all $X\subseteq V$.
 The {\em strong order} $\preceq_{\rm s}$ of matroids is defined such that, 
for two matroids $M_1, M_2$ on $V$, $M_1\preceq_{\rm s} M_2$ 
 if $r_{M_1}(Y)-r_{M_1}(X)\leq r_{M_2}(Y)-r_{M_2}(X)$ for all $X\subseteq Y\subseteq V$.
 By  taking $X=\emptyset$ in the definition of $\preceq_{\rm s}$, we see that 
 $M_1 \preceq_{\rm s} M_2$ implies $M_1 \preceq_{\rm w} M_2$.
 
 Our next two results give several equivalent definitions of the weak and strong orders. We refer the reader to \cite{Oxley11,Kun86s,KN86w} for their proofs.
 \begin{proposition}\label{prop:weak}
 Let $M_1$ and $M_2$ be matroids on a finite set $V$. Then
 the following statements are equivalent:
 \begin{enumerate}
     \item $M_1\preceq_{\rm w} M_2$;
     \item every independent set in $M_1$ is independent in $M_2$;
     %\item every spanning set in $M_2$ is spanning in $M_1$;
     \item every circuit in $M_2$ contains a circuit in $M_1$.
 \end{enumerate}
In addition,  when $M_1$ and $M_2$ have the same rank, we have $M_1\preceq_{\rm w} M_2$ if and only if $M_1^*\preceq_{\rm w} M_2^*$.
\end{proposition}

 \begin{proposition}\label{prop:strong}
 Let $M_1$ and $M_2$ be matroids on a finite set $V$.
 Then
 the following statements are equivalent:
 \begin{enumerate}
     \item $M_1\preceq_{\rm s} M_2$;
     \item every flat in $M_1$ is a flat in $M_2$;
     \item ${\rm cl}_{M_2}(X)\subseteq {\rm cl}_{M_1}(X)$ for all $X\subseteq V$;
     \item every circuit in $M_2$ is the union of circuits in $M_1$;
     \item $M_2^*\preceq_{\rm s} M_1^*$. 
 \end{enumerate}
 In addition,  when $M_1$ and $M_2$ have the same rank, we have $M_1\preceq_{\rm s} M_2$ if and only if  $M_1=M_2$.
\end{proposition}

%See, e.g., for the proofs.
%It should be also 
%Note that, under the assumption that $M_1$ and $M_2$ have the same rank,  
%\begin{itemize}
%    \item $M_1\preceq_{\rm s} M_2$ if and only if  $M_1=M_2$, and
%    \item $M_1\preceq_w M_2$ if and only if $M_1^*\preceq_w M_2^*$. 
%\end{itemize}

% JUST FOR OUR CONVENIENCE, LET ME GIVE THE PROOFS:
% \begin{proof}[Proof of Proposition~\ref{prop:weak}]
% 1 implies 2: For every independent $I$ in $M_1$, 1 implies $|I|=r_{M_1}(I)\leq r_{M_2}(I)$.

% 2 implies 1: If $r_{M_1}(X)>r_{M_2}(X)$ holds for some $X$, then 
% an $M_1$-base of $X$ is dependent in $M_2$.

% 2 and 4 are equivalent: To this this observe that, 3 holds if and only if every dependent set in $M_2$ is dependent in $M_1$.

% Thus, 1, 2, and 4 are equivalent.

% 2 implies 3: Pick any spanning set $S$ of $M_2$, and let $B$ be an $M_1$-base of $S$.
% If $X$ is not spanning in $M_1$, then $B+e$ is independent in $M_1$ for some $e\notin X$,

% \end{proof}
%\section{Abstract Multilinearity and Rigidity}

\section{Further Properties and Duality}
Throughout the section, $M$ will denote a matroid on $V$
and $N$ will denote a matroid on $\sym_k(V)$ for some positive integer $k$. We will describe various properties of  $N$ which are inspired by the operation of taking the  $k$-th symmetric  power of a vector space. We will state our main results, which give equivalences between these properties. We will then reformulate the properties and results in terms of the dual matroids $M^*$ and $N^*$.

\label{sec:Multilinearity}
\subsection{Abstract multilinearity}\label{subsec:Multilinearity}
%Given a matroid with groundset $V$, we shall describe various properties of %matroids on $\sym_k(V)$ which are inspired by the multilinearity of the map from %a vector space to its $k$-th symmetric  power.
%}.
%Recall that $L_M$ denotes the set of loops in $M$.

% \iffalse
% %\bill{
% We first recall the property 
% that $N$ must satisfy in order to be  a $k$-th  symmetric quasi-power of $M$.
% %power inspired by properties S2 and S3 in the definition of a $k$'th  symmetric quasi-power. 
% %}
% \begin{description}
% \item[Symmetric Q-power:] For all $1\leq i\leq k-1$ and all $\alpha_1,\alpha_2\in \sym_i(V\setminus L_M)$,
% \[
% \text{$N|_{\alpha_1\cdot \sym_{k-i}(V)}$ is isomorphic to $N|_{\alpha_2\cdot \sym_{k-i}(V)}$,}
% \]
% via the bijection $\alpha_1\cdot \beta\mapsto \alpha_2\cdot \beta$ for each $\beta\in \sym_{k-i}(V)$. In addition, $N|_{\alpha\cdot V}$ is isomorphic to $M$ for all $\alpha\in \sym_{k-1}(V\setminus L_M)$, and $\alpha\cdot v$ is a loop in $N$ for all $\alpha\in \sym_{k-1}(V)$ and $v\in L_M$. 
% \end{description}

% Our next property is a weaker abstraction of multilinearity.
% \begin{description}
% \item[Multilinearity:] 
% For all $\tau \in \sym_{k-1}(V\setminus L_M)$,
% \[
% \text{$N|_{\tau \cdot V}$ is isomorphic to $M$} 
% \]
% via the bijection $\tau \cdot v \mapsto v$ between $\tau \cdot V$ and $V$. And
% $\tau \cdot v$ is a loop in $N$ for all $\tau \in \sym_{k-1}(V)$ and $v\in L_M$.
% %for all $\tau \in L_M\cdot \sym_{k-2}(V)$,  $N|_{\tau \cdot \sym_1(V)}$ is the rank zero matroid on $\tau \cdot \sym_1(V)$.
% \end{description}
% \fi 
Our first property is a slight weakening of the Symmetric Q-power property given in the introduction. It is a natural analogue of the observation that a  multilinear $k$-form becomes a linear form if one fixes $k-1$ arguments. 
\begin{description}
\item[Multilinearity:] 
For all $\tau \in \sym_{k-1}(V\setminus L_M)$,
\[
\text{$N|_{\tau \cdot V}$ is isomorphic to $M$} 
\]
via the bijection $\tau \cdot v \mapsto v$ between $\tau \cdot V$ and $V$. And
$\tau \cdot v$ is a loop in $N$ for all $\tau \in \sym_{k-1}(V)$ and $v\in L_M$.
\end{description}

Note that, in order to be a $k$-th  symmetric quasi-power of $M$, 
$N$ must satisfy the additional symmetry condition given in the Symmetric Q-power property from the Introduction. 

%\begin{description}
%\item[Symmetric Q-power:] $N$ satisfies Multilinearity, and for all $1\leq i\leq %k-1$ and all $\alpha_1,\alpha_2\in \sym_i(V\setminus L_M)$,
%\[
%\text{$N|_{\alpha_1\cdot \sym_{k-i}(V)}$ is isomorphic to $N|_{\alpha_2\cdot %\sym_{k-i}(V)}$,}
%\]
%via the bijection $\alpha_1\cdot \beta\mapsto \alpha_2\cdot \beta$ for each %$\beta\in \sym_{k-i}(V)$. 
% In addition, $N|_{\alpha\cdot V}$ is isomorphic to $M$ for all $\alpha\in %\sym_{k-1}(V\setminus L_M)$, and $\alpha\cdot v$ is a loop in $N$ for all %$\alpha\in \sym_{k-1}(V)$ and $v\in L_M$. 
%\end{description}

% Multilinearity is a weaker property than having a quasi-symmetric power sequence since 
% for the latter property the isomorphism between 
% $N|(\tau_1\cdot \sym_{k-i}(V))$
% and $N|(\tau_2\cdot \sym_{k-i}(V))$
% is necessary for all $1\leq i\leq k$ and all $\tau_1,\tau_2\in \sym_i(V)$ containing no loops.

By relaxing the isomorphism requirement in Multilinearity  to a strong or weak order requirement, we may obtain  two weaker properties. To do this we will need to extend our matroid order relations to allow relabelling of the elements of a matroid. Given two matroids $M_1=(V_1,r_1)$ and $M_2=(V_2,r_2)$ and a bijection $f:V_2\to V_1$ we say that {\em $M_1 \preceq_{\rm w} M_2$ through $f$} if $r_1(f(X))\leq r_2(X)$ for all $X\subseteq V_2$. The corresponding extension of the strong order is defined analogously.
\begin{description}
\item[Strong (resp. Weak) Order Multilinearity:] For all $\tau \in \sym_{k-1}(V\setminus L_M)$,
\[
 M^*\preceq_{\rm s} (N|_{\tau \cdot V})^* \quad \text{(resp., $M^*\preceq_{\rm w} (N|_{\tau \cdot V})^*$)}
 \]
 via the bijection $\tau \cdot v \mapsto v$ between $\tau \cdot V$ and $V$. And
 $\tau\cdot v$ is a loop in $N$ for all $\tau\in \sym_{k-1}(V)$ and $v\in L_M$.
 %for all $\tau \in (L_M\cdot \sym_{k-2}(V))$,  $N|_{\tau \cdot V}$ is the rank zero matroid on $\tau \cdot V$.
\end{description}

Since the strong order for matroids  is compatible with matroid duality (by reversing the order), 
Strong Order Multilinearity can be defined equivalently  by the condition that $M\succeq_{\rm s} N|_{\tau \cdot V}$.
This is not the case for the weak order in general.
The primal version of Weak Order Multilinearity can be  defined similarly, but it is too weak for our purposes. 
%We will give an example later.

The isomorphism given by the Multilinearity property and the equivalence between the fourth and fifth  items in Proposition~\ref{prop:strong} imply that
our next property lies between Multilinearity and Strong Order Multilinearity.
\begin{description}
\item[Circuit Multilinearity:] For all $\tau \in \sym_{k-1}(V\setminus L_M)$
and every circuit $X$ of $M$, 
$\tau\cdot X$ is a circuit of $N$. And
 $\tau\cdot v$ is a loop in $N$ for all $\tau\in \sym_{k-1}(V)$ and $v\in L_M$.
 %for all $\tau \in L_M\cdot \sym_{k-2}(V)$,  $N|_{\tau \cdot \sym_1(V)}$ is the rank zero matroid on $\tau \cdot \sym_1(V)$.
\end{description}
%By Propositions~\ref{prop:weak} and \ref{prop:strong},
In summary, we have {\small 
\begin{equation}\label{eq:biliearity_implication}
\text{Sym.~Q-power} \Rightarrow 
\text{Multilinearity} \Rightarrow
\text{Circuit M.} \Rightarrow
\text{Strong Order M.}\Rightarrow
\text{Weak Order M.}
\end{equation}
}
% Since weak order is not compatible with duality, we may also consider the following:
% \begin{description}
% \item[Dual Weak Order Multilinearity:] for all $\tau \in \sym_{k-1}(V)$,
%  $M^*$ has a weak map to $(N|_{\tau \cdot \sym_1(V)})^*$ via the bijection $\tau \cdot v \mapsto v$ between $\tau \cdot \sym_1(V)$ and $V$ when $v$ is not a loop in $M$.
%  When $v$ is a loop in $M$,  $N|_{\tau \cdot \sym_1(V)}$ is the rank zero matroid.
% \end{description}

\subsection{Subspace properties}

We also consider three properties of $N$ obtained by abstracting the properties of subspaces of the symmetric power of a vector space. The first was given in the Introduction.
\begin{description}
\item[Flat Property:]
For every flat $F$ of $M$, $F\cdot \sym_{k-1}(V)$ is a flat of $N$.
%\item[Hyperplane Property:]
%For every hyperplane $X$ in $M$, $X\cdot \sym_{k-1}(V)$ is a hyperplane in $N$. 
\item[Closure Property:]
For all $X\subseteq V$, ${\rm cl}_M(X)\cdot \sym_{k-1}(V)\subseteq {\rm cl}_N(X\cdot \sym_{k-1}(V))$.
\item[Strong  SP-rank Property:] For all $X\subseteq V$, 
$$r_N(X\cdot \sym_{k-1}(V))={\rank M+k-1\choose k}-{\rank M-r_M(X)+k-1\choose k}.$$
\end{description}
We also recall the following special case of the Strong SP-rank Property 
%when $X=V$  
from the Introduction.

\begin{description}
\item[SP-rank Property:] The rank of $N$ is ${\rank M+k-1\choose k}$.
\end{description}

% \subsection{Connection to Symmetric Power}
% Let $N$ be a matroid on $\sym_k(V)$.
% We say that $N$ is {\em symmetric} if, for any integer $0\leq i\leq k$ and any $\tau_1,\tau_2\in \sym_i(V)$, 
% $N|(\tau_1\cdot \sym_{k-i}(V))$ is isomorphic to 
% $N|(\tau_2\cdot \sym_{k-i}(V))$ though the bijection
% $\tau_1 \cdot \alpha \mapsto \tau_2\cdot \alpha$ for $\alpha\in \sym_{k-i}(V)$.
% \begin{lemma}\label{lem:connection}
% Let $M$ be a matroid on $V$ 
% and $N$ be a matroid on $\sym_k(V)$.
% Then, there is a quasi-symmetric power (resp., symmetric power) $(N_1,\dots, N_k)$ with $N_1=M$ and $N_k=N$
% if and only if $N$ is symmetric and satisfies 
% Multilinearity (resp., Multilinearity and Rank Property) with respect to $M$.
% \end{lemma}
% \begin{proof}
% Clearly, the existence of a quasi-symmetric power sequence implies the symmetry and Multilinearity of $N$.

% We prove the converse direction by induction on $k$.
% The base case when $k=1$ is obvious.
% Suppose $N$ is symmetric and satisfies Multilinearity.
% Then, by symmetry, $N_i$ is defined such that $N_i$ is isomorphic to $\tau\cdot \sym_{k-i}(V)$
% for all $\tau\in \sym_i(V)$.
% Then, $N_{k-1}$ also satisfies symmetry and Multilinearity.
% Indeed, 

% \end{proof}
\subsection{Statements of main results}\label{subsec:results}
%We show Multilinearity of matroids behaves in a very different way depending on $k=2$ or $k>2$.
We will prove the following equivalences in the case when $k=2$.
\begin{theorem}\label{thm:2}
Let $M$ be a matroid on a finite set $V$ and $N$ be a matroid on $\sym_2(V)$.
Then the following statements are equivalent.
\begin{itemize}
\item $N$ is a second symmetric power of $M$.
\item $N$ satisfies SP-rank Property and Multilinearity.
\item $N$ satisfies SP-rank Property and Circuit Multilinearity.
\item $N$ satisfies SP-rank Property and Strong Order Multilinearity.
\item $N$ satisfies Flat Property and Strong Order Multilinearity.
\end{itemize}
\end{theorem}
Theorem~\ref{thm:2} verifies Mason's Rank Conjecture for second symmetric powers. We have already seen that necessity follows from \cite{And}. Sufficiency follows since, if $N$ is a   second symmetric quasi-power of $M$  which satisfies the Flat Property, then $N$ also satisfies Multilinearity, and hence $N$  satisfies the SP-rank Property by the equivalence of the second and fourth items in Theorem ~\ref{thm:2}.
%since $N_2$ in a quasi-symmetric power sequence satisfies (Strong Order) Multilinearity with respect to $N_1$.

An example given later shows that we cannot replace Strong Order Multilinearity with Weak Order Multilinearity in Theorem~\ref{thm:2}.
We will show, however,  that the combination of the SP-rank Property and Weak Order Multilinearity gives the class of matroids considered by Lov\'asz by proving the following.
\begin{theorem}\label{thm:2_weak}
Let $M$ be a matroid on a finite set $V$ and $N$ be a matroid on $\sym_2(V)$.\\[1mm]
%Then 
(a) The following statements are equivalent.
\begin{itemize}
\item $N$ satisfies SP-rank Property and Flat Property.
\item $N$ satisfies Flat Property and Weak Order Multilinearity.
\item $N$ satisfies Strong  SP-rank Property and Weak Order Multilinearity.
\item $N$ satisfies SP-rank Property, Closure Property, and Weak Order Multilinearity.
\end{itemize}
(b)
%In addition, 
There exists a matroid on $\sym_2(V)$ that satisfies SP-rank Property and Flat Property,
but does not satisfy Multilinearity.
\end{theorem}

%As noted in the Introduction, Lov{\'a}sz' and Mason considered sequences of matroids of length $k$ which satisfy properties S1,S4,S5 and S1,S2,S3,S5, respectively, in their original papers.
%he defined symmetric powers of matroids only by using Rank Property and Flat Property.
Theorem~\ref{thm:2} 
%and \ref{thm:2_weak} 
implies that, when $k=2$, the matroids considered by Mason are precisely the same as Anderson's second symmetric powers. On the other hand, Theorem \ref{thm:2_weak} tells us that, when $k=2$, the family of matroids considered by Lov\'asz is a strictly larger family than Anderson's second symmetric powers.

\medskip

Our next result shows that the situation differs significantly when $k>2$, and tells us, in particular, that Mason's Rank Conjecture is false when $k\geq 3$.
\begin{theorem}\label{thm:k}
Let $M$ be a matroid on a finite set $V$ and $N$ be a matroid on $\sym_k(V)$
for some $k\geq 2$.\\[1mm]
(a) If  $N$ satisfies SP-rank Property and Strong Order Multilinearity with respect to $M$,
then $N$ satisfies Multilinearity, Flat Property, Strong SP-rank Property, and Closure Property.
\\[1mm]
(b) If $k\geq 3$, then there exists a matroid on $\sym_k(V)$ that satisfies Flat Property and Symmetric Q-power, but does not satisfy SP-rank Property.
\end{theorem}

%\st{THE FOLLOWING SENTENCE PARTIALLY OVERLAPS WITH THE NEW THEOREM, THEOREM 2.5.}
%More specifically, we will see that: Flat Property and Strong Order Multilinearity
%imply SP-rank Property  when $k=3$ and $M$ is a uniform matroid other than the free matroid (and hence Mason's Rank Conjecture holds in this case); Flat Property and Symmetric Q-power do not imply  SP-rank Property
%when $k\geq 4$ even if $M$ is uniform  (and hence Mason's Rank Conjecture does not hold in this case).

 We will also verify the following equivalence involving Weak Order Multilinearity.
In particular, it extends the result of Anderson~\cite{And} that every $k$'th symmetric power $N$ of $M$ has the Flat Property by showing that the same conclusion holds for all matroids $N$ on $\sym_k(V)$ which satisfy SP-rank Property and Weak Order Multilinearity.

\begin{theorem}\label{thm:k_weak}
Let $M$ be a matroid on a finite set $V$ and $N$ be a matroid on $\sym_k(V)$ for some $k\geq 2$.
%\\[1mm]
%(a) 
Then the following statements are equivalent.
\begin{itemize}
\item $N$ satisfies Strong  SP-rank Property and Weak Order Multilinearity.
\item $N$ satisfies SP-rank Property, Closure Property, and Weak Order Multilinearity.
\end{itemize}
In addition, if $N$ satisfies Strong  SP-rank Property and Weak Order Multilinearity, then $N$ satisfies Flat Property.
%\\[1mm]
%On the other hand, when 
%(b) If $k\geq 3$, then there exists a matroid on $\sym_k(V)$ that satisfies 
%SP-rank Property, Flat Property, and Weak Order Multilinearity, but does not %satisfy Strong SP-rank Property or Closure Property.
\end{theorem}
%\st{MAYBE GOOD TO ADD THE FOLLOWING EXTRA COMMENT?: Anderson~\cite{And} shows that any symmetric power $N$ satisfies the Flat Property. Theorem~\ref{thm:k_weak} shows that the Flat Property follows from a much weaker condition.} 

Our final result implies that Mason's Rank Conjecture holds in the special case when $M$ is a uniform matroid distinct from the free matroid  and $k=3$ , but is false even in this special case when $k=4$.

\begin{theorem}\label{thm:mason_k_3}
Let $M$ be a uniform matroid distinct from the free matroid. \\
(a) If $N$ is a matroid on $\sym_3(V)$ which satisfies Flat Property
and Multilinearity, then $N$ satisfies 
SP-rank Property.\\
(b) There exists a matroid $N$ on $\sym_4(V)$ which satisfies Flat Property
and Multilinearity, but does not satisfy  SP-rank Property.
\end{theorem}

\subsection{Dual Formulations}\label{subsec:dual}
%\bill{
We will verify the results in Subsection~\ref{subsec:results} by proving their dual statements.
As a preliminary step, we will formulate the dual properties to those described in Subsection~\ref{subsec:Multilinearity} and then state the   dual forms of the results of Subsection~\ref{subsec:results} in this subsection. Several of the dual properties are reminiscent of fundamental properties from combinatorial rigidity theory. This link will be described  in more detail in this subsection and the next section. 

We will continue to use $M$ to denote a matroid on $V$ and 
$N$ to denote a matroid on $\sym_k(V)$.  Let $n=|V|$. 
We will consider the following properties of $N$ with respect to $M$.
\begin{description}
\item[Dual Symmetric Q-power:] For all $1\leq i\leq k-1$ and all $\alpha_1,\alpha_2\in \sym_i(V\setminus CL_M)$, 
$N/(\sym_k(V)\setminus (\alpha_1\cdot  \sym_{k-i}(V))$ is isomorphic to 
$N/(\sym_k(V)\setminus (\alpha_2\cdot  \sym_{k-i}(V))$. In addition, $N/(\sym_k(V)\setminus (\alpha\cdot  V))$  is isomorphic to $M$ for all $\alpha\in \sym_{k-1}(V\setminus CL_M)$
and $\alpha\cdot v$ is a coloop in $N$ for all $\alpha\in \sym_{k-1}(V)$ and $v\in CL_M$.
%$N/(\sym_k(V)\setminus (\alpha\cdot \sym_{k-i}(V)))$ is a free matroid for all $\alpha\in (CL_M\cdot \sym_{i-1}(V))$ and all $1\leq i\leq k-1$.
\item[Dual Multilinearity (DM):] For all $\tau\in \sym_{k-1}(V\setminus CL_M)$, 
$N/(\sym_k(V)\setminus (\tau\cdot  V))$ is isomorphic to $M$ via the bijection $\tau\cdot v \mapsto v$ between  $\tau\cdot V$ and $V$. And, 
$\tau\cdot v$ is a coloop in $N$ for all $\tau\in \sym_{k-1}(V)$ and $v\in CL_M$.
%for all $\alpha\in (CL_M\cdot \sym_{i-1}(V))$,
%$N/(\sym_k(V)\setminus (\alpha\cdot \sym_{k-i}(V)))$ is a free matroid . \bill{SHOULD $i$ BE $k-1$?}
\item[Rigidity Rank Property (RRP):] 
The rank of $N$ is 
\[{n+k-1\choose k}-{n-\rank M+k-1\choose k}.\]
\item[Strong Rigidity Rank Property (StrongRRP):] For all $X\subseteq V$, 
\[r_N(\sym_k(X))= {|X|+k-1\choose k}-{|X|-r_M(X)+k-1\choose k}.\] 
\newpage
\item[Coning Property (ConingP):] 
For all $X\subseteq V$ and $v\in V\setminus {\rm cl}_M(X)$, $v\cdot \sym_{k-1}(X+v)$ is a set of coloops in $N|_{\sym_k(X+v)}$.
\item[Weak Coning Property (WConingP):] 
For all $v\in CL_M$, $v\cdot \sym_{k-1}(V)$ is a set of coloops in $N$.
%\item[Circuit Property (CP)] 
%For every circuit $X$ of $M$, $\sym_k(X)$ is a circuit of $N$.
\item[Cyclic Property (CycP):]
For every cyclic set $X$ of $M$, $\sym_k(X)$ is a cyclic set of $N$.
\item[Cocircuit Property (CoCP):] 
For every 
$\tau\in \sym_{k-1}(V\setminus CL_M)$ and every cocircuit $X$ of $M$, 
$\tau \cdot X$ is a cocircuit of $N$.
And $\tau\cdot v$ is a coloop in $N$ for all $\tau\in \sym_{k-1}(V)$ and $v\in CL_M$.
%for every 
%$\tau\in \sym_{k-1}(V)$ and every coloop $v\in CL_M$, we have
%$\tau \cdot v$ is a coloop of $N$.
% \item[Loop Property (LP)]
% For every 
% $\tau\in \sym_{k-1}(V\setminus CL_M)$ and every loop $v\in L_M$,
% $\tau \cdot v$ is a loop of $N$.
% \item[One-point Gluing Property (1GP)] 
%  Let $X\subseteq V$, $v\in {\rm cl}_M(X)$, and $Y\subseteq X+v$ such that
%   ${\rm cl}_M(X)={\rm cl}_M(Y)$. Then \bill{$$\sym_k(X+v)\subseteq  {\rm cl}_N(\sym_k(X)\cup (v\cdot \sym_1(Y)\cdot \sym_{k-2}(X+v))).$$}
\item[Extension Property (ExtP):]
For every $N$-independent set $I\subseteq\sym_k(V)$,  every $\tau\in \sym_{k-1}(V)$, and every $v\in V$ such that $v\notin {\rm cl}_M(\lk(\tau,I))$, we have
%if $I$ is , then 
$I+v\cdot \tau$ is $N$-independent.
\item[0-Extension Property (0-ExtP):]
For every  $M$-independent set $X\subseteq V$, every  $N$-independent set $I\subseteq\sym_k(V)$, and every $\tau\in \sym_{k-1}(V)$ with $\lk (\tau,I)=\emptyset$, we have
%if $X$ is , then
$I\cup (\tau\cdot X)$ is $N$-independent. 
%\st{$I\cup (\tau\cdot \sym_1(X))$??}
\end{description}

\medskip

When $k=2$ and $M$ is a matroid 
%on $V$ 
of rank $r$, the rank condition of RRP becomes
\[
\rank N={n+2-1\choose 2}-{n-r+2-1\choose 2}=rn-{r\choose 2}.
\]
This  is precisely the rank formula for the rank-$r$ symmetric completability matroid on the looped complete graph $K_n^{\circ}$, see for example~\cite{SC,JacksonJordanTanigawa2014}. We can obtain 
the $(r-1)$-dimensional rigidity matroid $\mathcal{R}_{d-1}$ on $K_n$ from this matroid by contracting the $n$ loop edges of $K_n^{\circ}$. This connection will be used in Section~\ref{sec:rigidity}.
%By setting $d=r-1$ and subtracting $n$, we obtain $dn-{d+1\choose 2}$,
%which is the well-known rank formula for the generic $d$-dimensional rigidity matroid $\mathcal{R}_d$. 
%\st{HOW ABOUT THIS?: This $rn-{r\choose 2}$ is precisely the rank formula for the rank-$r$ symmetric completability matroid on $K_n^{\circ}$ and that we can obtain 
%the $(r-1)$-dimensional rigidity matroid on $K_n$ from this matroid by contracting $n$ loops, see, e.g.~\cite{SC,JacksonJordanTanigawa2014}. This connection will be used in Section~\ref{sec:rigidity}.}

The 0-Extension  Property is also inspired by the 0-extension property in rigidity theory: 
%A well-known 0-extension property for rigidity matroids states that, 
if $I$ is independent in 
%the generic $d$-dimensional rigidity matroid 
$\mathcal{R}_d$, $v$ is a vertex not incident to any edge in $I$, and $F$ is a set of at most $d$ edges incident to $v$, then $I\cup F$ is independent in $\mathcal{R}_d$. The 0-ExtP property gives a natural extension of this property to $k$-uniform hypergraphs (by considering the words in $\sym_k(V)$ as hyperedges on the vertex set $V$).

ExtP is an incremental, edge-wise 
strengthening of 0-ExtP. The restriction of the two properties  to generic rigidity matroids are equivalent but  they are substantially different for duals of arbitrary symmetric powers. We will show in a forthcoming paper that, even their restrictions to non-generic rigidity matroids can be very different. 
%\bill{SHOULD WE GIVE AN EXAMPLE?}\tcolblue{Yes, but in the next paper.}

Further comments on the connection to rigidity theory will be given in Section~\ref{sec:rigidity}.
 
\medskip

%\bill{
Our aim is to show that the above properties are dual to the properties described in Section~\ref{subsec:Multilinearity}.  We will need the following link between ExtP and ConingP to verify the special case that ExtP is the dual property to Strong Order Multilinearity.
%}

\begin{lemma}\label{lem:extension_coning}
Suppose $N$ satisfies ExtP. Then $N$ satisfies ConingP.
\end{lemma}
\begin{proof}
To verify that ConingP holds for $N$, choose $X\subset V$
%, $I\subseteq \sym_k(X)$ 
and 
$v\in V\setminus  {\rm cl}_M(X)$. We need to show that 
%$I\cup (v\cdot \sym_{k-1}(X+v))$ is independent in $N$ if  $I$ is independent in $N$. 
$v\cdot \sym_{k-1}(X+v)$ is a set of coloops of $N|_{\sym_{k}(X+v)}$.

Let $I$ be a base of $N|_{\sym_{k}(X)}$ and, for each integer $j\geq 1$, let  $v^j$ denote the word in $\sym_j(V)$ consisting of $j$ copies of $v$.
We will show that 
$$I_i:=I\cup \bigcup_{1\leq j\leq i}v^j\cdot \sym_{k-j}(X)$$ is 
$N$-independent
%a set of coloops of $N|\sym_{k}(X+v)$ 
for all $0\leq i\leq k$ by induction on $i$.
The base case when $i=0$ and $I_i=I$ holds by the choice of $I$.

Suppose $I_{i-1}$ is $N$-independent for some $i\geq 1$. To show that $I_i$ is $N$-independent it will suffice to
%is $N$-independent.
 show 
that $v^i\cdot \beta$ is a coloop in $N|_{I_i}$  for each $\beta\in \sym_{k-i}(X)$.
Observe that, for each $e\in I_{i}\setminus \{v^i\cdot \beta\}$, 
either the multiplicity of $v$ in $e$ is less than $i$
or $e$ is of the form $e=v^i\cdot \beta'$ for some $\beta'\in \sym_{k-i}(X)$ with $\beta'\neq \beta$.
Hence, if $e$ contains $v^{i-1}\cdot \beta$,
then $e=v^{i-1}\cdot \beta\cdot u$ for some $u\in X$.
This implies that $\lk(v^{i-1}\cdot \beta, I_i\setminus \{v^i\cdot \beta\})\subseteq X$.
Since  $v\notin {\rm cl}_M(X)$, this tells us that $v\notin {\rm cl}_M(\lk(v^{i-1}\cdot \beta, I_i\setminus \{v^i\cdot \beta\}))$.
ExtP now implies that adding $v\cdot v^{i-1}\cdot \beta$ to any base of $N|_{I_i\setminus \{v^i\cdot \beta\}}$ preserves $N$-independence. Hence $v^{i}\cdot \beta$ is a coloop in $N|_{I_i}$.

Thus, $I_i$ is $N$-independent for all $i=0,\dots, k$. 
Since $I$ is a base of  $N|_{\sym_k(X)}$, 
$$r_N(\sym_k(X+v))\geq |I_k|=r_N(\sym_k(X))+|v\cdot \sym_{k-1}(X+v)|.$$
Since $\sym_k(X+v)\setminus \sym_k(X)=v\cdot \sym_{k-1}(X+v)$, 
we have  equality in the above inequality,
and every edge in $v\cdot \sym_{k-1}(X+v)$ is a coloop in $N|_{\sym_{k}(X+v)}$.
% Since $I_k$ spans $N|_{\sym_k(X+v)}$, $I_k$ is a base of $N|_{\sym_k(X+v)}$. This gives
% $$r_N(\sym_k(X+v)=|I_k|=r_N(\sym_k(X))+|v\cdot \sym_{k-1}(X+v)|$$
% and hence every edge in $v\cdot \sym_{k-1}(X+v)$ is a coloop in $N|_{\sym_{k}(X+v)}$.
\end{proof}

\begin{lemma}\label{lem:duality}
Let $M$ be a matroid on $V$ and $N$ be a matroid on $\sym_k(V)$.
%\bill{
Then the equivalences between the specified properties of $N$  (with respect to $M$) and $N^*$ (with respect to $M^*$) given by matroid duality are as shown in the rows of the table below.
%}
%With respect to $M$, there is an equivalence between a property of $N$
%and that of $N^*$ as shown below:
\medskip
\begin{center}
\begin{tabular}{c|c}
    \hline 
$N$, $M$ & $N^*$, $M^*$ \\\hline
Rigidity Rank Property & SP-rank Property \\
Strong Rigidity Rank Property & Strong SP-rank Property \\
Coning Property & Closure Property \\
Cyclic Property & Flat Property \\
Dual Multilinearity & Multilinearity\\
Cocircuit Property & Circuit Multilinearity \\
Extension Property & Strong Order Multilinearity\\
0-Extension Property and Weak Coning Property & Weak Order Multilinearity \\\hline
  \end{tabular}
  \end{center}
\end{lemma}
\begin{proof}
The equivalence between Rigidity Rank Property and SP-rank Property follows directly from the rank formulae $\rank M^*=n-\rank M$ and $\rank N^*=|\sym_k(V)|-\rank N=\binom{n+k-1}{k}-\rank N$.

\medskip

{The equivalence between Strong Rigidity Rank Property and Strong SP-rank Property follows 
from the bijection $X\mapsto V\setminus X$,  the observation that $X\cdot \sym_{k-1}(V)=\sym_k(V)\setminus \sym_k(V\setminus X)$, and the rank formulae  $r_{M^*}(X)=|X|-\rank M+r_M(V\setminus X)$ and 
$r_{N^*}(\sym_k(V)\setminus \sym_k(V\setminus X))=|\sym_k(V)|-|\sym_k(V\setminus X)|-\rank N+r_{N}(\sym_k(V\setminus X))$. }
%\st{COMMENT: Is the last equation correct?} \bill{IS IT OK NOW?}

\medskip

To see the equivalence between Coning Property and Closure Property,
we use the general fact that, for a matroid $L$ and a partition 
$A\sqcup B\sqcup C$ of the ground set of $L$, $C$ is a set of coloops in $L|_{A\cup C}$ if and only 
if $C\subseteq {\rm cl}_{L^*}(B)$.
Now, in our problem, 
consider any partition $X\sqcup Y\sqcup \{v\}$ of $V$.
Then, $v\notin {\rm cl}_M(X)$ if and only if $v\in {\rm cl}_{M^*}(Y)$.
Similarly, since  
\[
\sym_k(X)\sqcup(Y\cdot \sym_{k-1}(V))\sqcup (v\cdot \sym_{k-1}(X+v)) \text{ is a partition of $\sym_k(V)$},
\]
 we have that
$v\cdot \sym_{k-1}(X+v)$  is a set of coloops in $N|_{\sym_k(X+v)}$ if and only if 
$v\cdot \sym_{k-1}(X+v)\subseteq {\rm cl}_{N^*}(Y\cdot \sym_{k-1}(V))$.
Since $(v\cdot \sym_{k-1}(V))\setminus (v\cdot \sym_{k-1}(X+v))\subseteq Y\cdot \sym_{k-1}(V)$, the latter property is equivalent to
$v\cdot \sym_{k-1}(V)\subseteq {\rm cl}_{N^*}(Y\cdot \sym_{k-1}(V))$.
Hence, Coning Property for $N$ is equivalent to the property that 
$v\cdot \sym_{k-1}(V)\subseteq {\rm cl}_{N^*}(Y\cdot \sym_{k-1}(V))$
for all $v\in {\rm cl}_{M^*}(Y)$.
And the latter property is equivalent to Closure Property for $N^*$.

\medskip

The equivalence between Cyclic Property and Flat Property is straightforward since 
$X$ is a cyclic set in $M$ if and only if $V\setminus X$ is a flat in $M^*$.

\medskip

The equivalences between Multilinearity and Dual Multilinearity, and between Cocircuit Property and Circuit Multilinearity are also straightforward. %\bill{IT IS NOT SO STRAIGHTFORWARD FOR THE LAST PART OF CIRCUIT Multi-linearITY. CAN WE REWRITE THIS IN THE SAME FORMAT AS COCIRCUIT PROPERTY ie for every $\tau\in \sym_{k-1}(V)$ and $v\in L_M$, $\tau\cdot v$ is a loop of $N$?}

%\medskip
% The equivalence between Cocircuit-Loop Property and Multilinearity follows since a matroid is determined by the family of cocircuits. 
% Indeed, for each $\tau\in \sym_{k-1}(V\setminus CL_M)$, define
% \[
% {\cal C}_{\tau}:=\{\tau\cdot \sym_1(X)\mid  X: \text{circuit in } M^*\}.
% \]
% If $N$ satisfies Cocircuit Property, each $\tau\cdot \sym_1(X)$ in ${\cal C}_{\tau}$ is a circuit in 
% $N^*|(\tau\cdot \sym_{1}(V))$. 
% Since ${\cal C}_{\tau}$ satisfies the circuit axiom, 
% ${\cal C}_{\tau}$ is the collection of all circuits of $N^*|(\tau\cdot \sym_{1}(V))$.

\medskip
%\newpage
We next prove that 
\begin{equation}\label{eq:ExtP}
\text{$N$ satisfies ExtP if and only if $N^*$ satisfies Strong Order Multilinearity.}
\end{equation}
%By Proposition~\ref{prop:strong},
%\begin{align}\nonumber
%\text{$N^*$ satisfies Strong Order Multilinearity} 
%&\Leftrightarrow N^*|\tau\cdot \sym_{1}(V)\preceq_{\rm s} M^* \quad  (\tau \in \sym_{k-1}(V)) \\
%&\Leftrightarrow N/ (\tau\cdot \sym_{1}(V))^c\succeq_{\rm s} M \quad  (\tau \in \sym_{k-1}(V)) %\label{eq:ext_strong1}
%%&\Leftrightarrow {\rm cl}_{N/(\tau\cdot \sym_{1}(V))^c}(X)\subseteq {\rm cl}_M(X) \quad (X\subseteq V).
%\end{align}
%\bill{
%I DONT SEE WHY WE NEED TO USE Proposition~\ref{prop:strong}, AND I AM WORRIED ABOUT THE %DELIMINATOR OF $\tau$. IT SEEMS TO ME THAT 
By definition,
$N^*$ satisfies Strong Order Multilinearity with respect to $M^*$ if and only if
\begin{align}
&M \preceq_{\rm s}    (N^*|_{\tau\cdot V})^* =N/ (\sym_k(V)\setminus (\tau\cdot V)) \mbox{ for all $\tau \in \sym_{k-1}(V\setminus CL_{M})$}\nonumber\\
&\mbox{via the bijection  $v\mapsto \tau\cdot v$}\label{eq:ext_strong1},
\end{align}
and $N$ satisfies WConingP with respect to $M$.
%\begin{equation}\label{eq:ext_strong2}
%\mbox{$v\cdot \tau$ is a coloop of $N$
%for all $v\in CL_{M}$ and $\tau \in \sym_{k-1}(V)$.} 
%\end{equation}
%\st{THIS (4) IS NOW EXACTLY WConingP.} \bill{SHOULD WE REPLACE IT BY WConingP?}
We will show that (\ref{eq:ext_strong1}) and 
%(\ref{eq:ext_strong2}) 
WConingP are together equivalent to the statement that ExtP holds for $N$ with respect to $M$.
%The key fact to see this is that 
%$\lk(\tau,F)=\lk(\tau,F\cup (\tau\cdot\sym_1(V)^c)$ for any $F\subseteq \sym_k(V)$.

Suppose  ExtP holds for $N$ with respect to $M$. Then ConingP holds for $N$ by Lemma~\ref{lem:extension_coning}. This immediately gives WConingP for $N$ (by taking $v\in CL_M$ and $X=V-v$). 
%(\ref{eq:ext_strong2}). 
We will use the equivalence between the first and second items of Proposition~\ref{prop:strong} to show that (\ref{eq:ext_strong1}) also holds. To this end, choose $\tau\in \sym_{k-1}(V\setminus CL_M)$  and put $(\tau\cdot V)^c=\sym_k(V)\setminus \tau\cdot V$. Let  $F$ be a flat of  $M$. We need to show that $\tau\cdot F$ is a flat of  $N/ (\tau\cdot V)^c$. This holds trivially if $F=V$, so we may assume that $F\neq V$.
Choose  $v\in V\setminus F$.
 We have $$\lk(\tau, (\tau\cdot F)\cup (\tau\cdot V)^c)=\lk(\tau,\tau\cdot F)=F.$$
 Since $F$ is a flat in $M$ and $v\not\in F$, this gives $v\notin {\rm cl}_M(\lk(\tau, (\tau\cdot F)\cup (\tau\cdot V)^c))$.
 Together with the assumption that ExtP holds for $N$, this in turn implies that 
 $(\tau\cdot F)\cup (\tau\cdot V)^c$ does not span $\tau \cdot v$ in $N$.
 Since $v$ is an arbitrary element of $V\setminus F$, we have $(\tau\cdot F)\cup (\tau\cdot V)^c$ is a flat in $N$, and hence $\tau\cdot F$ is a flat of  $N/ (\tau\cdot V)^c$, as required.
 %and (\ref{eq:ext_strong1}) follows by Proposition~\ref{prop:strong}.

 Conversely, suppose (\ref{eq:ext_strong1}) and 
 %(\ref{eq:ext_strong2}) 
 WConingP hold for $N$. 
 Choose an $N$-independent set $I\subseteq \sym_k(V)$,  $\tau\in \sym_{k-1}(V)$ and $v\in V\setminus {\rm cl}_M(\lk(\tau, I))$. If $\tau\in \sym_{k-1}(V\setminus CL_M)$ then, by (\ref{eq:ext_strong1}),
 $\tau\cdot v\notin {\rm cl}_{N/(\tau\cdot V)^c}(\tau\cdot \lk(\tau, I ))$, 
 %\st{To understand this notation, it might be better to insert "through the bijetion $v\leftrightarrow v\cdot \tau$" immediately after (3).}
 and hence 
 $\tau\cdot v \notin {\rm cl}_N(I\cup (\tau\cdot V)^c)$. On the other hand, if $\tau\not\in \sym_{k-1}(V\setminus CL_M)$ then $v\cdot \tau=u\cdot \beta$ for some $u\in CL_M$ and $\beta\in \sym_{k-1}(V)$, and hence $v\cdot \tau$ is a coloop of $N$ by 
 %(\ref{eq:ext_strong2})
 WConingP. In both cases, $I+v\cdot \tau$ is $N$-independent.
 Thus, ExtP holds for $N$.
This completes the proof of (\ref{eq:ExtP}).

\medskip

%\tcolblue{
We next prove that 
$N$ satisfies 0-ExtP and WConingP with respect to $M$ if and only if $N^*$ satisfies Weak Order Multilinearity with respect to $M^*$.
By Proposition~\ref{prop:weak},
$N^*$ satisfies Weak Order Multilinearity  with respect to $M^*$ if and only if 
\begin{equation}
 N/ (\tau\cdot V)^c\succeq_{\rm w} M \quad  \text{for all }\tau \in \sym_{k-1}(V\setminus CL_M) \label{eq:ext_weak0}
\end{equation}
and WConingP holds for $N$ with respect to $M$.
Note that WConingP implies $N/ (\tau\cdot V)^c$ is a free matroid for each $
\tau\in \sym_{k-1}(V)\setminus \sym_{k-1}(V\setminus CL_M)$.
Hence, the combination of  WConingP and (\ref{eq:ext_weak0})  is equivalent to
the combination of WConingP and 
\begin{equation}
 N/ (\tau\cdot V)^c\succeq_{\rm w} M \quad  \text{for all }\tau \in \sym_{k-1}(V). \label{eq:ext_weak1}
\end{equation}
Hence, it suffices to prove that  (\ref{eq:ext_weak1}) is equivalent to 0-ExtP in $N$.

Suppose $N$ has 0-ExtP.
Then, for any $M$-independent set $X\subseteq V$ and $\tau\in \sym_{k-1}(V)$,
$\tau\cdot X$ can be added to any $N$-independent set in $(\tau\cdot V)^c$ keeping $N$-independence by 0-ExtP. 
Equivalently, $\tau\cdot X$ is independent in $N/(\tau\cdot V)^c$ and (\ref{eq:ext_weak1}) holds.

 Conversely, suppose (\ref{eq:ext_weak1}) holds. 
 For any $\tau\in \sym_{k-1}(V)$ and any $M$-independent set $X\subseteq V$,
 $\tau\cdot X$ is independent in $N/(\tau\cdot V)^c$.
 This implies 0-ExtP.
 This completes the proof.
\end{proof}

Applying Lemma \ref{lem:duality} to (\ref{eq:biliearity_implication}) immediately gives:

%\begin{lemma}\label{lem:ExtPto0ExtP}
%Suppose $N$ satisfies ExtP. Then $N$ satisfies 0-ExtP.
%\end{lemma}

\begin{equation}\label{eq:dual_biliearity_implication}
\text{Dual Symmetric Q-power} \Rightarrow 
\text{DM} \Rightarrow
\text{CoCP} \Rightarrow
\text{ExtP}\Rightarrow
\text{0-ExtP}
\end{equation}

We close this subsection by using Lemma~\ref{lem:duality} to derive the dual forms of the theorems stated in Section~\ref{subsec:results}.
%have the dual counterparts.
%Since our proof is written in terms of the dual formulations,
%below we shall give them explicitly.
\begin{theorem}[Dual of Theorem~\ref{thm:2}]\label{thm:dual_2}
Let $M$ be a matroid on a finite set $V$ and $N$ be a matroid on $\sym_2(V)$.
Then the following statements are equivalent.
\begin{itemize}
\item $N^*$ is a second symmetric power of $M^*$.
\item $N$ satisfies  RRP and DM.
\item $N$ satisfies RRP and CoCP.
\item $N$ satisfies RRP and ExtP.
\item $N$ satisfies CycP and ExtP.
\end{itemize}
\end{theorem}
\begin{theorem}[Dual of  Theorem~\ref{thm:2_weak}(a)]\label{thm:dual_2_weak}
Let $M$ be a matroid on a finite set $V$ and $N$ be a matroid on $\sym_2(V)$.
Then the following statements are equivalent.
\begin{itemize}
\item $N$ satisfies RRP and CycP.
\item $N$ satisfies CycP and 0-ExtP.
\item $N$ satisfies StrongRRP and 0-ExtP.
\item $N$ satisfies RRP, ConingP, and 0-ExtP.
\end{itemize}
%Moreover, there is a matroid on $\sym_2(V)$ that satisfies RRP and CycP, but does not satisfy DB.
\end{theorem}
\begin{theorem}[Dual of Theorem~\ref{thm:k}(a)]\label{thm:dual_k}
Let $M$ be a matroid on a finite set $V$ and $N$ be a matroid on $\sym_k(V)$
with $k\geq 2$.
If $N$ satisfies RRP and ExtP with respect to $M$, then $N$ satisfies DM, CycP, StrongRRP and ConingP.
%Then the following statements are equivalent.
%\begin{itemize}
%\item $N$ satisfies RRP and DM.
%\item $N$ satisfies RRP and CoCP.
%\item $N$ satisfies RRP and ExtP.
%\end{itemize}
%In addition, if  $N$ satisfies RRP and DM, 
%then $N$ satisfies CycP, StrongRRP, ConingP, and 0-Ext.
% On the other hand, there is a matroid on $\sym_k(V)$ that satisfies CycP and Dual Symmetric Q-power, but does not satisfy RRP.
\end{theorem}

%We close this subsection by stating a dual form of   Theorem~\ref{thm:k_weak}(a). 
Our next statement also introduces a new property, the Canonical Base Property, which will be defined in Section~\ref{sec:contree}. Informally, it tells us that we can construct a base of $N$ by recursively applying 0-ExtP. This property will play a key role in our proofs. 

\begin{theorem}
%[Dual of Theorem~\ref{thm:k_weak}(a)]
\label{thm:dual_k_weak}
Let $M$ be a matroid on a finite set $V$ and $N$ be a matroid on $\sym_k(V)$
with $k\geq 2$.
Then the following statements are equivalent.
\begin{itemize}
\item $N$ satisfies StrongRRP and 0-ExtP.
\item $N$ satisfies RRP, ConingP, and 0-ExtP.
\item $N$ satisfies the Canonical Base Property.
%(defined in Section~\ref{sec:contree}). \st{I added this to promote CBP, but may be deleted.}
\end{itemize}
In addition, if  $N$ satisfies StrongRRP and 0-ExtP, 
then $N$ satisfies CycP.
%On the other hand, there is a matroid on $\sym_k(V)$ that satisfies 
%RRP, CycP, and 0-ExtP,  but does not satisfy StrongRRP or ConingP.
\end{theorem}

\begin{theorem}[Dual of Theorem~\ref{thm:mason_k_3}(a)] \label{thm:dual_mason_k_3}
Let $M$ be a uniform matroid on a finite set $V$ of rank at least one. 
Suppose $N$ is a matroid on $\sym_3(V)$ which satisfies CycP
and DM. Then $N$ satisfies RRP.
\end{theorem}

%\st{The Canonical Base Property is the property that a space admits bases constructed inductively in a manner inspired by rigidity theory.}

\subsection{Examples}\label{subsec:example}
We first give a concrete example of the dual of a symmetric power of a matroid.

Let $V=\{1,2,\ldots,n\}$ and consider the matroid ${\cal B}$ on $\sym_2(V)$ in which
$F\subset \sym_2(V)$ is independent  if and only if, regarding $F$ as the edge set of a graph with vertex set $V$, each connected component in this graph contains at most one cycle.
(Thus $\B$ is the {\em bicircular matroid} on the looped complete graph with vertex set $V$.)
%\bill{
Then  $\B$ satisfies DM and RRP with respect to $U_n^1$, and hence ${\cal B}^*$ is a second symmetric power of $U_n^{n-1}$ by Theorem \ref{thm:dual_2}.
%x}

We can use the  
observation that bicircular matroids are  transversal matroids to define  an analogue of bicircular matroids on $\sym_k(V)$ for each $k\geq 3$.
%as follows.
%Let $V=\{1,\dots, n\}$.
The {\em $k$-th bipartite incidence graph} $B_{n,k}$ for $V$
is the bipartite graph with vertex bipartition $(\sym_{k-1}(V), \sym_k(V))$
in which two vertices $\alpha\in \sym_{k-1}(V)$ and $\beta\in \sym_k(V)$ are adjacent if %$\alpha|\beta$ 
$\alpha$ is a subword of $\beta$.
For $k\geq 1$, the {\em $k$-th incidence matroid} $M(B_{n,k})$ is the transversal matroid on $\sym_k(V)$ represented by $B_{n,k}$. In particular,
since $\sym_0(V)=\{\emptyset\}$ and $\emptyset$ is a subword of $\alpha$ for every $\alpha\in V$,    $B_{n,1}$ is isomorphic to $K_{n,1}$, and  $M(B_{n,1})$ is a copy of $U_n^1$ with groundset $V$. 

% Observe first that, when $M=U_n^1$, Dual Rank Condition with respect to  for a quasi-rigidity sequence $(N_1,N_2,\ldots,N_k)$ for $M$ to be a rigidity sequence is 

% \begin{equation}\label{eq:rank_1}
% r_{N_k}={n+k-1 \choose k}-{n-1+k-1\choose k}={n+(k-1)-1\choose k-1}=|\sym_{k-1}(V)|.
% \end{equation}
%When $k=2$, this value is $n$, and the bicircular matroid on the edge set of the looped complete graph on $n$ vertices is an example of a $2$-nd looped rigidity matroid of $U_n^1$.
%A bicircular matroid is an example of transversal matroids, and one can consider a hypergraph extension of bicircular matroids within the class of transversal matroids as follows.

\begin{theorem}\label{thm:transversal}
%\bill{
The $k$-th incidence matroid $M(B_{n,k})$
 satisfies the Dual Symmetric Q-power and Rigidity Rank Properties with respect to $U_n^1$ for all $k\geq 1$, and hence $M(B_{n,k})$ is the dual of a $k$-th symmetric  power of $U_n^{n-1}$. 
 %}
 \end{theorem}
\begin{proof}
We proceed by induction on $k+n$. The case when $k=1$ holds since $M(B_{n,1})$ is a copy of $U_n^1$.
%with groundset $\sym_1(V)$. 
Hence we may assume that $k\geq 2$.
The statement is also trivially true when $n=1$ since 
$M(B_{1,i})$ is a copy of $U_1^1$ on $\sym_i(V)$ for all $1\leq i\leq k$.
Hence, we may also assume that $n\geq 2$.

It is not difficult to see that 
$B_{n,k}$ has a matching that covers $\sym_{k-1}(V)$.
(For example, consider the matching which matches $\alpha\in \sym_{k-1}(V)$ with $v\cdot \alpha\in \sym_{k}(V)$ for some fixed $v\in V$.)
This shows that   
\begin{equation}\label{eq:rank_1}
r_{M(B_{n,k})}=|\sym_{k-1}(V)|={n+(k-1)-1\choose k-1}={n+k-1 \choose k}-{n-1+k-1\choose k}, 
\end{equation}
implying that $M(B_{n,k})$ satisfies RRP with respect to $U_n^1$.

It remains to check that the Dual Symmetric Q-power property holds for $M(B_{n,k})$ with respect to $U_n^1$.
By induction on $n+k$, it will suffice to prove that, for each $v\in V$, $M(B_{n,k})/\sym_{k}(V\setminus \{v\})$ is isomorphic to $M(B_{n,k-1})$ via the correspondence $v\cdot \beta \mapsto \beta$.

To see this, recall that 
 the vertex set  of $B_{n,k}$ is $\sym_{k-1}(V)\sqcup\sym_k(V)$.
 Consider the bipartitions:
\[
\begin{split}
\sym_{k-1}(V)&=(v\cdot \sym_{k-2}(V))\sqcup \sym_{k-1}(V\setminus \{v\});  \\
\sym_{k}(V)&=(v\cdot \sym_{k-1}(V))\sqcup \sym_{k}(V\setminus \{v\}).
\end{split}
\]

We have already seen in (\ref{eq:rank_1}) that the rank of $M(B_{n-1,k})$ is equal to $|\sym_{k-1}(V\setminus\{v\})|$. Hence 
there is a matching on $\sym_{k-1}(V\setminus \{v\}) \times \sym_{k}(V\setminus \{v\})$ that covers $\sym_{k-1}(V\setminus \{v\})$.
In addition, there are no edges in $B_{n,k}$  between $v\cdot \sym_{k-2}(V)$ and 
$\sym_{k}(V\setminus \{v\})$.
Therefore, 
$M(B_{n,k})/\sym_{k}(V\setminus \{v\})$ is the  transversal matroid on $v\cdot \sym_{k-1}(V)$ represented by the bipartite subgraph of $B_{n,k}$ induced by $(v\cdot \sym_{k-2}(V))\sqcup (v\cdot \sym_{k-1}(V))$.
The graph isomorphism between this induced subgraph
and $B_{n,k-1}$ now gives the required isomorphism between 
$M(B_{n,k})/\sym_{k}(V\setminus \{v\})$ and $M(B_{n,k-1})$.
This completes the proof of the theorem.
\end{proof}

\section{Connection to Rigidity}\label{sec:rigidity}

Brakensiek et al.~\cite{brakensiek} showed that generic rigidity matroids are connected to second symmetric tensor powers of generic realisations of uniform matroids by matroid duality. This connection was extended to abstract rigidity matroids and second symmetric powers of uniform matroids in \cite{JTfields}. We will describe a further generalisation which relates second symmetric powers of an arbitray matroid to abstract rigidity matroids with respect to the dual matroid.  As an application we use an abstract rigidity matroid of a non-generic realisation of $K_6$ to construct an example  which verifies Theorem \ref{thm:2_weak}(b).

%connection between $k$'th symmetric powers of matroids and rigidity theory is most apparent  when $k=2$.
%Let ${V\choose 2}$ be the set of all unordered pairs of distinct elements in $V$.
One of the major topics in graph rigidity theory is the analysis of generic rigidity matroids, which are linear matroids on the edge set of a complete graph.
%${V\choose 2}$.
Motivated by a desire to obtain a better understanding of  generic rigidity matroids, Graver~\cite{G91} used the gluing properties of these matroids to define a larger class of matroids, called {\em abstract rigidity matroids}.
%, by extracting gluing properties of graph rigidity as properties of matroids. 
Several equivalent definitions of abstract rigidity matroids are known, see ~\cite{GSS93,N10}. We will see that the following definition is linked to that of a second symmetric power of a uniform matroid by  matroid duality.

For an integer $d\geq 1$ and a finite set $V$ of $n$ elements with  $n\geq d+1$,
a matroid $N$ on ${V\choose 2}$ is said to be an {\em abstract $d$-rigidity matroid} if
\begin{itemize}
    \item the rank of $N$ is $dn-{d+1\choose 2}$, and
    \item the edge set of each copy of $K_{1,n-d}$ is a cocircuit in $N$,
    where $K_{1,n-d}$ denotes  the complete bipartite graph with partite sets of sizes $1$ and $n-d$.
\end{itemize}

The class of abstract $d$-rigidity matroids includes the  rigidity matroid of 
%every set of at least $d+1$  affinely independent points in $\R^d$
a generic point configuration in $\R^d$, as well as cofactor matroids from the theory of bivariate splines~\cite{W}, and Kalai's hyperconnectivity matroids~\cite{K} which feature in the theory of low-rank symmetric matrix completion~\cite{SC}.
%, see, e.g.,~\cite{cruickshank2025rigidity} for more details.

By considering the elements of $\binom{V}{2}$ as unordered words of length two, the second bullet point in the definition of abstract rigidity can be restated as:
%has a matroidal interpretation  as follows:
\begin{itemize}
\item For every $v\in V$ and every $X\subseteq V\setminus \{v\}$ such that 
$X$ is a cocircuit of the uniform matroid on $V$ of rank $d+1$,
$v\cdot X$ is a cocircuit of $N$.
\end{itemize}
The above condition can be generalized  by replacing the uniform matroid with an arbitrary matroid  of rank $d+1$.
This observation will allow us to extend the notion of abstract $d$-rigidity to abstract $M$-rigidity for an arbitrary matroid $M$.

Let $M$ be a 
%nontrivial\footnote{A matroid is said to be nontrivial if the rank is not zero.} 
matroid on $V$ having neither loops nor coloops,
and let 
\[
d_M=\rank M-1.
\]
A matroid $N$ on ${V\choose 2}$ is said to be an {\em abstract $M$-rigidity matroid} if
\begin{itemize}
    \item[\bf A1] the rank of $N$ is $d_M|V|-{d_M+1\choose 2}$, and
    \item[\bf A2] for every $v\in V$ and every $M$-cocircuit $X\subseteq V\setminus \{v\}$, $v\cdot X$ is an $N$-cocircuit.
\end{itemize}

Graver's abstract $d$-rigidity matroid corresponds to the case when $M$ is the uniform matroid on $V$ with rank $d+1$.
When $M$ is a rank $d+1$ affine matroid  of any point configuration $P=\{p_1,\dots, p_n\}$ with $p_i\in \mathbb{R}^d$, the rigidity matroid on $P$ is an abstract $M$-rigidity matroid. 
Hence we can use the extension of abstract $d$-rigidity  to abstract $M$-rigidity to study the rigidity of point configurations which are not affinely independent in $\R^d$.
%extends Graver's idea to non-generic rigidity. 
%More details on this aspect of abstract $M$-rigidity will be given in a forthcoming paper.

%\bill{
We will delay a detailed study of abstract $M$-rigidity to a forthcoming paper and will concentrate on its links to symmetric powers in this section. Our first result describes how we can use the second symmetric power of a matroid to construct an abstract rigidity matroid of its dual. To accomplish this we will continue to consider each element of ${V\choose 2}$ to be an unordered word of length two, so that ${V\choose 2}\subset \sym_2(V)$.
%}
%One can construct abstract $M$-rigidity matroids from symmetric powers of $M^*$ as follows.
\begin{lemma}\label{lem:power_rigidity}
Let $M$ be a matroid on $V$ having neither loops nor coloops,
and let $N$ be a second symmetric power of $M$.
Then, $N^*/(\sym_2(V)\setminus {V\choose 2})$
is an abstract $M^*$-rigidity matroid.
\end{lemma}
\begin{proof}
Let $|V|=n$ and let $\tilde N=N^*/(\sym_2(V)\setminus {V\choose 2})$. By Theorem~\ref{thm:dual_2}, $N^*$ satisfies ExtP, RRP and CoCP with respect to $M^*$ with $k=2$.

We first show that $\tilde N$ satisfies (A1). Since $N^*$ satisfies ExtP, 
 $\sym_2(V)\setminus {V\choose 2}$ is independent in $N^*$ and hence $\rank \tilde N=\rank N^*-n$.
Since $N^*$ satisfies RRP, this gives  
\begin{align*}
\rank \tilde N&=
{n+1\choose 2}-{n-\rank M^* +1\choose 2}-n\\
&=(\rank M^*-1)n-{\rank M^*\choose 2}
=d_{M^*}n-{d_{M^*}+1\choose 2}.
\end{align*}
%Hence, the rank condition to be abstract $M^*$-rigidity holds.

It remains to show that $\tilde N$ satisfies (A2). Since $N^*$ satisfies CoCP,
for every $v\in V$ and every $M^*$-cocircuit $X\subseteq V\setminus \{v\}$,
$v\cdot X$ is a circuit in $N$.
This implies that $v\cdot X$  is a circuit in $N\setminus (\sym_2(V)\setminus {V\choose 2})$,
and hence is a cocircuit of $\tilde N$.
%$N^*/ (\sym_2(V)\setminus {V\choose 2})$.
\end{proof}

It is an  open problem to decide if a reverse construction exists, i.e.~for any matroid $M$, can we  always  construct a second symmetric power of $M^*$ from an abstract $M$-rigidity matroid? The following construction for the special case 
when $M$ is a uniform matroid is taken from \cite{cruickshank2025rigidity}.
Given a matroid $N$ on ${V\choose 2}$,
let $\hat{N}$ be the matroid on $\sym_2(V)$ obtained from $N$ by appending 
each element in $\sym_2(V)\setminus {V\choose 2}$ as a loop.
Also, let ${\cal B}$ be the bicircular matroid on $\sym_2(V)$.
Let $\hat{N}\vee {\cal B}$ be the matroid union of $\hat{N}$ and ${\cal B}$.
Then, $N$ is an abstract $U_{n}^{d+1}$-rigidity matroid if and only if 
$(\hat{N}\vee {\cal B})^*$ is a second symmetric power of $U_{n}^{n-(d+1)}$.

This simple construction can be used to  construct a matroid belonging to the  family of matroids described in Theorem~\ref{thm:2_weak} from a matroid which satisfies (A1) and a modified version of (A2).
%from an abstract $M^*$-rigidity matroid i.e. constructing a matroid in )Lov{\'a}sz' matroid family (the family in Theorem~\ref{thm:2_weak}) as follows.
\begin{lemma}\label{lem:weak_construction}
Let $M$ be a matroid on $V$ having neither loops nor coloops, and with $|V|=n$ and $\rank M= d+1$.
Consider the following two properties for a matroid  $N$ on ${V\choose 2}$:
\begin{itemize}
\item[(i)] $\rank N=dn-{d+1\choose 2}$;
\item[(ii)] for every $M$-cyclic set $X$, ${X\choose 2}$ is $N$-cyclic.
\end{itemize}
Then $N$ satisfies (i) and (ii)
if and only if 
$(\hat{N}\vee {\cal B})^*$ 
satisfies SP-rank Property and Flat Property with respect to $M^*$.
%(c.f.~Theorem~\ref{thm:2_weak}).
\end{lemma}
\begin{proof}
Let $N'=\hat{N}\vee {\cal B}$.
By Lemma \ref{lem:duality}, it will suffice to  show that $N$ satisfies (i) and (ii) if and only if 
$N'$ satisfies RRP and CycP with respect to $M$.

The equivalence between the rank condition (i) for $N$ and RRP for $N'$ follows easily from the fact that $\rank N'=\rank N+\rank {\cal B}=\rank N+n$.

Moreover, the fact that $\sym_2(V)\setminus {V\choose 2}$ is a base of ${\cal B}$ gives
\begin{equation}\label{eq:weak_construction}
N'/ \left(\sym_2(V)\setminus {V\choose 2}\right)=N.
\end{equation}
Hence, 
if $\sym_2(X)$ is cyclic in $N'$ for some $X\subseteq V$,
then ${X\choose 2}$ is cyclic in $N$.
Thus, CycP for $N'$ implies (ii) for $N$.

It remains to show that property (ii) for $N$ implies CycP for $N'$. To this end, we assume that (ii) holds and choose a {nonempty} $X\subseteq V$ such that $X$ is $M$-cyclic. Then 
%\bill{
$|X|\geq 2$ since $M$ is loopless
%} 
and 
${X\choose 2}$ is $N$-cyclic by (ii). Equation  (\ref{eq:weak_construction}) now implies that
no element of ${X\choose 2}$ is a coloop in 
$N'|_{\sym_2(X)}$.

To conclude that $\sym_2(X)$ is cyclic in $N'$, it remains to check that 
no element $e=xx\in \sym_2(X)\setminus {X\choose 2}$
is a coloop in $N'|_{\sym_2(X)}$. Choose  $f=xy\in {X\choose 2}$.
%$e\cap f\neq \emptyset$.
Then, $(\sym_2(X)\setminus {X\choose 2})-e+f$ forms a base of 
${\cal B}|_{\sym_2(X)}$.
In addition, since ${X\choose 2}$ is cyclic in $N$, ${X\choose 2}-f$ is a spanning set of $N|_{X\choose 2}$. Thus, 
\[
\sym_2(X)-e=\mbox{$\left({X\choose 2}-f\right)\cup \left((\sym_2(X)\setminus {X\choose 2})-e+f\right)$}
\]
is a spanning set in $N'|_{\sym_2(X)}=\hat{N}|_{\sym_2(X)}\vee {\cal B}|_{\sym_2(X)}$.
Hence, $e$ is not a coloop in $N'|_{\sym_2(X)}$, and $\sym_2(X)$ is cyclic in $N'$.
\end{proof}

We can use Lemma~\ref{lem:weak_construction} to construct an example which verifies Theorem \ref{thm:2_weak}(b) (and hence distinguishes the matroids considered by Lov{\'a}sz from Anderson's second symmetric powers).

\begin{theorem}\label{thm:example}
There exists a set $V$ and a pair of matroids $M'$ on $V$ and $N'$ on $\sym_2(V)$ such that $N'$ satisfies
SP-rank Property and Flat Property with respect to $M'$, but does not satisfy Multilinearity with respect to $M'$.
\end{theorem}
\begin{proof}
We proceed by first constructing the dual matroid $M=(M')^*$, then constructing $N$ to satisfy properties 
%The proof is done by constructing a matroid $N$ on ${V\choose 2}$ satisfying 
(i) and (ii) of Lemma~\ref{lem:weak_construction}.
%and then showing that $(\hat{N}\vee {\cal B})^*$ is a matroid required in the statement.

Let $P$ be a set of six points $p_1, p_2, p_3, p_4, p_5, p_6$ in $\mathbb{R}^3$
such that  $p_1, p_2, p_3$ are collinear and the remaining points are generically located in $\mathbb{R}^3$.
We define  $M=A_{P}$ to be  the affine matroid  on $V=\{1,2,3,4,5,6\}$ defined by $P$.

%\bill{
We next define $N$ by describing its circuits. 
Let $X=\{1,2,3\}$, 
\begin{align*}
{\cal C}_1&=\left\{{X\choose 2}\right\},\\
{\cal C}_2&=\left\{ {Z\choose 2}: Z\subset V, |Z|=5, X\not\subseteq Z\right\}, \\ 
{\cal C}_3&=\left\{ (C_1\cup C_2)\sm (C_1\cap C_2): C_1\in {\cal C}_1,\, C_2\in {\cal C}_2 \right\}, \\
{\cal C}_4&=\left\{ C\subseteq {V\choose 2}: \mbox{$|C|=13$, and $C$ does not contain an element of  ${\cal C}_1\cup {\cal C}_2\cup {\cal C}_3$} \right\},
\end{align*}
and put ${\cal C}={\cal C}_1\cup {\cal C}_2\cup {\cal C}_3\cup {\cal C}_4$. 

% Let $X=\{1,2,3\}$, ${\cal C}_1=\{{X\choose 2}\}$,
% ${\cal C}_2=\{ {Z\choose 2}: Z\subset V, |Z|=5, X\not\subseteq Z\}$, 
% $${\cal C}_3=\left\{ (C_1\cup C_2)\sm (C_1\cap C_2): C_1\in {\cal C}_1,\, C_2\in {\cal C}_2, 
% %\,|C_1\cap C_2|=1
% \right\},$$ 
% $${\cal C}_4=\mbox{$\{ C\subseteq {V\choose 2}: |C|=13$, and $C$ does not contain an element of  ${\cal C}_1\cup {\cal C}_2\cup {\cal C}_3 \}$},$$
% and put ${\cal C}={\cal C}_1\cup {\cal C}_2\cup {\cal C}_3\cup {\cal C}_4$. 
%}
%Also, define a family ${\cal C}$ of sets in ${V\choose 2}$ by
%\begin{align*}
%{\cal C}&=\left\{{X\choose 2}\right\} \cup \left\{ {Z\choose 2}: Z\subset V, |Z|=5, X\not\subseteq Z\right\}\\
%&\cup \left\{{X\choose 2}\cup {Z\choose 2}-uv: Z\subset V,|Z|=5, X\cap Z=\{u,v\}\right\} \\
%&\cup \left\{ F\subseteq {V\choose 2}: |F|=13, \text{$F$  is not contained in any of the previous sets}\right\}.
%\end{align*}
\begin{claim}
${\cal C}$ satisfies the matroid circuit axiom.
\end{claim}
\begin{proof}
This can be checked directly by case analysis.
Alternatively we can check that ${\cal C}$ is the family of circuits of the $C_2^1$-cofactor matroid 
of a two-dimensional point configuration $P'$ 
obtained by projecting $P$  onto a generic plane. 
%See~\cite{W,CJT0} for the definition of the $C_2^1$-cofactor matroid.
\end{proof}

Let $N$ be the matroid on $\binom{V}{2}$ defined by ${\cal C}$. Then $\rank M=4$, $\rank N=12$, and hence $N$ satisfies property (i) with $d=3$ and $|V|=6$.
In addition,  we can use the definition of  ${\cal C}$ to verify that 
\begin{equation}\label{eq:bad_base}
B:=\{12,13,14,15,16,24,25,26,34,35,36,45\} \text{ is a base of $N$}.
\end{equation}
%Hence, $N$ has rank $12=d|V|-{d+1\choose 2}$ for $d=3$ and $|V|=6$,
%and $N$ satisfies (i) with respect to $M$.

In the affine matroid $M$, a {nonempty} cyclic set $Z$ can be one of the three different types:  $Z=X$;
 $|Z|=5$ and $X\not\subseteq Z$; 
$Z=V$.
We can use the definition of ${\cal C}$ to verify that 
${Z\choose 2}$ is a cyclic set in $N$ in each case.
Therefore, $N$ satisfies (ii) with respect to $M$.
Lemma~\ref{lem:weak_construction} now implies that
$(\hat{N}\vee {\cal B})^*$ satisfies SP-rank Property and Flat Property.
It remains to show that $(\hat{N}\vee {\cal B})^*$ does not satisfy  Multilinearity.

Suppose, for a contradiction,  that $(\hat{N}\vee {\cal B})^*$ satisfies Multilinearity.
Then $\hat{N}\vee {\cal B}$ satisfies ExtP  {by (\ref{eq:dual_biliearity_implication})}. %by Theorem~\ref{thm:dual_2}.
Let $B'=B\cup \{11,22,33,44,55,66\}$.
Since $B$ is a base of $N$ by (\ref{eq:bad_base}),
$B'$ is a base of $\hat{N}\vee {\cal B}$.
However, since $\lk(6,B')=\{1,2,3,6\}$, 
we have $4\notin {\rm cl}_M(\lk(6,B'))$ by the definition of $M$,
and hence $B'+46$ is independent in $\hat{N}\vee {\cal B}$ by ExtP.
This contradiction completes the proof.
\end{proof}

\section{Counterexamples to Mason's Conjecture}\label{sec:counterexample}
We will describe our counterexamples to Mason's conjecture.
By duality and Lemma \ref{lem:duality}, it will suffice to construct matroids $M$ on $V$ and $N$ on $\sym_k(V)$ such that $N$ 
satisfies Dual Symmetric Q-power Property and CycP with respect to $M$, but does not satisfy RRP.
We will do this for  the case when 
$k\geq 3$ and $M$ is a rank zero matroid in Subsection~\ref{subsec:counterexample3},
and for the case when 
$k=4$ and $M$ is a  uniform matroid of positive rank
in Subsection~\ref{subsec:counterexample4}.

\subsection{\boldmath Counterexample for $k\geq 3$}\label{subsec:counterexample3}
We will verify Theorem \ref{thm:k}(b) (and hence show that Mason' conjecture is false 
when $k\geq 3$).
%and $M$ is a free  matroid. 
This will follow immediately from our next result by 
matroid duality and Lemma~\ref{lem:duality}.

%\newpage
\begin{theorem}\label{thm:rank_zero}
Let $k\geq 3$ be an integer and $M$ be the rank-zero matroid with groundset $V$ with $|V|\geq 3$. 
%Let $N_i$ be the rank-zero matroid on $\sym_i(V)$ for $1\leq i\leq k$ and 
Let $N$ be the matroid on $\sym_k(V)$ obtained from 
the rank-one uniform matroid on $\{e\in \sym_k(V): |{\rm supp}(e)|\geq 2\}$ by
adding all elements $e\in \sym_k(V)$ with $|{\rm supp}(e)|=1$ as loops.
Then, $N$ satisfies the
Dual Symmetric Q-power property and CycP with respect to $M$, but does not satisfy RRP.
\end{theorem}
\begin{proof}
We first check that $N$ satisfies the
Dual Symmetric Q-power property with respect to $M$.
Note that $CL_M=\emptyset$ since $M$ is the rank-zero matroid. Choose
$\alpha \in \sym_i(V)$ for some $1\leq i\leq k-1$.
%Let $Z_i$ be the rank zero matroid on $\sym_i(V)$.
Since $|V|\geq 3$,
$\sym_k(V)\setminus (\alpha\cdot \sym_{k-i}(V))$ contains an element $e$ with $|{\rm supp}(e)|\geq 2$.
This implies that $\sym_k(V)\setminus (\alpha\cdot \sym_{k-i}(V))$ is a spanning set in $N$
and hence $N/(\sym_k(V)\setminus (\alpha\cdot \sym_{k-i}(V)))$ is a rank-zero matroid.
This, in turn, implies the Dual Symmetric Q-power property holds for $N$.

\medskip

We next check that $N$ satisfies CycP  with respect to $M$.
Since $M$ has rank zero, every 
%nonempty 
%\st{nonempty?}
$Y\subseteq V$ is $M$-cyclic. 
Choose $Y\subseteq V$.  If $Y=\emptyset$ then $\sym_k(Y)=\emptyset$ which is cyclic in $N$. And,
if $|Y|=1$, then $\sym_k(Y)$ consists of a single word $e$ with $|{\rm supp}(e)|=1$. Then $e$ is a loop in $N$ by the definition of $N$, and hence $\sym_k(Y)$ is cyclic in $N$.
It remains to consider the case when $|Y|\geq 2$. Let $S=\{e\in \sym_k(Y):|{\rm supp}(e)|\geq 2\}$. Then $|S|\geq 2$ since $|Y|\geq 2$ and $k\geq 3$. 
%(Here, we need  $k\geq 3$.)
Since $N|_S$ is a rank-one uniform matroid,
$S$ is cyclic in $N$. And since every element of $\sym_k(Y)\setminus S$ is a loop in $N$, $\sym_k(Y)$ is cyclic in $N$.

\medskip

To complete the proof, observe that $N$ does not satisfy RRP with respect to $M$
since $M$ has rank zero and $N$ has rank one.
%This completes the proof.
\end{proof}

\subsection{\boldmath Counterexample for uniform matroids when $k=4$}
\label{subsec:counterexample4}
We will construct our counterexample using the $k$-th incidence matroid $M(B_{n,k})$
defined in Section~\ref{subsec:example}.
Recall that $M(B_{n,k})$ is the transversal matroid represented by 
the incidence bipartite graph $B_{n,k}$ between $\sym_{k-1}(V)$
and $\sym_k(V)$, when $V=\{1,2,\ldots,n\}$. 

We first give a linear representation of $M(B_{n,k})$.
We associate  an indeterminate $x_{\alpha,\beta}$ with 
 each pair $(\alpha,\beta)$ for $\alpha\in \sym_k(V)$ and $\beta\in \sym_{k-1}(V)$.
We will be concerned with matrices defined over the field 
of rational functions in $x_{\alpha,\beta}$.
%(or equivalently over  $\mathbb{R}$ if we regard the set of $x_{\alpha,\beta}$ as a set of real numbers which is algebraically independent over $\mathbb{Q}$).

We define the 
%$k$-th {\em generic incidence matrix} $I_{n,k}$ on $n$ vertices 
%\bill{\em 
generic adjacency matrix of $B_{n,k}$
%} 
to be the  $|\sym_k(V)|\times |\sym_{k-1}(V)|$ matrix $I_{n,k}$ in which each row is indexed by an element of $\sym_k(V)$,
each column is indexed by an element of $\sym_{k-1}(V)$,
and the entry in row $\alpha\in \sym_k(V)$ and column $\beta\in \sym_{k-1}(V)$ is $x_{\alpha,\beta}$ if $\beta$ is a subword of $\alpha$,
and is zero, otherwise.
For example, $I_{3,3}$ is
\[
\begin{blockarray}{ccccccc}
  &% \{1,1\} & \{1,2\} & \{1,3\} & \{2,2\} & \{2,3\} & \{3,3\} \\
   11 & 12 & 13 & 22 & 23 & 33 \\
\begin{block}{c(cccccc)}
111 & x_{111,11} & 0 & 0 & 0 & 0 & 0 \\
112 & x_{112,11} & x_{112,12} & 0 & 0 & 0 & 0 \\
113 & x_{113,11} & 0 & x_{113,13} & 0 & 0 & 0 \\
122 & 0 & x_{122,12} & 0 & x_{122,22} & 0 & 0 \\
123 & 0 & x_{123,12} & x_{123,13} & 0 & x_{123,23} & 0 \\
133 & 0 & 0 & x_{133,13} & 0 & 0 & x_{133,33} \\
222 & 0 & 0 & 0 & x_{222,22} & 0 & 0 \\
223 & 0 & 0 & 0 & x_{223,22} & x_{223,23} & 0 \\
233 & 0 & 0 & 0 & 0 & x_{233,23} & x_{233,33} \\
333 & 0 & 0 & 0 & 0 & 0 & x_{333,33} \\
\end{block}
\end{blockarray}
\]

A classical observation of Edmonds~\cite{edmonds1967systems} implies the following.

\begin{proposition}\label{prop:edmonds}
$M(B_{n,k})$ is the row matroid of $I_{n,k}$.
\end{proposition}

%This proposition implies, in particular, that $M(B_{n,2})$ is the row matroid of a weighted edge/vertex incidence matrix of $K_n^{\circ}$. It is well known that this gives a representation of the bicycle matroid of $K_n^{\circ}$.

We will focus on the case when $k=4$ for the remainder of this section.
We construct a matrix $I_{n,4}'$ from $I_{n,4}$ by appending a new column indexed by a special symbol $\spadesuit$.
The entries of this new column are defined as follows.
\[
I_{n,4}'[\alpha,\spadesuit]=\begin{cases}
0 & \text{if }|{\rm supp}(\alpha)|\leq 2\\
y_{\alpha} & \text{if }|{\rm supp}(\alpha)|\geq 3\\
\end{cases} \qquad (\alpha\in \sym_4(V)),
\]
where $y_{\alpha}$ is a new indeterminate associated to each $\alpha\in \sym_4(V)$
with $|{\rm supp}(\alpha)|\geq 3$.
For example, $I_{3,4}'$ is 
\[\tiny
\begin{blockarray}{c@{\hspace{7pt}}ccccccccccc}
 & 111 & 112 & 113 & 122 & 123 & 133 & 222 & 223 & 233 & 333 & \spadesuit\\
\begin{block}{c@{\hspace{7pt}}(ccccccccccc)}
1111 & x_{1111,111} & 0 & 0 & 0 & 0 & 0 & 0 & 0 & 0 & 0 & 0 \\
1112 & x_{1112,111} & x_{1112,112} & 0 & 0 & 0 & 0 & 0 & 0 & 0 & 0 & 0 \\
1113 & x_{1113,111} & 0 & x_{1113,113} & 0 & 0 & 0 & 0 & 0 & 0 & 0 & 0 \\
1122 & 0 & x_{1122,112} & 0 & x_{1122,122} & 0 & 0 & 0 & 0 & 0 & 0 & 0\\
1123 & 0 & x_{1123,112} & x_{1123,113} & 0 & x_{1123,123} & 0 & 0 & 0 & 0 & 0 & y_{1123}\\
1133 & 0 & 0 & x_{1133,113} & 0 & 0 & x_{1133,133} & 0 & 0 & 0 & 0 & 0 \\
1222 & 0 & 0 & 0 & x_{1222,122} & 0 & 0 & x_{1222,222} & 0 & 0 & 0  & 0\\
1223 & 0 & 0 & 0 & x_{1223,122} & x_{1223,123} & 0 & 0 & x_{1223,223} & 0 & 0 & y_{1223}\\
1233 & 0 & 0 & 0 & 0 & x_{1233,123} & x_{1233,133} & 0 & 0 & x_{1233,233} & 0 & y_{1233}\\
1333 & 0 & 0 & 0 & 0 & 0 & x_{1333,133} & 0 & 0 & 0 & x_{1333,333} & 0\\
2222 & 0 & 0 & 0 & 0 & 0 & 0 & x_{2222,222} & 0 & 0 & 0 & 0\\
2223 & 0 & 0 & 0 & 0 & 0 & 0 & x_{2223,222} & x_{2223,223} & 0 & 0 & 0\\
2233 & 0 & 0 & 0 & 0 & 0 & 0 & 0 & x_{2233,223} & x_{2233,233} & 0 & 0\\
2333 & 0 & 0 & 0 & 0 & 0 & 0 & 0 & 0 & x_{2333,233} & x_{2333,333} & 0\\
3333 & 0 & 0 & 0 & 0 & 0 & 0 & 0 & 0 & 0 & x_{3333,333} & 0\\
\end{block}
\end{blockarray}
\]
The row matroid of $I_{n,4}'$ on $\sym_4(V)$ is denoted by $M(I_{n,4}')$.

We can now show that Mason' conjecture is false 
when $k=4$ and $M$ is a non-free uniform matroid. This will follow immediately from our next result by 
matroid duality and Lemma \ref{lem:duality}.

%We will show that the row matroid  of $I_{n,4}'$, $M(I_{n,4}')$, gives a counterexample to Mason's Conjecture.

\begin{theorem}\label{thm:counterexample}
Let $M$ be a copy of $U_n^1$ with groundset $V=\{1,2,\dots,n\}$ and $n\geq 4$. 
Then, $M(I_{n,4}')$ satisfies the Dual Symmetric Q-power property and CycP with respect to $M$,
but does not satisfy RRP.
% \begin{itemize}
% \item[(a)] For every $\alpha\in \sym_i(V)$ with $1\leq i\leq k-1$, 
% $M(I'_{n,4})/ (\sym_k(V)\setminus (\alpha\cdot \sym_{k-i}(V))$ is 
% isomorphic to a quasi-rigidity sequence for $M$;
% \item[(b)] the rank of $M(I_{n,4}')$ is one larger than that of $M(B_{n,4})$.
% \item[(c)] for any cyclic set $Y$ of $M$, 
% %(i.e., for any $Y$ with $|Y|\geq 2$), 
% $\sym_4(Y)$ is a cyclic set in $M(I_{n,4}')$;
% \end{itemize}
\end{theorem}
\begin{proof}
%It is not difficult to check that 
We first note that  $M(B_{n,k})$ satisfies Dual Symmetric Q-power and RRP with respect to $U_n^1$ by Theorem~\ref{thm:transversal}. We can now use (\ref{eq:dual_biliearity_implication}) to deduce that $M(B_{n,k})$ satisfies ExtP with respect to $U_n^1$. Hence $M(B_{n,k})$ satisfies CycP with respect to $U_n^1$ by Theorem~\ref{thm:dual_k}.
%$M(B_{n,k})$ satisfies CycP with respect to $U_n^1$ for all $k\geq 1$. \bill{THIS DOES NOT SEEM TO FOLLOW FROM THE NEW VERSION OF THEOREM 2.9.}
This tells us that
\begin{equation}\label{eq:counter1}
\text{$\sym_k(Y)$ is cyclic in $M(B_{n,k})$ for all $k\geq 1$ and all $Y\subseteq V$ with $|Y|\geq 2$.} 
\end{equation}
% This follows by dualizing a result of Anderson \cite[Proposition 2.21]{And} that 
% %$Y\cdot \sym_i(V)$ is a flat in any $k$-th symmetric power of $M$ if $Y$ is a flat in $M$, 
% \bill{every $k$-th symmetric tensor power of a matroid is a $k$-th symmetric power, and then applying Theorem \ref{thm:transversal}}. (It can also be easily verified directly.)

For simplicity let  $\sym_k(v_1,v_2,\ldots,v_t)=\sym_k(\{v_1,v_2,\ldots,v_t\})$ for all $v_1,v_2,\ldots,v_t\in V$, and put $N=M(B_{n,4})$ and  $N'=M(I_{n,4}')$.
We first analyse the restriction of $N'$ to 
$\sym_4(u,v)$ for two distinct elements $u,v$ in $V$.

\begin{claim}\label{claim:counter1}
Let $u,v$ be distinct elements in $V$.
Then $N'|_{\sym_4(u,v)}=N|_{\sym_4(u,v)}$ and  $\sym_4(u,v)$ is cyclic in $N'$.
%and  $$\rank N'|_{\sym_4(\{u,v\})}=\rank N|_{\sym_4(\{u,v\})}.$$ 
\end{claim}
\begin{proof}
We have $|{\rm supp}(\alpha)|\leq 2$ for each $\alpha\in \sym_4(u,v)$, 
and hence the row of $I_{n,4}'$ indexed by $\alpha$   is obtained from the row of $I_{n,4}$ indexed by $\alpha$  by adding a zero as the last component
%are identical 
(by the definition of $I_{n,4}')$.
This implies that $N|_{\sym_4(u,v)}=N'|_{\sym_4(u,v)}$.
The final part of the claim now follows from (\ref{eq:counter1}).
\end{proof}

We next analyse the restriction of $N'$ to three distinct elements in $V$.

\begin{claim}\label{claim:counter2}
Let $u,v,w$ be distinct elements in $V$.
Then 
%$N'|_{\sym_4(u,v,w)}$ 
$\sym_4(u,v,w)$ is cyclic in $N'$ and
the rank of $N'|_{\sym_4(u,v,w)}$ is  one greater than that of $N|_{\sym_4(u,v,w)}$.
\end{claim}
\begin{proof}
Consider the following partition of $\sym_4(u,v,w)$ into 
the four sets.
\[
%\begin{split}
\sym_4(u,v,w)=
%&
%\{uuvw, uvvw,uvww\} 
(uvw\cdot \{u,v,w\})
%\\&
\sqcup
(uv\cdot \mathrm{Sym}_2(u,v))\sqcup
(u\cdot \mathrm{Sym}_3(u,w))\sqcup
\mathrm{Sym}_4(v,w). 
%\end{split}
\]
Note that the first set in the partition  is the set of all $\alpha\in \sym_4(u,v,w)$ containing each of $u,v,w$
and the last three sets consist of those $\alpha$ missing at least one of
 $u, v, w$.
Then,  $I_{n,4}'$ restricted to the rows indexed by $\sym_4(u,v,w)$ and columns indexed by $\sym_3(u,v,w)\cup\{\spadesuit\}$ has the following block structure:

\begin{equation}\label{eq:block}\small
\begin{blockarray}{c@{\hspace{15pt}}ccccc}
    & uvw & uv\cdot \mathrm{Sym}_1(u,v) & u\cdot \mathrm{Sym}_2(u,w) & \mathrm{Sym}_3(v,w) & \spadesuit \\
\begin{block}{c@{\hspace{15pt}}(ccccc)}
%\mathrm{Sym}     &   \cong I_{3,1}   &      &      &      &      \\
uuvw &   x_{uuvw,uvw}  &  \ast    & \ast     & \ast     & y_{uuvw} \\
uvvw &   x_{uvvw,uvw}  &  \ast    &  \ast    &   \ast   & y_{uvvw} \\
uvww &   x_{uvww,uvw}  &  \ast    &  \ast    & \ast     & y_{uvww} \\
uv\cdot \mathrm{Sym}_2(u,v) 
    & 0    &  I_{2,2} &  \ast    & \ast     &0      \\
u\cdot \mathrm{Sym}_3(u,w) 
    & 0     & 0    &  I_{2,3} &  \ast    & 0     \\
\mathrm{Sym}_4(v,w) 
    & 0     &0      & 0    &  I_{2,4} & 0    \\
\end{block}
\end{blockarray}\,.
\end{equation}
Note that, in the above matrix, the column $\spadesuit$ has zero entries in the last three blocks
since we have $|{\rm supp}(\alpha)|\leq 2$ for all  $\alpha\in uv\cdot \mathrm{Sym}_2(u,v)\cup u\cdot \mathrm{Sym}_3(u,w)\cup \mathrm{Sym}_4(v,w)$. 

By Theorem~\ref{thm:transversal} and Proposition~\ref{prop:edmonds},
$I_{n,k}$ has rank $|\sym_{k-1}(V)|$ 
and hence $I_{n,k}$ has linearly independent columns for all $k\geq 1$.
In particular, each of the diagonal blocks $I_{2,2}$, $I_{2,3}$ and $I_{2,4}$ in the above matrix has linearly independent columns. This implies that the matroids obtained from $N|_{\sym_4(u,v,w)}$ and $N'|_{\sym_4(u,v,w)}$ by contracting  $(uv\cdot \mathrm{Sym}_2(u,v))\sqcup
(u\cdot \mathrm{Sym}_3(u,w))\sqcup
\mathrm{Sym}_4(v,w))$ (corresponding to the last three row blocks in the above block-diagonalized form) are linearly represented by 
\[
\begin{blockarray}{c@{\hspace{15pt}}c}
    & uvw  \\
\begin{block}{c@{\hspace{15pt}}(c)}
uuvw &   x_{uuvw,uvw}    \\
uvvw &   x_{uvvw,uvw}    \\
uvww &   x_{uvww,uvw}   \\
\end{block}
\end{blockarray}
\text{\;\; and \;\;}
\begin{blockarray}{c@{\hspace{15pt}}cc}
    & uvw & \spadesuit \\
\begin{block}{c@{\hspace{15pt}}(cc)}
uuvw &   x_{uuvw,uvw}      & y_{uuvw} \\
uvvw &   x_{uvvw,uvw}     & y_{uvvw} \\
uvww &   x_{uvww,uvw}   & y_{uvww} \\
\end{block}
\end{blockarray},
%\text{\,, respectively. }
\]
respectively. This gives
\begin{align}
%\begin{equation}\label{eq:45}
%\begin{split}
N|_{\sym_4(u,v,w)}/((u v\cdot \mathrm{Sym}_2(u,v))\sqcup
(u\cdot \mathrm{Sym}_3(u,w))\sqcup
\mathrm{Sym}_4(v,w))&\simeq U_3^1, \\
N'|_{\sym_4(u,v,w)}/((u v\cdot \mathrm{Sym}_2(u,v))\sqcup
(u\cdot \mathrm{Sym}_3(u,w))\sqcup
\mathrm{Sym}_4(v,w))&\simeq U_3^2,\label{eq:45}
%\end{split}
%\end{equation}
\end{align}
and hence $\rank N'|_{\sym_4(u,v,w)}=\rank N|_{\sym_4(u,v,w)}+1$.

\smallskip
The assertion that ${\sym_4(u,v,w)}$ is cyclic in $N'$ follows from the following argument.
%from (\ref{eq:counter1}) and (\ref{eq:45}). 
By  (\ref{eq:counter1}) and the structure of the block matrix given in (\ref{eq:block}),
$uv\cdot \mathrm{Sym}_2(u,v)\cup u\cdot \mathrm{Sym}_3(u,w)\cup \mathrm{Sym}_4(v,w)$ is cyclic in $N'$. And $\{uuvw,uvvw,uvww \}$ is contained in a circuit of $N'$ by  (\ref{eq:45}).
%Claim~\ref{claim:counter2}.
\end{proof}

We can now complete the proof of Theorem~\ref{thm:counterexample} by induction on $n$.
For $u\in V$,  the partition of $\sym_4(V)$ into 
two sets given by
\[
\begin{split}
\sym_4(V)=&
(u\cdot  \sym_3(V))\sqcup 
\mathrm{Sym}_4(V-u),
\end{split}
\]
gives rise to the following block diagonalized structure for $I_{n,4}'$:   
\[
%I_{n,4}'=
\begin{blockarray}{c@{\hspace{15pt}}cc}
    &  u\cdot \mathrm{Sym}_2(V) & \mathrm{Sym}_3(V-u)\cup\{\spadesuit\}  \\
\begin{block}{c@{\hspace{15pt}}(cc)}
u\cdot  \sym_3(V)
        &  I_{n,3} & \ast      \\
\mathrm{Sym}_4(V-u)
    & 0        &  I_{n-1,4}'    \\
\end{block}
\end{blockarray}\,.
\]
%Note that the bottom-right block is $I_{n-1,4}'$.
By Claim~\ref{claim:counter2} when $n= 4$, 
and induction when $n\geq 5$, 
$I'_{n-1,4}$ has linearly independent columns. 
%(This is the place we need $n\geq 4$ to construct a counterexample.)
The above block-diagonalized form for $I'_{n,4}$ now implies that  
\[
N'/\sym_4(V-u)=N/\sym_4(V-u)\simeq M(B_{n,3}). 
\]
This and the fact that the Dual Symmetric Q-power property holds for $M(B_{n,3})$ with respect to $M$,  imply that the  
Dual Symmetric Q-power property holds for $N'$ with respect to $M$.
Also, since the rank of $I_{n-1,4}'$ is one greater than that of $I_{n-1,4}$, 
$N'$ has rank strictly greater than the value given by RRP, so RRP does not hold for $N'$.

It remains to verify that $N'$ satisfies CycP with respect to $M=U_n^1$. Choose  $Y\subseteq V$ with $|Y|\geq 2$.
Claims~\ref{claim:counter1} and~\ref{claim:counter2} imply that $\sym_4(Y)$ is cyclic in $N'$ when  $|Y|\leq 3$, so we may assume that $|Y|\geq 4$. Choose any $\alpha=uvwx\in \sym_4(Y)$
and consider the following block form of $I_{|Y|,4}'$.
\[
%I_{|Y|,4}'=
\begin{blockarray}{c@{\hspace{15pt}}cc}
    &  u\cdot \mathrm{Sym}_2(Y) & \mathrm{Sym}_3(Y-u)\cup\{\spadesuit\}  \\
\begin{block}{c@{\hspace{15pt}}(cc)}
u\cdot  \sym_3(Y)
        &  I_{|Y|,3} &   \ast    \\
\mathrm{Sym}_4(Y-u)
    & 0        &  I_{|Y|-1,4}'    \\
\end{block}
\end{blockarray}\,.
\]
Note that $I_{|Y|,3}$ represents $M(B_{|Y|,3})$ while 
$I_{|Y|-1,4}'$ represents $M(I'_{|Y|-1,4})$.
We can apply Claim~\ref{claim:counter2} when $|Y|=4$ 
and induction when $|Y|\geq 5$ to deduce that $I_{|Y|-1,4}'$ has linearly independent columns.
This and (\ref{eq:counter1}) imply that 
$u\cdot  \sym_3(Y)$ is contained in a cyclic set in $N'$.
Similarly, by Claim~\ref{claim:counter2} when $|Y|=4$ and by induction when $|Y|\geq 5$, 
$\mathrm{Sym}_4(Y-u)$ is cyclic in $N'$.
Hence,  $N'$ satisfies CycP.
% Also, since $M(B_{|Y|,3})$ is cyclic by (\ref{eq:counter1}), 
% the row of $I_{|Y|,3}$ indexed by $\{v,w,x\}$ is linearly dependent on the other rows of  $I_{|Y|,3}$.
% This implies that the row of $I_{|Y|,4}'$ indexed by $\alpha=\{u,v,w,z\}$ is linearly dependent on the other rows of  $I_{|Y|,4}'$.  Since $\alpha$  was chosen to be an arbitrary element of $\sym_4(Y)$,   $\sym_4(Y)$ is a cyclic set 
% in $N'$. This completes the proof of both (c) and the theorem.
\end{proof}

% Theorem \ref{thm:counterexample} immediately implies that $U_n^1$ is a counterexample to Conjecture \ref{conj:dual} when $k=4$ and hence, by matroid duality, $U_n^{n-1}$ is a counterexample to Conjecture \ref{conj:And} when $k=4$.

\section{Relations between the dual properties, canonical subsets and the 0-extension operation}\label{sec:lemmas}
We will derive relations linking the properties 
given in Section~\ref{subsec:dual}.  These relations will be used to verify our main results
%Theorems~\ref{thm:2} to \ref{thm:dual_k_weak} 
in the next section.
Our proofs in this section use two key tools which are both motivated by rigidity theory:   canonical subsets of $\sym_k(V)$ with respect to a given matroid $M$ on $V$, and the 0-extension operation. 
%\bill{When $M\cong U_n^{d+1}$, $k=2$ and we %associate $\sym_2(V)$ with the edges of a looped complete graph $K_n^{\circ}$ on %$V$, the 0-extension operation adds a vertex of degree $d$ to a subgraph of %$K_n^{\circ}$, and  the edge sets of every subgraph of $K_n^{\circ}$ which can be %obtained from a copy of $K_{d+1}^{\circ}$ by recursively applying the 0-extension %operation are canonical subsets of $E(K_n^{\circ})$.}\tcolblue{Thanks for %revising my sentences here, but I feel now that this is not easy to understand %for most readers. Maybe we can delete this paragraph?}

%We now move on to the proofs of Theorems~\ref{thm:2}-\ref{thm:dual_k_weak}.
%We first collect relations among properties 
%given in Section~\ref{subsec:dual}
%and then integrate them in the next section.

%Our key tools in the analysis of matroids on $\sym_k(V)$ are canonical subsets of $\sym_k(V)$ and %0-extension operations, which are motivated from rigidity theory. 
%We first introduce canonical subsets and 0-extension operations,
%and then analyse relations among properties via those tools.

Throughout this section, $k$ will denote a fixed positive integer,
$M$ will be  a matroid on 
%a finite set 
$V$ with rank function $r_M$,
and $N$ will be a matroid on $\sym_k(V)$ with rank function $r_N$. 
%We will obtain relationships between various properties for $N$, many of which are direct analogues of properties which are known to hold for the generic $d$-dimensional rigidity matroid on $E(K_n)$.

\subsection{Canonical subsets and the weak canonical base property}\label{subsec:canonical}
%The following generalization of the concept of a canonical subset is key in the extension of our %result from graphs to hypergraphs.
%The key object in the analysis of matroids on $\sym_k(V)$ is a canonical subset of $\sym_k(V)$. We will show that there is a well-structured recursive flat structure in a matroid of  \ref{thm:dual_k_weak}.

Let $X\subseteq V$. We say that
a set $S\subseteq \sym_k(X)$ is  a {\em canonical subset  of $\sym_k(X)$ with respect to $M$} if it satisfies the following conditions. 
\begin{itemize}
\item When $k=1$, $S$ is a base of $M|_X$.
\item When $X$ is $M$-independent, $S=\sym_k(X)$.
\item When $k\geq 2$ and $X$ is $M$-dependent, there is a vertex $v\in X$, called a {\em pivot vertex for $S$},  such that
\begin{itemize}
\item $v$ is not a coloop in $M|_X$,
\item $S\cap \sym_k(X-v)$ is a canonical subset of $\sym_k(X-v)$, and
\item $\lk(v,S)$ is a canonical subset of $\sym_{k-1}(X)$.
\end{itemize}
\end{itemize}

\paragraph{Remark.} This recursive definition ensures that $\sym_k(X)$ has at least one canonical subset $S$. When $k=1$ or $X$ is $M$-independent, this follows directly from the definition. In the case when  $k\geq 2$ and $X$ is $M$-dependent,  we can choose a vertex $v\in X$ which is not a coloop in $M|_X$, and  canonical subsets $S'$ and $S''$ of $\sym_{k}(X-v)$ and $\sym_{k-1}(X)$, respectively, and then put $S=S'\sqcup (v\cdot S'')$.

%When $k=2$, the definition of a canonical subset coincides with the definition of a canonical subset for looped graphs.

%\noindent
\paragraph{Example.}
Consider the case when $M\cong U_5^3$, $V=\{1,2,3,4,5\}$, $X=\{1,2,3,4\}$ and $k=3$.
Then
\[S=\sym_3(\{1,2,3\})\sqcup \{114,124,134,144,234,244,334, 344, 444\}\]
%\[\sym_3(\{1,2,3\})\sqcup \{114,124,224,134,234,244, 344\}\]
is a canonical subset of $\sym_3(X)$ with pivot vertex 4. This follows 
%by choosing 4 to be a pivot vertex for  $X$,
since $\sym_3(\{1,2,3\})$ is a canonical subset of  $\sym_3(\{1,2,3\})$ 
and 
we can show that 
$$\lk(4,S)=\{11,12,13,14,23,24,33,34,44\}$$
is a canonical subset of $\sym_2(X)$ by choosing 2 as its pivot vertex. 
\qed

\medskip

We first show that the cardinality of each canonical subset of $\sym_k(X)$ is equal to $r_N(\sym_k(X))$ for a matroid $N$ with StrongRRP with respect to $M$.

\begin{lemma}\label{lem:canonical_size}
Let $X\subseteq V$ with $n_X=|X|$ and $k$ be a positive integer.
Put $r_X=r_M(X)$.
Then every canonical subset $S$ of $\sym_k(X)$ has size ${n_X+k-1\choose k}-{n_X-r_X+k-1\choose k}$.
\end{lemma}
\begin{proof}
For any three integers $n$, $k$ and $r$, put $f(n,k,r)={n+k-1\choose k}-{n-r+k-1\choose k}$.

The proof proceeds by induction on $n_X+k$.
The base cases are when $k=1$ and when  $X$ is independent in $M$.
If $k=1$, the claim follows from the definition of a canonical subset since we have $S$ is a base of $M|_X$ and $|S|=r_X=f(n_X,1,r_X)$.
If $X$ is independent in $M$, then $r_X=n_X$,  $S=\sym_k(X)$ and we have
$|S|=|\sym_k(X)|={n_X+k-1\choose k}=f(n_X,k,n_X)$.

Hence we may assume that $k\geq 2$ and $X$ is $M$-dependent.
Let $v$ be a pivot vertex for $S$. Then 
$S\cap \sym_k(X-v)$ is a canonical subset of  $\sym_k(X-v)$
and $\lk(v,S)$ is a canonical subset of  $\sym_{k-1}(X)$.
Since $v$ is not a coloop in $M|_X$,
we have $r_X=r_{X-v}$.
Hence, by induction,
$|S\cap \sym_k(X-v)|=f(n_X-1,k,r_X)$ and 
$|\lk(v,S)|=f(n_X,k-1,r_X)$.
By Pascal's formula, we obtain
\begin{align*}
|S|=f(n_X-1,k,r_X)+f(n_X,k-1,r_X)=f(n_X,k,r_X).
\end{align*}
% Suppose $v$ is an $M|_X$-coloop.
% Then, by induction and $r_X=r_{X-v}+1$,
% $|I\cap \sym_k(X-v)|=f(n_X-1,k,r_X-1)$
% and $|\lk(v,I)|=|\sym_{k-1}(X)|={n_X+k-2\choose k}$.
% \begin{align*}
% |I|={n_X-1+k-1\choose k}-{n_X-1-(r_X-1)+k-1\choose k}+{n_X+k-2\choose k}=f(n_X,k,r_X).
% \end{align*}
\end{proof}

Lemma \ref{lem:canonical_size} suggests that the following property for $N$ will be closely related to the Strong
Rigidity Rank Property for $N$.
\begin{description}
\item[The Weak Canonical Base Property (WCBP):]
%A matroid $N$ on $\sym_k(V)$  has the WCBP if, 
For all $X\subseteq V$,
every canonical subset of $\sym_k(X)$ is a base of $N|_{\sym_k(X)}$.
\end{description}

\paragraph{Example.} 
In Section~\ref{subsec:example}, we saw that
the bicircular matroid ${\cal B}$ on $\sym_2(V)$ is the dual of a second symmetric power of $M^*=U_n^{n-1}$.
We will show that ${\cal B}$ satisfies WCBP with respect to $M=U_n^1$. 
Choose $X\subseteq V$. We will use induction on $|X|$ to show that each canonical 
subset $S$ of $\sym_2(X)$ is a base of $\B|_{\sym_2(X)}$,
consisting of the edge set of a subgraph on $X$ in which each connected component is a tree with exactly one additional loop edge.
When $X=\{v\}$, $X$ is 
independent in $M$, and  $S=\{vv\}=\sym_2(X)$ is a  base of $\B|_{\sym_2(X)}$. Hence we may assume 
that $|X|\geq 2$. Then no element of $X$ is a coloop of $M|_X$ since $M= U_n^1$. Let $v$ be a pivot vertex for $S$. Then $S_v=S\cap \sym_2(X-v)$ is a canonical subset of $\sym_2(X-v)$ and $\lk(v,S)$ is a canonical subset of $\sym_1(X)=X$. By induction $S_v$ is  
the edge set of a subgraph with vertex set $X-v$ in which each connected component is a tree with exactly one additional loop edge.
%a spanning subgraph of the complete looped graph on $X-v$ in which each connected component is unicyclic. 
In addition  $|\lk(v,S)|=1$ since $\lk(v,S)$ is a base of $M|_X$. These observations imply that $S$ is  the edge set of  
a subgraph with vertex set $X$ in which each connected component is a tree with exactly one additional loop edge,
and hence $S$ is a base of 
$\B|_{\sym_2(X)}$. 
%\st{COMMENT: In this example, a canonical set is always a spanning tree with exactly one extra loop. Shall we mention this explicitly? Also, the lex-min canonical set given below is the star with a loop at the centre vertex in this case. This might be also worth mentioning.}
%and every  canonical subset $B$ of $\sym_2(X)$ can be obtained by choosing an arbitrary 
%vertex $v\in X$, a canonical subset  $B_v$ of $\sym_2(X-v)$, an arbitrary edge 
%$e\in \lk(v,X)$, and putting  $B=X+e$. Then $B_v$ is a base of $N|_{\sym_2(X-v)}$ 
%by induction, i.e.~$B$ is the edge set of a spanning subgraph on $X-v$ in which each connected %component is unicyclic, and hence $B$ is a base of 
%$N|_{\sym_2(X)}$.
\qed

\medskip

%X$ is independent in $M$, $\sym_1(X)=E(K^0(X))$ is the unique  canonical subset of  $\sym_k(X)$, and $\sym_1(X) is a base of $N|_{\sym_k(X)}$.

We next define a representative example of a canonical subset, which will be useful for checking the WCBP.
Given a total order $<$ on $V$, let $<_{\mathrm{lex}}$ denote the induced
lexicographic order on the words in $\sym_k(V)$. 
%In other words, if each element of $\sym_k(V)$ is regarded as a monomial in $|V|$ indeterminates, then $<_{\mathrm{lex}}$ is the lexicographic monomial order induced by $<$.

We can use the order  $<_{\mathrm{lex}}$ to define a lexicographic order on subsets of $\sym_k(V)$
of the same cardinality: for $E,F\subseteq  \sym_k(V)$ with
\[
E=\{\alpha_1<_{\mathrm{lex}}\alpha_2<_{\mathrm{lex}}\cdots<_{\mathrm{lex}}\alpha_t\}
\quad\text{and}\quad
F=\{\beta_1<_{\mathrm{lex}}\beta_2<_{\mathrm{lex}}\cdots<_{\mathrm{lex}}\beta_t\},
\]
we say that $E$ is lexicographically smaller than $F$ if there exists a
$j\leq t$ such that $\alpha_i=\beta_i$ for all $1\leq i<j$ and
$\alpha_j<_{\mathrm{lex}}\beta_j$.
For $X\subseteq V$, we define the \emph{lex-min canonical subset of
$\sym_k(X)$ with respect to $<$} to be the lexicographically smallest
canonical subset of $\sym_k(X)$ with respect to this 
%induced 
order.

\paragraph{Example (continued).} We saw above that, when $M=U_n^1$ with groundset $V$,  the  canonical subsets of $\sym_2(V)$ are the edge sets of graphs on $V$ in which each connected component is a tree with exactly one additional loop edge. The lex-min canonical subset of
$\sym_2(V)$ will be a star on $V$ centred on the first vertex in the chosen ordering for $V$ with exactly one extra loop edge on this central vertex.

For a set $X\subseteq V$, a total order $<$ on $V$ is said to be {\em $X$-first}
if the first $|X|$ elements in the total order are exactly the elements of $X$.
%A canonical subset of $\sym_k(V)$ is said to be {\em lex-min} if it is the lex-min canonical subset with respect to some total order $<$ on $V$.

\begin{lemma}\label{lem:lex_min}
Let $X\subseteq V$, $B$ be a base of $M|_X$, and $<$ be a $B$-first total order on $V$. Then, $S=B\cdot \sym_{k-1}(X)$ is the lex-min canonical subset of $\sym_k(X)$ with respect to $<$.
\end{lemma}
\begin{proof}
Let $X=\{v_1, v_2,\dots, v_{|X|}\}$ where
$v_1< v_2<\dots< v_{|X|}$. 
Then $B=\{v_1,\dots, v_{r_M(X)}\}$ 
%is a base of $M|_X$ 
since the order $<$ is $B$-first.

We first show that $B\cdot \sym_{k-1}(X)$ is a canonical subset of $\sym_k(X)$ by induction on $k+|X|$.
The base cases when $k=1$ and when $X=B$ follow immediately from the definition of a canonical subset.

Hence we may assume that $k\geq 2$ and $B$ is a proper subset of $X$.
Then, $v_{|X|}$ is not a coloop in $M|_X$.
Consider the bipartition:
\begin{equation}\label{eq:lex_min1}
B\cdot  \sym_{k-1}(X)=(B\cdot  \sym_{k-1}(X-v_{|X|}))\sqcup (v_{|X|}\cdot B\cdot \sym_{k-2}(X)).
\end{equation}
By induction, $B\cdot  \sym_{k-1}(X-v_{|X|})$ is a canonical subset of $\sym_k(X-v_{|X|})$,
and $B\cdot  \sym_{k-2}(X)$ is a canonical subset of $\sym_{k-1}(X)$.
We can now deduce that $B\cdot \sym_{k-1}(X)$ is a canonical subset of $\sym_k(X)$ by choosing $v_{|X|}$ to be  a pivot vertex for $B\cdot \sym_{k-1}(X)$.

To see that $S=B\cdot  \sym_{k-1}(X)$ is the  lex-min canonical subset of $\sym_k(X)$, choose any other canonical subset $S'$ of $\sym_k(X)$.
Then, for any two words $e\in S\setminus S'$ and $f\in S'\setminus S$,
$e$ is smaller than $f$ since $e$ has a letter in $B$ while $f$ has no letter in $B$.
Thus, $S$ is indeed the lex-min canonical subset of $\sym_k(X)$.
\end{proof}

\subsection{The Canonical Base Property}\label{sec:contree}

Unfortunately, the WCBP defined in the previous subsection is not quite strong enough for our analysis of matroids on $\sym_k(V)$ 
%the hypergraph case, 
so we need to introduce a stronger, but more complicated,  property.
To do this, we use a graphical representation of the recursive construction of canonical subsets.
For a set $X\subseteq V$ and an integer $k$, we define a {\em construction  tree} for a canonical subset $S$ of  $\sym_k(X)$
 to be a rooted binary tree $T$ in which each node $\alpha$ is labelled with the following information:
\begin{itemize}
    \item a pair $(X_{\alpha},k_{\alpha})$ of a set $X_{\alpha}\subseteq V$ and a positive integer $k_\alpha$;
    \item a 
    %pivot 
    vertex $v_{\alpha}\in X_{\alpha}$ when $\alpha$ is an internal node of $T$ (i.e. a non-leaf node);
    \item a set $Y_{\alpha}\subseteq \sym_{k_\alpha}(X_{\alpha})$ when $\alpha$ is a leaf node of $T$.
\end{itemize}
In addition, $T$ should have the following properties.
\begin{itemize}
\item The root node  $\alpha_1$ of $T$ has $(X_{\alpha_1},k_{\alpha_1})=(X,k)$, 
%where 
%$X$ and $k$ are the given vertex set and integer 
and $v_{\alpha_1}$ equal to  the pivot vertex used in the first reduction step of the recursive construction of $S$.
\item At each internal node $\alpha$:
\begin{itemize}
\item $k_{\alpha}\geq 2$; 
\item $X_{\alpha}$ is $M$-dependent;
\item $v_{\alpha}$ is not a coloop in  $M|_{X_\alpha}$;
%\end{itemize}
\item 
%each internal node $\alpha$ has exactly two children:  
the left child  of $\alpha$, ${\rm left}(\alpha)$, satisfies $X_{{\rm left}(\alpha)}=X_{\alpha}\setminus \{v_{\alpha}\}$ and $k_{{\rm left}(\alpha)}=k_{\alpha}$
and the right child, ${\rm right}(\alpha)$, satisfies $X_{{\rm right}(\alpha)}=X_{\alpha}$
and $k_{{\rm right}(\alpha)}=k_{\alpha}-1$.
\end{itemize}
\item At each leaf node $\alpha$, either
\begin{itemize}
    \item  $k_{\alpha}=1$ and $Y_{\alpha}$ is a base of $M|_{X_{\alpha}}$, or
    %$Y_{\alpha}=\sym_1(B)$ where $B$ is a base of $M|_{X_{\alpha}}$;
    %\item[Type $\boxtimes$:] this is the case when $X_{\alpha}=\emptyset$, where $Y_{\alpha}=\emptyset$;
    \item $X_{\alpha}$ is $M$-independent and  $Y_{\alpha}=\sym_{k_{\alpha}}(X_{\alpha})$.
\end{itemize}
\end{itemize}

%The construction tree  encodes the recursive construction of a canonical subset of $\sym_k(X)$. More precisely,
Note that, for each internal node $\alpha$ of $T$, the subtree of $T$ rooted at $\alpha$ is the construction tree of a canonical subset of $\sym_{k_{\alpha}}(X_{\alpha})$ in which the first pivot vertex is $v_\alpha$,
since the left subtree encodes a canonical subset of $\sym_{k_{\alpha}}(X_{\alpha}-v_{\alpha})$ and the right subtree encodes a canonical subset of 
$\sym_{k_{\alpha}-1}(X_{\alpha})$.

\medskip

\noindent
\paragraph{Example.}
Figure~\ref{fig:construction_tree} shows the construction tree $T$ of a canonical subset of $\sym_3(V)$ in the case when $V=\{1,2,3,4,5\}$
and $M$ is the affine matroid of points $p_1,\dots, p_5$ in the plane
with $p_1=p_2$ and no other {nontrivial} affine dependencies. Then $n=5$ and $\rank M=3$. 
In this example, the subtree of $T$ rooted at $\alpha_5$ represents the canonical subset 
\[
\sym_2(\{3,4,5\})\sqcup (2\cdot \{2,3,5\})
\]
of 
$\sym_2(\{2,3,4,5\})$, and hence the subtree rooted at $\alpha_2$ represents the canonical subset 
\[
\sym_3(\{3,4,5\})\sqcup\big(2\cdot \big(\sym_2(\{3,4,5\})\sqcup (2\cdot \{2,3,5\})\big)\big)
\]
of $\sym_3(\{2,3,4,5\})$.
\qed

\begin{figure}[t]
\centering
\includegraphics[scale=0.8]{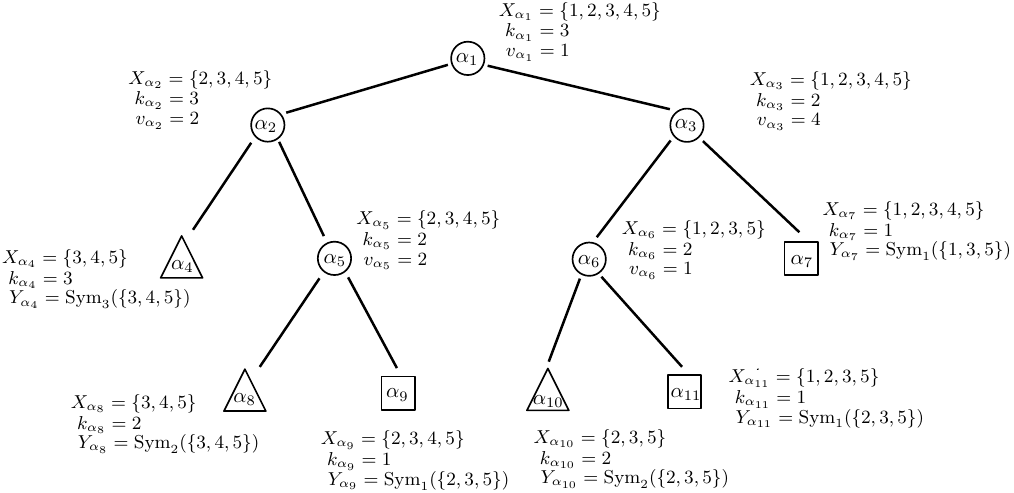}
\caption{A construction tree for a canonical subset of $\sym_3(\{1,2,3,4,5\})$. Internal nodes are denoted by circles. Leaf nodes $\alpha$ are denoted by squares when $k_\alpha=1$, and by triangles when $X_\alpha$ is independent. }
\label{fig:construction_tree}
\end{figure}

\medskip

Construction trees provide additional useful information.
Consider a construction tree $T$ of a canonical subset of $\sym_k(X)$.
For each node $\alpha$ in the tree $T$, define $\tau_{\alpha}\in \sym_{k-k_{\alpha}}(V)$ recursively from the root $\alpha_1$ by the following rule:
\begin{itemize}
\item $\tau_{\alpha_1}=\emptyset$;
\item for each internal node $\alpha$, $\tau_{{\rm left}(\alpha)}=\tau_{\alpha}$ and
%\item 
$\tau_{{\rm right}(\alpha)}=v_{\alpha}\cdot \tau_{\alpha}$,
where $v_{\alpha}$ denotes the pivot vertex at $\alpha$.
\end{itemize}
In addition, for each leaf node $\lambda$ of $T$, we define sets $E_{\lambda}, S_{\lambda}\subseteq \sym_k(X)$ 
%recursively from the leaf nodes of $T$ 
as follows:
\begin{itemize}
\item $E_{\lambda}=\tau_{\lambda}\cdot \sym_{k_{\lambda}}(X_{\lambda})$;
\item $S_{\lambda}=\tau_{\lambda}\cdot Y_{\lambda}$.
%if $\alpha$ is a leaf node;
%\item $S_{\alpha}=S_{{\rm left}(\alpha)}\sqcup S_{{\rm left}(\alpha)}$ if $\alpha$ is an internal node.
\end{itemize}
An easy induction based on these definitions implies that the sets $E_{\lambda},S_{\lambda}$ have the following properties: 
\begin{equation}\label{eq:leaf}
\sym_k(X)=\bigsqcup_{\lambda\in {\rm Lf}(T)} E_{\lambda}
\quad \text{ and } \quad 
S=\bigsqcup_{\lambda\in {\rm Lf}(T)} S_{\lambda},
\end{equation}
where ${\rm Lf}(T)$ is the set of leaf nodes of $T$ and
$S$ is the canonical subset of $\sym_k(X)$ encoded by  the construction tree $T$.
% \st{$L(\cdot)$ was used before denote the loop set of a matroid. There is no crash of notations since $L$ is used for a tree $T$ here. But, ${\rm {\rm Leafs}(T)}$ might be better?} \bill{I agree, but maybe ${\rm {\rm Lf}(T)}$ would be simpler?}

%\medskip

%\noindent
\paragraph{Example.}   In the previous  example,
the leaf nodes are $\alpha_4,\alpha_8,\alpha_9,\alpha_{10},\alpha_{11},\alpha_7$
and the corresponding $\tau_{\alpha}, E_{\alpha}, S_{\alpha}$ are as follows:
\begin{align*}
&\tau_{\alpha_4}=\emptyset,\, E_{\alpha_4}=S_{\alpha_4}=\sym_3(\{3,4,5\}); \\
&\tau_{\alpha_8}=2,\, E_{\alpha_8}= S_{\alpha_8}=2\cdot \sym_2(\{3,4,5\}); \\
&\tau_{\alpha_9}=22, E_{\alpha_9}= 22\cdot  \{2,3,4,5\},\, S_{\alpha_9}=22\cdot \{2,3,5\}; \\
&\tau_{\alpha_{10}}=1,\, E_{\alpha_{10}}= S_{\alpha_{10}}=1\cdot  \sym_2(\{2,3,5\}); \\
&\tau_{\alpha_{11}}=11, \,E_{\alpha_{11}}= 11\cdot  \{1,2,3,5\}, \, S_{\alpha_{11}}=11\cdot \{2,3,5\}; \\
&\tau_{\alpha_{7}}=14, \,E_{\alpha_{7}}= 14\cdot  \{1,2,3,4,5\}, \, S_{\alpha_{7}}=14\cdot \{1,3,5\}.
%&\tau_{\alpha_4}=\emptyset,\, E_{\alpha_4}=S_{\alpha_4}=\sym_3(\{3,4,5\}); \\
%&\tau_{\alpha_8}=\{2\},\, E_{\alpha_8}= S_{\alpha_8}=2\cdot \sym_2(\{3,4,5\}); \\
%&\tau_{\alpha_9}=\{2,2\}, E_{\alpha_9}= 2\cdot 2\cdot  \sym_1(\{2,3,4,5\}), S_{\alpha_9}=2\cdot 2\cdot \{2,3,5\}; \\
%&\tau_{\alpha_{10}}=\{1\},\, E_{\alpha_{10}}= S_{\alpha_{10}}=1\cdot  \sym_2(\{2,3,5\}); \\
%&\tau_{\alpha_{11}}=\{1,1\}, \,E_{\alpha_{11}}= 1\cdot 1\cdot  \sym_1(\{1,2,3,5\}), S_{\alpha_{11}}=1\cdot 1\cdot \{2,3,5\}; \\
%&\tau_{\alpha_{7}}=\{1,4\}, \,E_{\alpha_{7}}= 1\cdot 4\cdot  \sym_1(\{1,2,3,4,5\}), S_{\alpha_{7}}=1\cdot 4\cdot \{1,3,5\}; 
\end{align*}
 %\bill{maybe $E_{\alpha_9}=\{2,2\}\cdot  \sym_1(\{2,3,4,5\})$ would be better? Similarly for \\
 %$E_{\alpha_{11}},E_{\alpha_7},S_{\alpha_{9}},S_{\alpha_{11}}, S_{\alpha_7}$. }
% \st{In view of the new notation system using k-words, any of $2\cdot 2\cdot \{2,3,4,5\}$, $2\cdot 2\cdot \sym_1(\{2,3,4,5\})$ or $\{22\}\cdot \sym_1(\{2,3,4,5\})$ seems okay. What would be the most clear?}
\qed

The left-right ordering of children induces a total order $\leq_T$ on the set ${\rm Lf}(T)$ of leaf nodes  in any construction tree $T$. For any $\lambda,\lambda'\in {\rm Lf}(T)$ with $\lambda'<_T \lambda$, there exists an internal node $\alpha$ of $T$ such that $\lambda'$ is a left descendant of $\alpha$ and $\lambda$ is a right descendant of $\alpha$. Then the  multiplicity of $v_{\alpha}$ in each word of $E_{\lambda'}$ is strictly smaller than that in every word in $E_{\lambda}$, and hence
%(and hence no   in $E_{\lambda'}$ contains $v_\alpha$
%\st{I think this is not accurate because of multiplicity. 
%A word in $E_{\lambda'}$ may contain $v_{\alpha}$.
%The multiplicity of $v_{\alpha}$ in a word of $E_{\lambda'}$ is strictly smaller than that in a word in $E_{\lambda}$.}
%) and $\lambda$ is a right descendant of $\alpha$ (and hence \bill{$  t_\lambda$ contains $v_\alpha$}).  This 
%implies that no edge in $E_{\lambda'}$ contains $v_\alpha$, and that $v_\alpha| t_\lambda$.  Hence gives
\begin{equation}\label{eq:leaf1.5}
\mbox{$\tau_{\lambda}$ is not contained in any word in $E_{\lambda'}$ for all $\lambda,\lambda'\in {\rm Lf}(T)$ with $\lambda'<_T \lambda$.}
\end{equation}

% Since $\{I_{\lambda}: \lambda \text{ is a leaf}\}$ partitions $\sym_k(X)$, this total order induces a partial order on $\sym_k(X)$.
We can use the total order $\leq_T$ to define two new sets of 
%edges 
words $E_{\leq \lambda}$ 
and $S_{\leq \lambda}$ for each leaf node $\lambda$ of $T$:
\begin{equation}\label{eq:leaf2}
E_{\leq \lambda}:=\bigsqcup_{\lambda'\in {\rm Lf}(T)}
%\text{ with }  
^{\lambda'\leq_T \lambda}\, E_{\lambda'}
\quad \text{ and } \quad 
S_{\leq \lambda}:=\bigsqcup_{\lambda'\in {\rm Lf}(T)}
%\text{ with }  
^{\lambda'\leq_T \lambda} \,S_{\lambda'}.
\end{equation}

The following strengthening of  WCBP will play a crucial role in our analysis.

\begin{description}
\item [Canonical Base Property (CBP):] 
% A matroid $N$ on $\sym_k(V)$  has the CBP if, 
For every $X\subseteq V$, 
every construction tree $T$ of a canonical subset of $\sym_k(X)$, 
and every leaf node $\lambda$ in $T$,
$S_{\leq \lambda}$ is a base of $N|_{E_{\leq \lambda}}$.
\end{description}

We close this subsection by justifying our assertion that CBP implies WCBP, and showing that these two properties are equivalent when $k=2$.

\begin{lemma}\label{lem:weakstrongCBP}
Suppose $k$ is a positive integer, $M$ is a matroid with groundset $V$ and $N$ is a matroid on $\sym_k(V)$.\\
(a) If $N$ satisfies  CBP then $N$   satisfies WCBP. \\
(b) If $k=2$ and $N$ satisfies  WCBP, then $N$   satisfies CBP.
\end{lemma}
\begin{proof}
(a) Suppose $N$ satisfies  CBP. Choose a set $X\subseteq V$ and a canonical subset $S$ of $\sym_k(X)$. Let $T$ be a construction tree for $S$ and $\lambda^*$ be the rightmost leaf node of $T$. Then  $E_{\leq \lambda^*}=\sym_k(X)$ and   $S_{\leq \lambda^*}=S$ by (\ref{eq:leaf}) and (\ref{eq:leaf2}). We can now apply CBP to deduce that $S$ is a base of $N|_{\sym_k(X)}$. Hence WCBP holds for $N$. 
\\[1mm]
(b) Suppose that $k=2$ and $N$ satisfies  WCBP. Choose a vertex set $X\subseteq V$, a construction tree $T$ of a canonical subset of $\sym_2(X)$, 
and a leaf node $\lambda$ of $T$. Let $P$ be the path in $T$ from its root vertex $\alpha_1$ to its leftmost leaf vertex $\lambda_0$. The assumption that $k=2$ implies that every leaf node of $T$ other than $\lambda_0$ is the right child  of a vertex of $P$. This tells us that the parent $\alpha$ of $\lambda$ in $T$ belongs to $P$, 
%We can now use (\ref{eq:leaf}) and (\ref{eq:leaf2}) to deduce  that 
%$S_{\leq \lambda}$ is a canonical subset of $E_{\leq \lambda}$, and  that $E_{\leq \lambda}=\sym_2(X')$ for some $X'\subseteq X$. 
%\tcolblue{Since $\tau_{\alpha}=\emptyset$, the subtree rooted at $\alpha$ is a representation tree of a canonical set of $\sym_2(X')$ for some $X'\subseteq X$. 
%Since $\tau_{\alpha}=\emptyset$, 
 and the subtree rooted at $\alpha$ is a representation tree of a canonical subset of $\sym_2(X_\alpha)$. 
%for some $X'\subseteq X$.
In addition,
$E_{\leq \lambda}=\sym_2(X_\alpha)$  and $S_{\leq \lambda}$ is a canonical subset of $E_{\leq \lambda}$.
The assumption that $N$ satisfies  WCBP now implies that $S_{\leq \lambda}$ is a base of $N|_{E_{\leq \lambda}}$ and hence CBP holds for $N$.
\end{proof}

\subsection{The 0-extension operation}
We will refer to the operation of extending a subset of $\sym_k(V)$ using the procedure described in  0-ExtP as an $M$-valid, 0-extension operation.
More precisely, for 
%$k\geq 1$ and 
$I\subseteq \sym_k(V)$,
an  {\em $M$-valid,  0-extension operation} chooses an $M$-independent vertex set $X\subseteq V$ and $\tau\in \sym_{k-1}(V)$ with  $\lk(\tau,I)=\emptyset$,
and then adds  each of the elements in $\tau\cdot X$ to $I$.
We will refer to the pair $(X,\tau)$ as the {\em base} of this 0-extension operation. Note that when $X=\emptyset$, we have $\tau\cdot X=\{\emptyset\}$.
%\bill{and the operation leaves $I$ unchanged. I THINK THIS IS FALSE WHEN $I=\emptyset$ SINCE WE HAVE $I\cup (\tau\cdot X)=\{\emptyset\}$.}
%\end{itemize}

% \noindent
% \paragraph{Example.}
% Consider the case when $V=\{1,2,3,4\}$ and $M$ is the uniform matroid on $V$ of rank two.
% For $I=\{111,112,122,222,133,333\}\subset \sym_3(V)$, 
% %we can apply the  
% %coning operation at vertex $4$ to $I$ to create
% %$I\sqcup (4\cdot \sym_2(\{1,2,3,4\}))$. 
% %And we can apply a $0$-extension operation with the base $23$  to add either $\{123,223\}$, $\{123,233\}$, or $\{223,233\}$ to $I$. And the 
% \bill{the $0$-extension operation with  base $(\{1,2\},23)$  adds the  
% edges $123$ and  $223$ to $I$.}

We will show that, for any $X\subseteq V$,  every canonical subset of $\sym_k(X)$ can be constructed from the empty set by a sequence of $M$-valid, 0-extension operations. We first consider the case when $X$ is $M$-independent.
\begin{lemma}\label{lem:canonical_construction_independent}
Let $X\subseteq V$ be $M$-independent and $k$ be a positive integer.
Then $\sym_k(X)$ can be constructed from the empty set by a sequence of $M$-valid, $0$-extension operations.
\end{lemma}
\begin{proof}
We proceed by induction on $|X|+k$. 
The cases when $X=\emptyset$ or $k=1$ are trivial. Hence we may assume that $|X|\geq 1$ and $k\geq 2$.

Choose  $v\in X$ and consider the sets $\sym_k(X-v)$ and $\sym_{k-1}(X)$.
By induction, $\sym_k(X-v)$ and $\sym_{k-1}(X)$ can be constructed from $\emptyset$ by a sequence of $M$-valid, $0$-extension operations. Let the bases of the operations in the latter sequence be  $(X_i,\tau_i)$ for $1\leq i\leq m$.
%, and $(X_j',\tau_j')$ for $1\leq j\leq m'$, respectively. 
Since no word in $\sym_k(X-v)$ contains $v$, the sequence of $0$-extension operations for constructing $\sym_{k-1}(X)$ from $\emptyset$ can be converted to
a sequence of $0$-extension operations for constructing $\sym_{k}(X)$ from $\sym_k(X-v)$
by changing its bases to $(X_i,v\cdot \tau_i)$ for all $1\leq i\leq m$.
Then the concatenation of the sequence for $\sym_k(X-v)$ with the sequence for 
$\sym_{k}(X)$
%with bases 
%$(X_i,\tau_i)$ for $1\leq i\leq m$ and that with the bases for
%$(X_j', v\cdot \tau_j')$ for $1\leq j\leq m'$
 is  
a  sequence of $M$-valid, 0-extension operations for constructing $\sym_k(X)$ from $\emptyset$.
\end{proof}

%a sequence $0$-extensions for $v\cdot \sym_{k-1}(X)$ \bill{starting from $\sym_k(X-v)$
%with bases $X\tau_i$ of the $i$'th  0-extension in the sequence to $v\cdot \tau_i$.}
%Since no element in $\sym_k(X-v)$ contains $v$,
%the concatenation of the sequence for $\sym_k(X-v)$
%and that for $v\cdot \sym_{k-1}(X)$ is  
%a valid sequence of 0-extension operations for $\sym_k(X)$.
%\end{proof}

% The next Lemma~\ref{lem:canonical_construction} is a special case of Lemma~\ref{lem:canonical_construction_finer} below and can be skipped.
% We keep this lemma to explain a basic idea in the construction.
\begin{lemma}\label{lem:canonical_construction}
Let $X\subseteq V$, $k$ be a positive integer and $S$ be a canonical subset of $\sym_k(X)$.
Then $S$ can be constructed from the empty set by a sequence of $M$-valid, $0$-extension operations.
\end{lemma}
\begin{proof}
We proceed by induction on $|X|+k$.
We first consider the cases when $k=1$ or $X$ is $M$-independent.
If $k=1$, then $S$ is a base of $M|_X$ and hence $S$
%-independent and it 
can be constructed
 from the empty set by a single  $M$-valid, 0-extension with base $(S,\emptyset)$.
%(Note that, when $k=1$, a 0-extension means the addition of an $M$-independent set to the empty set.)
If $X$ is $M$-independent, then $S=\sym_k(X)$ and hence $S$ can be constructed
from the empty set by a sequence of  $M$-valid, 0-extensions by Lemma~\ref{lem:canonical_construction_independent}.

For the inductive step, we assume that $k\geq 2$ and $X$ is $M$-dependent.
Then, by the definition of a canonical subset, there is a pivot vertex $v\in X$ such that 
$S\cap \sym_k(X-v)$ is a canonical subset of $\sym_k(X-v)$
and $\lk(v,S)$ is a canonical subset of $\sym_{k-1}(X)$.
By induction $S\cap \sym_k(X-v)$ and $\lk(v,S)$ can both be constructed from the empty set by a sequence of  $M$-valid, 0-extension operations. 
Let the bases for the latter sequence 
%these operations 
be 
$(X_i,\tau_i)$ for $1\leq i\leq m$.
%\tcolblue{
Since no element in $\sym_k(X-v)$ contains $v$, the sequence 
for constructing $\lk(v,S)$ from $\emptyset$ can be converted to
a sequence 
for constructing $S$ from $S\cap \sym_k(X-v)$
by changing its bases to $(X_i,v\cdot \tau_i)$ for all $1\leq i\leq m$.
Then the concatenation of the sequence for $S\cap \sym_k(X-v)$ with 
the sequence with the bases $(X_i,v\cdot \tau_i)$ for $1\leq i\leq m$
 is  
a sequence of $M$-valid, 0-extension operations for constructing $S$ from $\emptyset$.
%}
% Since no element in $\sym_k(X-v)$ contains $v$, the sequence 
% %of $0$-extension operations 
% for constructing $S\cap \sym_{k}(X-v)$ from $\emptyset$ can be converted to
% a sequence 
% %of $0$-extension operations 
% for constructing $S$ from $\lk(v,S)$
% by changing its bases to $(X_i,v\cdot \tau_i)$ for all $1\leq i\leq m$.
% Then the concatenation of the sequence for $\lk(v,S)$ with 
% %the bases $(X_i,\tau_i)$ for $1\leq i\leq m$ and that with the bases 
% %$(X_j', v\cdot \tau_j')$ for $1\leq j\leq m'$
% the sequence for $S$
%  is  
% a sequence of $M$-valid, 0-extension operations for constructing $S$ from $\emptyset$.
% %Appending $v$ to the base of each operation in the construction of $\lk(v,S)$, we have a construction sequence of $v\cdot \lk(v,S)$.
% %The construction sequence of $S\cap \sym_k(X-v)$ followed by that of $v\cdot \lk(v,S)$ gives a construction sequence of $S$.
\end{proof}

A similar argument gives the following more detailed statement.
%concerning the sets  $ S_{\lambda}$ and $E_{\leq \lambda}$ defined in Section \ref{sec:contree} for each leaf node $\lambda$ of a construction tree.

\begin{lemma}\label{lem:canonical_construction_finer} 
Suppose $X\subseteq V$, $k$ is a positive integer, 
$T$ is a construction tree for a canonical subset $S$ of $\sym_k(X)$, 
 $\lambda_1\leq_T \lambda_2$ are two leaf nodes of $T$, and  $F\subseteq E_{\leq \lambda_1}$. Then
\[
J:=F\sqcup \bigsqcup_{\lambda_1\, <_T \,\lambda\,\leq_T\, \lambda_2} S_{\lambda}
\]
can be obtained from $F$ by a sequence of $M$-valid, 0-extension operations.
\end{lemma}
\begin{proof}
    We proceed by induction on 
    the difference between the ranks of $\lambda_2$ and $\lambda_1$ in the total order $<_T$.
    The base case when $\lambda_1=\lambda_2$ is trivial so we may assume that
$\lambda_1\neq \lambda_2$.

    Let $\lambda'$ be the leaf node immediately preceding $\lambda_2$ in the total order $<_T$.
    By induction, 
    \begin{equation}\label{eq:construction_from_F'}
    \text{
    $J':=F\sqcup \bigsqcup_{\lambda_1\, <_T \,\lambda\,\leq_T\, \lambda'} S_{\lambda}$
    can be obtained from $F$ by a sequence of $M$-valid, 0-extensions.}
    \end{equation}
     And since $
     %F\cup \bigcup_{\lambda_1\, <_T \,\lambda\,\leq_T\, \lambda'} S_{\lambda}
     J'\subseteq E_{\leq \lambda'}$, we can use (\ref{eq:leaf1.5}) to deduce that
    \begin{equation}\label{eq:construction_from_F_2'}
    \text{no word in 
    %$F\cup \bigcup_{\lambda_1\, <_T \,\lambda\,\leq_T\, \lambda'} S_{\lambda}$ 
    $J'$ contains $\tau_{\lambda_2}$.}
    \end{equation}
    In addition, the hypothesis that $\lambda_2$ is a leaf node of $T$ implies that $S_{\lambda_2}=\tau_{\lambda_2}\cdot Y_{\lambda_2}$, and either  $k_{\lambda_2}=1$ or $X_{\lambda_2}$ is independent in $M$.

    Suppose $k_{\lambda_2}=1$. Then $Y_{\lambda_2}$ is a base of $M|_{X_{\lambda_2}}$ and, by (\ref{eq:construction_from_F_2'}), we can obtain 
    %$F\cup \bigcup_{\lambda_1\, <_T \,\lambda\,\leq_T\, \lambda_2} S_{\lambda}$
    $J=J'\sqcup S_{\lambda_2}$ from $F$ by performing  the 0-extension sequence for $J'$ given by (\ref{eq:construction_from_F'})
    followed by a single 0-extension with  base $(Y_{\lambda_2},\tau_{\lambda_2})$.
    
    It remains to consider the case when $X_{\lambda_2}$ is $M$-independent.
Then $Y_{\lambda_2}=\sym_{k_{\lambda_2}}(X_{\lambda_2})$. 
%and $S_{\lambda_2}$ is in the form $ S_{\lambda_2}=\tau_{\lambda_2}\cdot Y_{\lambda_2}$.
    %for some$\tau_{\lambda_2}\in \sym_{k-k_{\lambda_2}}(X)$. 
    By Lemma~\ref{lem:canonical_construction_independent},
    $\sym_{k_{\lambda_2}}(X_{\lambda_2})$ can be constructed from the empty set 
    by a sequence of $M$-valid, 0-extensions. Let the bases of the operations in this sequence be  $(X_i,\tau_i)$ for $1\leq i\leq m$.  By (\ref{eq:construction_from_F_2'}), the sequence with bases $(X_i,\tau_{\lambda_2}\cdot \tau_i)$ will be an $M$-valid 0-extension sequence for 
constructing $J=J'\sqcup S_{\lambda_2}$ from $J'$. The concatenation of 
     the sequence of $J'$ given by  (\ref{eq:construction_from_F'})
    with the sequence for $J$ now gives the required $M$-valid, 0-extension sequence for $J$.
    \end{proof}

\subsection{Properties implied by  the Canonical Base Property}
We show that CBP implies all of the properties listed in the bullet points of Theorem~\ref{thm:dual_k_weak}.
We will need  the following two preliminary lemmas.

\begin{lemma}\label{lem:coning}
Suppose $N$ satisfies StrongRRP. Then  $N$ satisfies ConingP.
\end{lemma}
\begin{proof}
%Suppose that SRP holds for $N$. 
To verify ConingP, we choose a set $X\subseteq V$ and an element $v\in V\setminus \cl_M(X)$.
%any $N$-independent set of edges $I\subseteq \sym_k(V)$ and vertex $v\not\in {\rm cl}_M(V(I))$.
%Let $X=V(I)$.
Since $v\notin {\rm cl}_M(X)$, $r_M(X+v)=r_M(X)+1$.
Hence, by StrongRRP, 
\begin{align*}
    &r_N(\sym_k(X+v))-r_N(\sym_k(X))\\
    &={|X|+1+k-1\choose k}-{|X|+1-(r_M(X)+1)+k-1\choose k} 
      -{|X|+k-1\choose k}+{|X|-r_M(X)+k-1\choose k}\\
    &={|X|+1+(k-1)-1\choose k-1}=|v\cdot \sym_{k-1}(X+v)|.
\end{align*}
This implies that each edge in $v\cdot \sym_{k-1}(X+v)$ is a coloop of $N|_{\sym_k(X+v)}$
%This implies $I\sqcup v\cdot \sym_{k-1}(X+v)$ is independent in $N$ 
and hence ConingP holds for $N$.
\end{proof}

\begin{lemma}\label{lem:canonical_avoiding}
Suppose $C\subseteq V$ is a cyclic set in $M$ and $\tau\in \sym_k(C)$. Then there exists  a canonical subset of $\sym_k(C)$ that does not contain $\tau$.
\end{lemma}
\begin{proof}
Let $\tau=u_1u_2\dots u_k$.
We proceed by induction on $k$.
If $k=1$ then, since $C$ is cyclic in $M$,  there exists a base $B$ of $M|_C$ avoiding $u_1$. Then $B$ is a canonical subset of $C=\sym_1(C)$ which does not contain $u_1$.

Hence we may assume that $k\geq 2$.
By induction, there is a canonical subset $S'$ of $\sym_{k-1}(C)$ that avoids $u_1u_2\dots u_{k-1}$. 
Let $S''$ be a canonical subset of $\sym_k(C-u_k)$. (Note that $S''$ exists by the remark after the definition of canonical subsets.)
%WE SHOULD SHOW THAT WE CAN ALWAYS FIND A canonical subset. 
Since  $u_k\in C$ and $C$ is a cyclic set in $M$, $u_k$ is not a coloop in $M|_C$. Hence $S=(u_k\cdot S')\sqcup S''$ is the required  canonical subset of $\sym_k(C)$ that does not contain $\tau$.
%The union of $u_k\cdot S'$ and any canonical subset of $\sym_k(C-u_k)$ would be a canonical subset of $\sym_k(C)$ that does not contain $\tau$.
\end{proof}

\begin{lemma}\label{lem:weakCBP}
Suppose $N$ satisfies WCBP. Then $N$ satisfies RRP, StrongRRP, ConingP, and CycP.
%CP, and 1GP.
\end{lemma}
\begin{proof}
%\noindent{Proof of RP, StrongRRP, and ConingP:}
The assertions that RRP and StrongRRP hold for $N$ are direct consequences of Lemma~\ref{lem:canonical_size} and WCBP.
The assertion that ConingP holds for $N$ follows from StrongRRP and Lemma~\ref{lem:coning}.

%\medskip
%\noindent{Proof of CycP:}
To see that $N$ satisfies CycP, choose any cyclic set $C$ of $M$.
Lemma~\ref{lem:canonical_avoiding} implies that, for each $\tau\in \sym_k(C)$, there is a canonical subset $S$ of $\sym_k(C)$ that does not contain $\tau$.
Since $S$ is a base of $N|_{\sym_k(C)}$ by WCBP, this implies that $\sym_k(C)$ is a cyclic set of $N$.
%\medskip
%\noindent{Proof of CP:}
% To see that $N$ satisfies CP, choose a cycle $C$ of $M$.
% By CycP, $\sym_k(C)$ is cyclic.
% Also, by CBP and Lemma~\ref{lem:canonical_size}, $\sym_k(C)$ has rank equal to $|\sym_k(C)|-1$ in $N$.
% Hence, $\sym_k(C)$ is a circuit of $N$.
%\medskip
%\noindent{Proof of 1GP:}
% To see that $N$ satisfies 1GP, choose  $X\subseteq V$, $v\in {\rm cl}_M(X)$, and $Y\subseteq X+v$ such that ${\rm cl}_M(X)={\rm cl}_M(Y)$.
% Then ${\rm cl}_M(X+v)={\rm cl}_M(Y)$, and hence we can choose a base $B$ of $M|(X+v)$ with $B\subseteq Y$.
% Consider a $B$-first total order  $<$ on $V$,
% and let $S$ be the lex-min canonical subset of $\sym_k(X+v)$ with respect to $<$.
% Then Lemma~\ref{lem:lex_min} and the assumption that $B\subseteq Y\subseteq X+v$ tell us that %$S=\sym_1(B)\cdot \sym_{k-1}(X+v)$.
% %In addition,  
% $$S=\sym_1(B)\cdot \sym_{k-1}(X+v)\subseteq \sym_k(X)\cup (v\cdot Y\cdot \sym_{k-2}(X+v))).$$ 
% Since $S$ is a base of $\sym_k(X+v)$ by  CBP, this implies that $$\sym_k(X+v)={\rm cl}_N(S)\subseteq {\rm cl}_N(\sym_k(X)\cup (v\cdot Y\cdot \sym_{k-2}(X+v)))).$$
\end{proof}

\begin{lemma}\label{lem:strongCBP}
Suppose $N$ satisfies CBP. Then $N$ satisfies 0-ExtP.
\end{lemma}
\begin{proof}
Choose an $N$-independent set $I\subseteq \sym_k(V)$,
a word $\tau= u_1u_2\dots u_{k-1}\in \sym_{k-1}(V)$ with $\lk(\tau,I)=\emptyset$,
and an $M$-independent set  $X\subseteq V$.
We need to show that $I\sqcup (\tau\cdot X)$ is $N$-independent. 
We may assume, without loss of generality, that $X$ is  a base of $M$. 

Since $N$ satisfies CBP, $N$ satisfies WCBP by Lemma \ref{lem:weakstrongCBP} and hence also satisfies ConingP by Lemma \ref{lem:weakCBP}.

We first consider the case when  $u_i$ is a coloop of $M$ for some $1\leq i\leq k-1$. Relabelling if necessary, we may assume that $u_1$ is a coloop of $M$. Then $u_1\not\in\cl_M(V-u_1)$ and we can use ConingP
%Lemmas \ref{lem:weakstrongCBP} and  \ref{lem:coning} 
to deduce that every edge in $u_1\cdot \sym_{k-1}(V)$ is a coloop in $N$. Since $\tau\cdot X\subseteq u_1\cdot \sym_{k-1}(V)$, this implies that 
$I\sqcup (\tau\cdot X)$ is $N$-independent.

It remains to consider the case when $u_i$ is not a coloop of $M$ for all $1\leq i\leq k-1$.
Consider
%We shall look at a construction tree of a canonical subset of $\sym_k(U)$ using the notation given in Section~\ref{subsec:canonical}. Specifically, consider 
a construction tree $T$ for a canonical subset $S$ of $\sym_k(V)$. Adopting the notation given in Section~\ref{subsec:canonical}, we let $\alpha_1$ be the root vertex of $T$, $\lambda^*$ be the rightmost leaf node of $T$ and $\alpha_2,\alpha_3, \ldots, \alpha_{k-1}$ be the internal nodes  on the path in $T$ from $\alpha_1$ to $\lambda^*$. Then $X_{\alpha_i}=V$ for all $1\leq i\leq k-1$.  Since $u_i$ is not a coloop of $M$ for all $1\leq i\leq k-1$ and $X$ is  a base of $M$, we may assume that  
the pivot vertex $v_{\alpha_i}$ associated to ${\alpha_i}$ is $u_i$, and that 
the  set $Y_{\lambda^*}$ associated to $\lambda^*$ is $X$.
%and the path in $T$ from ${\alpha_1}$ to  $\lambda^*$ consists of nodes with pivot vertices $u_1, u_2,\dots, u_{k-1}$ in this order.
Then the sets $E_{\lambda^*},S_{\lambda^*}$ associated to $\lambda^*$ satisfy
%in terms of notation given in (\ref{eq:leaf}), 
\begin{equation}\label{eq:0ext1}
E_{\lambda^*}=\tau\cdot V \quad \text{ and } \quad S_{\lambda^*}=\tau\cdot X.
\end{equation}
%where  $\lambda^*$ denotes the rightmost leaf of the construction tree.

Let $\lambda'$ be the leaf node  immediately preceding $\lambda^*$ in $T$.
Then, $\sym_k(V)=E_{\leq \lambda'}\sqcup E_{\lambda^*}$ and 
$S=S_{\leq \lambda'}\sqcup S_{\lambda^*}$ by (\ref{eq:leaf}). In addition,
CBP implies that $S_{\leq \lambda'}$ is a base of $N|_{E_{\leq \lambda'}}$.
Since $I$ has no element that contains $\tau$, 
(\ref{eq:0ext1}) implies that
$I\cap E_{\lambda^*}=\emptyset$, and hence $I\subseteq E_{\leq  \lambda'}$.
In particular, $I$ is in the span of $S_{\leq \lambda'}$.
Since $S=S_{\leq \lambda'}\sqcup S_{\lambda^*}$
and $S$ is $N$-independent by CBP,
$I\sqcup S_{\lambda^*}$ is $N$-independent. 
We can now use (\ref{eq:0ext1}) to deduce that $I\sqcup (\tau\cdot X)$ is $N$-independent.
\end{proof}

\subsection{Further relations between properties}
We close this section with two additional lemmas which will play a significant role in  the proofs of our main results.

\begin{lemma}\label{lem:R_CBP}
Suppose $N$ satisfies RRP, 0-ExtP and ConingP.
Then $N$ satisfies CBP.
\end{lemma}
\begin{proof}
Note that, by 0-ExtP and Lemma~\ref{lem:canonical_construction},
every canonical subset of $\sym_k(X)$ is independent in $N$ for all $X\subseteq V$.
We first prove two claims which will imply that StrongRRP holds for $N$.

\begin{claim}\label{claim:R_CBP1}
Suppose $X\subseteq V$ is a  spanning set of $M$. Then every canonical subset of $\sym_k(X)$ is a base of $N|_{\sym_k(X)}$.
\end{claim}
\begin{proof}
The claim is trivial if $k-1$, and we assume $k\geq 2$.

We  proceed by induction on $|V\setminus X|$.
%for $X\subseteq V$.
The base case (when $X=V$) follows by RRP and the fact that every canonical subset of $\sym_k(V)$ is $N$-independent. Hence we may assume that $X\neq V$.

Choose  a vertex  $v\in V\setminus X$. Since $X$ is $M$-spanning, $v\in {\rm cl}_M(X)$.
This implies that, for any canonical subset $S_{X,k}$ of $\sym_k(X)$ and any canonical subset $S_{X+v,k-1}$ of $\sym_{k-1}(X+v)$, we have
\[
S_{X+v,k}:=S_{X,k}\sqcup (v\cdot S_{X+v,k-1})
\]
is a canonical subset of $\sym_k(X+v)$.
By induction, $S_{X+v,k}$ is a base of $N|_{\sym_k(X+v)}$.

$S_{X+v,k}$ admits a construction tree $T$ such that the root has  the pivot vertex $v$, 
the left subtree $T_1$ of the root is a construction tree of  $S_{X,k}$ 
and the right subtree $T_2$ of the root is a construction tree of  $S_{X+v,k-1}$.
Let $\lambda_i$ be the rightmost leaf in $T_i$ for each $i=1,2$.
Then, in $T$, 
\begin{align}
S_{X,k}&=S_{\leq \lambda_1}\subseteq E_{\leq \lambda_1}=\sym_k(X) \label{eq:512_1}\\
v\cdot S_{X+v,k-1}&=\bigsqcup_{\lambda_1<\lambda\leq \lambda_2} S_{\lambda}.    \label{eq:512_2}
\end{align}

Our goal is to prove that $S_{X,k}$ is a base of $N_{\sym_k(X)}$.
Suppose, for a contradiction, that $S_{X,k}$ is not a base of $N_{\sym_k(X)}$. Since $S_{X,k}$ is a canonical subset, it is $N$-independent and hence 
there exists a word $e\in \sym_k(X)\setminus S_{X,k}$ such that $S_{X,k}+e$ is $N$-independent.
By (\ref{eq:512_1}), $S_{X,k}+e\subseteq E_{\leq \lambda_1}$.
We can now apply Lemma~\ref{lem:canonical_construction_finer} with $F=S_{X,k}+e$
and (\ref{eq:512_2}) to deduce that  $(S_{X,k}+e)\sqcup (v\cdot S_{X+v,k-1})$ can be
obtained from $S_{X,k}+e$ by a sequence of $M$-valid, 0-extension operations.
The hypothesis that $N$ satisfies 0-ExtP  now implies that  $(S_{X,k}+e)\sqcup (v\cdot S_{X+v,k-1})$ is $N$-independent. This contradicts the fact  that 
$S_{X,k}\sqcup (v\cdot S_{X+v,k-1})$ is a base of $N|_{\sym_k(X+v)}$.

%  Suppose, for a contradiction, that $S_{X,k}$ is not a base of $\sym_k(X)$. Since $S_{X,k}$ is a canonical subset, it is $N$-independent and hence 
% there exists a word $e\in \sym_k(X)\setminus S_{X,k}$ such that $S_{X,k}+e$ is $N$-independent.
% Since $(S_{X,k}+e)\sqcup (v\cdot S_{X,k-1})$ can be obtained from $S_{X,k}+e$
% by a sequence of $M$-valid, 0-extension by  Lemma~\ref{lem:canonical_construction_finer} \bill{ADD MORE EXPLANATION}\st{Agree. I will do it tomorrow.}, 0-ExtP  implies that  $(S_{X,K}+e)\sqcup (v\cdot S_{X,k-1})$ is $N$-independent. This contradicts the fact  that 
% $S_{X,k}\sqcup (v\cdot S_{X,k-1})$ is a base of $\sym_k(X+v)$.

% Hence $S_{X,k}$ is a base of $\sym_k(X)$ and the claim follows by induction.
\end{proof}

\begin{claim}\label{claim:R_CBP2}
$N$ satisfies StrongRRP. 
%\bill{DOES THIS FOLLOW IMMEDIATELY FROM CLAIM 5.12 AND LEMMA 5.1?}
%\st{Not sure. Claim 5.12 is proved only for a spanning set $X$ of $M$. As far as I understand, Claim 5.13 is exactly the place we use ConingP. In fact, the instance constructed in Theorem 3.5 is a counterexample to your guess: the instance of Theoreom 3.5 satisfies RRP and 0-Ext, but does not satisfy StrongRRP.}
\end{claim}
\begin{proof}
We use induction on $|V\setminus X|$ to show that the rank formula in StrongRRP holds for all $X\subseteq V$.
The case when $X$ is $M$-spanning follows from Lemma \ref{lem:canonical_size} and Claim~\ref{claim:R_CBP1}.

Suppose $X$ is not $M$-spanning. 
Then, there exists a vertex $v\in V\setminus  {\rm cl}_M(X)$.
Since every canonical subset of $\sym_k(X)$ is $N$-independent,
$r_N(\sym_k(X))\geq {|X|+k-1\choose k}-{|X|-r_M(X)+k-1\choose k}$.
We can now use ConingP to obtain 
\begin{align*}
r_N(\sym_k(X+v))&\geq r_N(\sym_k(X))+|v\cdot \sym_{k-1}(X+v)| 
%\quad (\text{by ConingP})
\\
&\geq 
{|X|+k-1\choose k}-{|X|-r_M(X)+k-1\choose k}+{|X+v|+k-2\choose k-1} \\
&={|X+v|+k-1\choose k}-{|X+v|-r_M(X+v)+k-1\choose k}.
\end{align*}
The induction hypothesis for $X+v$ implies that each inequality holds with equality.
In particular, $r_N(\sym_k(X))= {|X|+k-1\choose k}-{|X|-r_M(X)+k-1\choose k}$ holds.
\end{proof}

We can now show that CBP holds for $N$.
Choose $X\subseteq V$ and let $T$ be the construction tree of a canonical subset $S$ of $\sym_k(X)$.

Suppose, for a contradiction, that $S_{\leq \lambda}$ is not a base of $N|_{E_{\leq \lambda}}$ for some leaf node $\lambda$ of $T$. Since $S$ is a canonical subset of $\sym_k(X)$, $S$ is $N$-independent. Since $S_{\leq \lambda}\subseteq S$, it is also
%0-ExtP and Lemma~\ref{lem:canonical_construction} imply that 
$N$-independent.
Hence, $S_{\leq \lambda}$ is not a spanning set of $N|_{E_{\leq \lambda}}$
and we may choose an $e\in E_{\leq \lambda}\setminus S_{\leq \lambda}$ such that $S_{\leq \lambda}+e$ is $N$-independent.
By Lemma~\ref{lem:canonical_construction_finer} (with $\lambda_1$ equal to $\lambda$ and $\lambda_2$ equal to the rightmost leaf of $T$),
$S+e$ can be obtained from $S_{\leq \lambda}+e$ by a sequence of 0-extensions.
We can now use 0-ExtP to deduce that $S+e$ is $N$-independent
and hence $r_N(\sym_k(X))$ is larger than the size of the canonical subset $S$ for $\sym_k(X)$.
This contradicts   Claim~\ref{claim:R_CBP2} and completes the proof of the lemma.
\end{proof}

\begin{lemma}\label{lem:ext_star}
Suppose $N$ satisfies ExtP and RRP. Then $N$ satisfies Dual Multilinearity. \end{lemma}
\begin{proof}
We first note that the hypothesis that $N$ satisfies ExtP implies that $N$  satisfies both ConingP and 0-ExtP by Lemma~\ref{lem:extension_coning} and (\ref{eq:dual_biliearity_implication}), respectively.

To verify Dual Multilinearity, choose $\tau\in \sym_{k-1}(V)$.
If some letter $v$ of $\tau$ is a coloop of $M$ then,
%,by Lemma~\ref{lem:extension_coning}, 
since $N$ satisfies ConingP, all edges in $v\cdot \sym_{k-1}(V)$ are coloops in $N$.
%In particular, all edges in $\tau \cdot V$ are $N$-coloops,
%i.e. $N^*|(\tau\cdot V)$ is the rank zero matroid.
Thus, we may assume that no letter in $\tau$ is a coloop in $M$.
Let $$E:=\sym_k(V)\setminus (\tau\cdot V),$$
and let $B$ be a base of $N|_E$.
We need to show that $N/E$ is isomorphic to $M$ via the bijection $\tau\cdot v \mapsto v$ between  $\tau\cdot V$ and $V$.
%Let $B$ be a base of $N|E$.

Let $X$ be an independent set of $M$. 
Since 
%$N$ satisfies ExtP and  
no word in $E$  contains $\tau$, $\lk(\tau,B\sqcup (\tau\cdot X))= X$, and hence we can 
%iteratively 
apply 0-ExtP to $B$ to deduce that
%, for each $M$-independent set $X$, 
%we can augment from $B$ to 
$B\sqcup(\tau \cdot X)$ 
%keeping $N$-independence  by ExtP.
is independent in $N$. This implies that  $\tau\cdot X$ is independent in $N/E$, for every $M$-independent set $X\subseteq V$.

To verify the converse direction, we suppose for a contradiction that some $X\subseteq V$ is $M$-dependent, and $\tau\cdot X$ is independent in $N/E$.
Then, $B\sqcup (\tau\cdot X)$ is independent in $N$.
We can extend $X$ to an $M$-spanning set $Y$ by adding $t:=r_M(V)-r_M(X)$ elements $y_1,y_2,\ldots, y_t$ of $V$ to $X$.
Since 
%$X=\lk(\tau,B+\tau\cdot X)$,
$\lk(\tau,B\sqcup (\tau\cdot Y))= Y$,
we can iteratively apply ExtP to $B\sqcup (\tau\cdot (X\cup \{y_1,y_2,\ldots,y_i\}))$ for $0\leq i\leq t-1$,   to deduce that $B\sqcup (\tau\cdot Y)$ is independent in $N$.
%\st{COMMENT: $\uparrow$ This is where ExtP is needed.}
%by ExtP.
In addition, since $X$ is $M$-dependent, we have
\begin{equation}\label{eq:ext_star}
|Y|\geq r_M(V)+1.
\end{equation}

On the other hand, we can construct a canonical subset of $\sym_k(V)$ which has the form $S=B'\cup(\tau\cdot Z)$
for some $B'\subseteq E$ and some base $Z$ of $M$ as follows. We recursively construct  $S$ starting from the root vertex $\alpha_1$ in its construction tree by putting $X_{\alpha_1}=V$, $k_{\alpha_1}=k$, and  choosing the first letter in $\tau$ as our first pivot vertex for $S$. The recursive construction  proceeds by  choosing the next unused letter in $\tau$ as the pivot vertex at the rightmost leaf of the current tree and making an arbitrary choice for the pivot vertices at the other leaves. The rightmost leaf $\lambda^*$ in the final construction tree $T$ for $S$ will have $\tau_{\lambda^*}=\tau$ and $X_{\lambda^*}=V$. We can now put $Z=Y_{\lambda^*}$ and $B'=S\sm (\tau\cdot Z)$. 

Since $N$ satisfies 0-ExtP, ConingP, and RRP,  it also satisfies CBP by Lemma~\ref{lem:R_CBP}.
Hence, $S$ is a base of $N$.
Since $B$ is a base of $N|_{E}$ and $B'$ is $N$-independent with $B'\subseteq E$,
$|B|\geq |B'|$.
Also, by (\ref{eq:ext_star}), $|Y|>|Z|$.
Therefore, $|B\sqcup(\tau\cdot Y)|>|B'\sqcup(\tau\cdot Z)|=|S|$. This contradicts the facts that
$B\sqcup (\tau\cdot Y)$ is $N$-independent and $S$ is a base of $N$
and completes the proof of the lemma.
\end{proof}

%\section{Proofs of Theorems~\ref{thm:2} to \ref{thm:dual_k_weak}}
\section{Proofs of Main Results}
\label{sec:positive}
We have already verified Theorems~\ref{thm:2_weak}(b), \ref{thm:k}(b),  %\ref{thm:k_weak}(b) 
and  \ref{thm:mason_k_3}(b) by proving   Theorems~\ref{thm:example}, \ref{thm:rank_zero}, 
%\ref{thm:example1}, 
and \ref{thm:counterexample}, respectively. We will verify the remaining parts of Theorems~\ref{thm:2} to 
%\ref{thm:k_weak} 
\ref{thm:mason_k_3} by proving their dual statements, Theorems~\ref{thm:dual_2} to %\ref{thm:dual_k_weak}
\ref{thm:dual_mason_k_3}, and then applying  Lemma \ref{lem:duality}. 
%We will also show that ...

%\subsection{General case}
\subsection{Proof of Theorem~\ref{thm:dual_k_weak}}
%\begin{proof}[Proof of Theorem~\ref{thm:dual_k_weak}]
To verify the equivalence of the three bullet points in the statement of Theorem~\ref{thm:dual_k_weak}
we first note that, if CBP holds for $N$, then 0-ExtP, RRP, StrongRRP and ConingP all hold for $N$ by
Lemmas~\ref{lem:weakstrongCBP},  \ref{lem:weakCBP} and \ref{lem:strongCBP}. %\bill{first reference should give lemma 5.3 but instead it comes out as lemma 5.2???}
 In addition,
Lemma~\ref{lem:R_CBP} tells us that, if $N$ satisfies RRP, 0-ExtP, and ConingP, then $N$ satisfies CBP.
It remains to consider the case when  $N$ satisfies StrongRRP and 0-ExtP.
Then $N$ satisfies  ConingP by Lemma~\ref{lem:coning}, and
hence $N$ satisfies RRP, 0-ExtP, and ConingP.

The final part of the statement of Theorem~\ref{thm:dual_k_weak} now follows by combining the equivalence between the first and third bullet points of this theorem with Lemmas~\ref{lem:weakstrongCBP} and \ref{lem:weakCBP}.
%This completes the proof of the theorem.
%\end{proof}
\qed

\subsection{Proof of Theorem~\ref{thm:dual_k}}
%\begin{proof}[Proof of Theorem~\ref{thm:dual_k}]
%We first note that if $N$ satisfies $DM$ then $N$ satisfies CoCP and ExtP by  
%DB $\Rightarrow$ CoCP $\Rightarrow$ ExtP
%by the matroid duality from Lemma~\ref{lem:duality} and %(\ref{eq:biliearity_implication}).
Suppose $N$ satisfies RRP and ExtP.
Then $N$ satisfies DM  by Lemma~\ref{lem:ext_star}.
We can now use (\ref{eq:dual_biliearity_implication}) to deduce that $N$ satisfies CoCP, ExtP and 0-ExtP.
Lemma \ref{lem:extension_coning} now tells us that $N$ satisfies ConingP and hence we can use Lemma \ref{lem:R_CBP} to deduce that $N$ satisfies CBP. Lemmas~\ref{lem:weakstrongCBP} and \ref{lem:weakCBP} now give StrongRRP and CycP for $N$.
\qed

%\subsection{$k=2$}
\subsection{Proofs of Theorems~\ref{thm:dual_2}, \ref{thm:dual_2_weak} and \ref{thm:dual_mason_k_3}}
%The counterexamples of Mason's conjecture means that 
%CycP and DB does not imply RRP in general.
Both Theorems~\ref{thm:dual_2} and \ref{thm:dual_2_weak} concern the special case when $N$ is a matroid on $\sym_2(V)$.  We will need the following result which tells us that CycP implies a stronger property in this special case.

\begin{lemma}\label{lem:Mason_k=2}
Let $M=(V,r_M)$ be a matroid and $N$ be a matroid on $\sym_2(V)$ which satisfies CycP with respect to $M$.
Suppose $X\subseteq V$, and $S$ is a  canonical subset of $\sym_2(X)$.  
Then $S$ is a spanning set of $N|_{\sym_2(X)}$.
\end{lemma}
\begin{proof}
Let 
%$X=\{x_1,x_2,\ldots,x_m\}$ and 
$r_X=r_M(X)$. We will adopt  the notation and terminology given in Section \ref{sec:contree}. Let $T$ be a construction tree for $S$ and $P$ be the path in $T$ from its root $\alpha_1$ to its leftmost leaf node $\lambda_0$. We saw in the proof of Lemma \ref{lem:weakstrongCBP}(b) that every leaf node other than $\lambda_0$ is the right child of a node of $P$.
%In addition, we can 
Let $\lambda_0,\lambda_1, \ldots,\lambda_{m}$ be the enumeration of the leaf nodes of $T$ using the total order $\leq_T$.  To simplify notation we put
$X_i=X_{\lambda_i}$, $Y_i=Y_{\lambda_i}$,  $\tau_i=\tau_{\lambda_i}$ and $S_i=S_{\lambda_i}$ for all $0\leq i\leq m$. Then (\ref{eq:leaf}) gives $S=\bigsqcup_{i=0}^{m}S_{i}$. Relabelling the vertices in $X$, we may assume that $(v_1,v_2,\ldots,v_{m})$ is the sequence of  pivot vertices we encounter when traversing $P$ from $\alpha_1$ to $\lambda_0$ (so $m=|X|-r_X$).
Then: 
\begin{itemize}
    \item $X_{i}=X\setminus\{v_{1},\ldots,v_{m-i}\}$; 
    \item $S_{0}=\sym_2(X_0)$; 
    \item for $1\leq i\leq m$, $Y_{i}$ is a base of $M|_{X_{i}}$,
    %$Y_{i}=\sym_1(B_i)$ where $B_i$ is a base of $M|_{X_{i}}$,
    $\tau_{i}=v_{m-i+1}$ and $S_{i}=\tau_{i}\cdot Y_{i}$.
\end{itemize}
%$\sym(v_i, Y_i)$,
%where   $X=\{v_1, v_2, \dots, v_k\}$ and
%$Y_i\subseteq \{v_1,\dots, v_i\}$ is an $M$-independent set
%of size $i$ if $i\leq r_X$ and of size $r_X$ if $i>r_X$.
% Since $|B|=r_X|X|-{r_X\choose 2}$,
% we further have that 
% $Y_i=\{v_1,\dots, v_{i-1}\}$ if $i\leq d_X+1$
% and $|Y_i|=d_X$ if $i\geq d_X+2$.
%In particular, $\sym_2(\{v_1,\dots, v_{r_X}\})\subseteq B$.

Let $S_{\leq j}=\bigsqcup_{i=0}^j S_{i}$ for  $j=0,\dots, m$.
%and $B_0=\emptyset$.
We will use induction on $j$ to show that 
\begin{equation}\label{eq:Bj}
\text{$S_{\leq j}$ is a spanning set of $N|_{\sym_2(X_{j})}$ for all $0\leq j\leq m$.}
\end{equation}
Since $S_0=\sym_2(X_{0})$,
(\ref{eq:Bj}) trivially holds when $j=0$.
    Hence, we may assume inductively that  $S_{\leq j-1}$ is a spanning set of 
    $N|_{\sym_2(X_{j-1})}$ for some $1\leq  j\leq m$.

Since $v_{m-j+1}$ is an element of $X_j$ which is not a coloop in $M|_{X_j}$ and $Y_j$ is a base of $M|_{X_j}$, 
we can choose a circuit $C$ of $M|_{X_j}$ and a vertex $v_{i}\in X_j$  such that  $v_{i}, v_{m-j+1}\in C\subseteq Y_j+v_{i}$. 
%\bj{changed $i'$ to $i$.}
(If $v_{m-j+1}\notin Y_j$, then we choose $C$ to be the fundamental circuit of $Y_j+v_{m-j+1}$  and put $v_{i}=v_{m-j+1}$.
If $v_{m-j+1}\in Y_j$, then we choose $C$ to be a circuit of $M|_{X_j}$ with $v_{m-j+1}\in C$ and $|C\setminus Y_j|$ as small as possible. The circuit elimination axiom then implies that $|C\setminus Y_j|=1$ and we take $v_i$ to be the unique vertex in $C\setminus Y_j$.)
%Since $j> r_X$, $v_j$ is not a coloop in $M|_X$.  Since $Y_j$ is a base of $M|_X$, 
%\bj{there exists a circuit $C$ of $M|_X$ with $v_j\in C\subseteq \{v_1,v_2,\ldots,v_j\}$}
%we can choose a circuit $C$ of $M|_X$ and a vertex $v_{i}\in \{v_1,v_2,\dots,v_j\} $  such that  $v_{i}, v_j\in C\subseteq Y_j+v_{i}$. \bj{changed $i'$ to $i$.}
%(If $v_j\notin Y_j$, then we can choose $v_{i}=v_j$;
%otherwise, choose a circuit $C$ of $M|_X$ with $v_j\in C$ such that $|C\setminus Y_j|$ is as small as possible.)
By CycP, $\sym_2(C)$ is cyclic in $N$. 
Since all edges of $\sym_2(C)$ except $v_{i}v_{m-j+1}$ are contained in $\sym_2(X_{j-1})\cup (v_{m-j+1}\cdot Y_j)$, we have
$v_{i}v_{m-j+1}\in {\rm cl}_N(S_{\leq j})$.

We next choose an arbitrary vertex $w\in X_j\setminus Y_j$.
Let $C_w$ be the fundamental circuit of $Y_j+w$.
%Since $Y_j$ is a base of $M|_X$, 
%$Y_j+w$ contains an $M$-circuit $C_w$ with $w\in C_w$.
Then, $C_w\cup C$ is cyclic in $M$.  Since $N$ satisfies  CycP, $\sym_2(C_w\cup C)$ is cyclic in $N$.
Since all edges of $\sym_2(C_w\cup C)$ except $wv_{m-j+1}$ are contained in $\sym_2(X_{j-1})\cup (v_{m-j+1}\cdot Y_j)\cup\{v_{i}v_{m-j+1}\}$, we have
${w}v_{m-j+1}\in {\rm cl}_N(S_{\leq j}+v_{i}v_{m-j+1})={\rm cl}_N(S_{\leq j})$. %\bj{Changed from $v_iv_j\in ...$.}
Since this holds for every $w\in X_j\setminus Y_j$,
 %$S_{\leq j}=S_{\leq j-1}\cup (v_j\cdot Y_j)$
 $S_{\leq j}=S_{\leq j-1}\cup (v_{m-j+1}\cdot Y_j)$ is a spanning set of $N|_{\sym_2(X_j)}$ and (\ref{eq:Bj}) holds.  The   lemma now follows by taking $j=m$.
\end{proof}

%\paragraph{Proof of Theorem \ref{thm:Mason_k=2}}
%\begin{proof}[Proof of Theorem \ref{thm:dual_2_weak}]
%By Lemma~\ref{lem:uniform_CP}, CP implies CyCP. Also, 
\subsubsection*{Proof of Theorem \ref{thm:dual_2_weak}}
Theorem~\ref{thm:dual_k_weak} implies that each of the last two bullet points of Theorem \ref{thm:dual_2_weak} is equivalent to the statement that $N$ satisfies CBP, and Lemma~\ref{lem:weakstrongCBP}(b) tells us that CBP is equivalent to WCBP.
Theorem~\ref{thm:dual_k_weak} also implies that, if $N$ satisfies CBP, then $N$ satisfies CycP, so each of the last two bullet points of Theorem \ref{thm:dual_2_weak} implies the first two bullet points.
It remains to show that each of the first two bullet points 
%CycP and RRP or CycP and 0-ExtP 
implies WCBP.

We first suppose that $N$ satisfies RRP and CycP. 
%Let $X\subseteq V$.
%By Lemma~\ref{lem:Mason_k=2}, every canonical subset $S$ of $X$
%is a spanning set of $N|_{\sym_2(X)}$.
%In addition,  Lemma~\ref{lem:canonical_size} implies that $S$ has  size equal to the rank value in RRP. \bill{IT SEEMS TO ME WE NEED StrongRRP HERE}
%Hence, $S$ is a base of $N|_{\sym_2(X)}$, and WCBP follows.
By Lemma~\ref{lem:Mason_k=2}, each canonical subset $S$ of $\sym_2(V)$
is a spanning set of $N$. 
In addition,  Lemma~\ref{lem:canonical_size} implies that $S$ has  size equal to the rank of $N$ by RRP. Hence, $S$ is a base of $N$.  
To verify WCBP, we need to show that the corresponding property holds for $N|_{\sym_2(X)}$ for every $X\subseteq V$. %consider any 
We will accomplish this by first showing that $N$ satisfies StrongRRP. 

Choose a set $X\subseteq V$ and a base $B_X$ of $M|_X$, and extend $B_X$ to a base $B_V$ of $M$.
Let $<$ be a total order on $V$ which is both $B_X$-first and $B_V$-first.
By Lemma~\ref{lem:lex_min}, $B_X\cdot \sym_{1}(X)$ is the lex-min canonical subset of $\sym_2(X)$
and $B_V\cdot \sym_{1}(V)$ is the lex-min canonical subset of $\sym_2(V)$.
Then $B_V\cdot \sym_{1}(V)$ is a base of $N$ by the previous paragraph. Since $B_X\cdot \sym_{1}(X)\subseteq B_V\cdot \sym_{1}(V)$, 
$B_X\cdot \sym_{1}(X)$ is independent in $N$.
Lemma~\ref{lem:Mason_k=2} now implies that $B_X\cdot \sym_{1}(X)$  is a base of $N|_{\sym_2(X)}$.
Since $B_X\cdot \sym_{1}(X)$ is a canonical set of $\sym_k(X)$, 
Lemma~\ref{lem:canonical_size} tells us that $N|_{\sym_2(X)}$ satisfies the rank condition in  StrongRRP.

To verify WCBP,  choose $X\subseteq V$ and consider a canonical subset $S$ of $\sym_2(X)$.
%By Lemmas~\ref{lem:canonical_construction}, 
By Lemma~\ref{lem:canonical_size}, Lemma~\ref{lem:Mason_k=2}, and StrongRRP,
$S$ is a base of $N|_{\sym_2(X)}$.
Thus, WCBP holds.

\smallskip

Suppose, on the other hand, that $N$ satisfies CycP and 0-ExtP. Choose $X\subseteq V$ and let $S$ be a canonical subset of $\sym_k(X)$.
 Lemma~\ref{lem:Mason_k=2} implies that $S$ is a spanning set for $N|_{\sym_2(X)}$. In addition, Lemma~\ref{lem:canonical_construction} and the assumption that $N$ satisfies 0-ExtP imply that $S$ is $N$-independent. Hence $S$ is a base of $N|_{\sym_2(X)}$, and WCBP follows.
%\end{proof}
\qed

%\begin{proof}[Proof of Theorem~\ref{thm:dual_2}]
\subsubsection*{Proof of Theorem~\ref{thm:dual_2}} 
We first note that the first two bullet points in  the statement of Theorem~\ref{thm:dual_2} are equivalent since $k=2$. In addition, if the last bullet point holds, then $N$ satisfies the properties in all the other bullet points by  Theorems~\ref{thm:dual_k} and \ref{thm:dual_2_weak}, and the fact that DM implies CoCP by (\ref{eq:dual_biliearity_implication}). It remains to show that each of the second, third, and fourth bullet points imply the last bullet point. This follows immediately from Theorem~\ref{thm:dual_2_weak} and (\ref{eq:dual_biliearity_implication}).
%
%Suppose that the fourth bullet point holds i.e. $N$ satisfies RRP and ExtP. Then $N$ satisfies ConingP and 0-ExtP by Lemmas~\ref{lem:extension_coning} and \ref{lem:ExtPto0ExtP}, and we can now apply Theorems~\ref{thm:dual_k} to deduce that $N$ satisfies CycP so the last bullet point holds.
%
%Suppose that the third bullet point holds i.e.~$N$ satisfies RRP and CoCP. Then the fourth bullet point holds since CoCP implies ExtP (by (\ref{eq:biliearity_implication}) and Lemma \ref{lem:duality}). Hence the last paragraph holds by the previous paragraph.
%
%Suppose that the second bullet point holds i.e.~$N$ satisfies RRP and DM.
%\end{proof}
\qed

\medskip
%\subsection{Further special case}
%We will show that Mason's conjecture holds when $k=3$ and $M$ is a non-free %uniform matroid.  More precisely, we will show that the dual counterpart holds %when $k=3$ and $M=(V,r_M)$ is a uniform matroid of rank greater than zero. 
%This is in contrast to the result of Subsection~\ref{subsec:counterexample4}, %which provides a negative example in the same setting for $k=4$.

We next turn  to the proof of Theorem \ref{thm:dual_mason_k_3}. We will need the following weakened version of Lemma \ref{lem:Mason_k=2} for  matroids on $\sym_3(V)$ in the special case when $M$ is a uniform matroid.

\begin{lemma}\label{lem:CycP3}
Suppose $M=(V,r_M)$ is a copy of the uniform matroid $U_n^r$ with $r\geq 1$ and $N$ is a matroid on $\sym_3(V)$ which satisfies CycP with respect to $M$. Then, for each $X\subseteq V$, 
there exists a canonical subset of $\sym_3(X)$ which is a spanning set in $N|_{\sym_3(X)}$.
\end{lemma}
\begin{proof}
We proceed by induction on $|X|$.
If $X$ is $M$-independent, then $\sym_3(X)$ is a canonical subset of itself and the lemma holds.
Hence, we may assume that $X$ is dependent in $M$. This implies that $|X|\geq r+1\geq 2$.
%Since $M$ has rank at least one, this implies that $X$ is nonempty.

Choose $v\in X$. Since $M=U_n^r$ and  $|X|\geq r+1$, 
$v$ is not a loop or a coloop in $M|_X$.
By induction, there exists a canonical subset $S_{X-v}$ of $\sym_3(X-v)$ which is a spanning set in $N|_{\sym_3(X-v)}$.
Since $v$ is not a loop in $M$, there is a base $B_X$ of $M|_X$ with $v\in B_X$.
Then $B_X\cdot X$ is a canonical subset of $\sym_2(X)$. (More precisely, it is the lex-min canonical subset of $\sym_2(X)$ with respect to any $B_X$-first total order of $X$ by Lemma \ref{lem:lex_min}.)
Hence
\begin{equation}\label{eq:CycP30}
S_X:=S_{X-v}\sqcup (v\cdot B_X\cdot X)
\end{equation}
is a canonical subset of $\sym_3(X)$ by the definition of a canonical set.
Our goal is to show that $S_X$ is a spanning set of $N|_{\sym_3(X)}$.
To accomplish this, we consider the following partition of $\sym_3(X)$.
\begin{equation}\label{eq:CycP3}
\begin{split}
\sym_3(X)=&\sym_3(X-v)\sqcup (v\cdot B_X\cdot X) \,
\sqcup \\
&\{vuu:  u\in X\setminus B_X\} 
\sqcup \{vu_1u_2:  u_1, u_2\in X\setminus B_X, u_1\neq u_2\}. 
\end{split}
\end{equation}
We will show that each set on the right hand side of (\ref{eq:CycP3}) is contained in ${\rm cl}_N(S_X)$.

Since $S_{X-v}$ spans $\sym_3(X-v)$, we have
\begin{equation}\label{eq:CycP31}
   \sym_3(X-v)\subseteq {\rm cl}_N(S_{X-v})\subseteq {\rm cl}_N(S_X).
\end{equation}

To see that
\begin{equation}\label{eq:CycP32}
    vuu\in {\rm cl}_N(S_X) \quad \text{for each } u\in X\setminus B_X,
\end{equation}
we first observe that, since $M$ is uniform, $B_X+u$ is a circuit of $M$ for each $u\in X\setminus B_X$.
In addition,
the only edge in $\sym_3(B_X+u)$ which is missing from $\sym_3(X-v)\sqcup (v\cdot B_X\cdot X)$ is $vuu$.
(Note that $vvu\in v\cdot B_X\cdot X$ since $v\in B_X$.)
Hence, by (\ref{eq:CycP30}), (\ref{eq:CycP31}), and CycP (applied to the $M$-circuit  $B_X+u$), 
$vuu$ is spanned by $S_X$, and (\ref{eq:CycP32}) follows.

We will use a similar argument to show that
\begin{equation}\label{eq:CycP33}
    vu_1u_2\in {\rm cl}_N(S_X) \quad \text{for all distinct } u_1, u_2\in X\setminus B_X.
\end{equation}
 Observe that the only edge in $\sym_3(B_X+u_1+u_2)$ which is missing from   $$\sym_3(X-v)\sqcup (v\cdot B_X\cdot X)\sqcup \bigcup_{u\in X\sm B_X} \sym_3(B_X+u)$$ is $vu_1u_2$.
Since $\bigcup_{u\in X\sm B_X} \sym_3(B_X+u)\subseteq {\rm cl}_N(S_X)$ by the previous paragraph and    $B_X+u_1+u_2$ is $M$-cyclic, we can apply CycP to $\sym_3(B_X+u_1+u_2)$ to deduce (\ref{eq:CycP33}).

By (\ref{eq:CycP3}), (\ref{eq:CycP31}), (\ref{eq:CycP32}), and (\ref{eq:CycP33}),
${\rm cl}_N(S_X)=\sym_3(X)$.
\end{proof}

%Using the extra implications given in the lemmas above,
%we have the following theorem, which gives a positive answer to Mason's conjecture for uniform matroids $M$ and $k=3$.
% \begin{theorem}\label{thm:abstract_rigidity_k_3}
% Let $M$ be a uniform matroid on $V$ with rank at least one 
% and $N$ be a matroid on $\sym_3(V)$.
% Then, any two pair of the properties in 
% $\{$RP, CP, CycP, 1GP, CoCP, ExtP, 1-ExtP, StarP$\}$,
% except for the pairs in $\{$CoCP, ExtP, 1-ExtP, StarP$\}$,
% implies all the properties.
% \end{theorem}

Theorem \ref{thm:dual_mason_k_3} will follow immediately from:

\begin{theorem}\label{thm:abstract_rigidity_k_3}
Let $M$ be a uniform matroid with rank at least one. 
Suppose $N$ is a matroid on $\sym_3(V)$ which satisfies at least one
property in $\{$RRP,  CycP$\}$
and at least one  property in $\{$ExtP, CoCP, DM$\}$. Then $N$ satisfies 
RRP,  CycP, ExtP, CoCP and DM.
% \begin{itemize}
% %\item $N$ satisfies 1GP and 0-ExtP.
% %\item $N$ satisfies RP, CP, and DCP.
% %\item $N$ satisfies RP and CoCP.
% \item $N$ satisfies RP and ExtP.
% \item $N$ satisfies RP and CoCP.
% \item $N$ satisfies RP and StarP.
% %\item $N$ satisfies CP, DCP and 0-ExtP.
% \item $N$ satisfies CycP and ExtP.
% \item $N$ satisfies  CycP and CoCP.
% \item $N$ satisfies CycP and StarP.
% \end{itemize}
\end{theorem}
\begin{proof}
By Theorem~\ref{thm:dual_k} and (\ref{eq:dual_biliearity_implication}), it will suffice to prove that 
CycP and ExtP imply RRP.

Suppose $N$ satisfies CycP and ExtP.
Then, $N$ satisfies 0-ExtP by (\ref{eq:dual_biliearity_implication}).  
Lemmas~\ref{lem:canonical_size}, \ref{lem:canonical_construction}, and~\ref{lem:CycP3} now imply that
$N$ satisfies RRP.
\end{proof}

% Theorem \ref{thm:abstract_rigidity_k_3} tells us that, if $N$ satisfies StarP and CycP, then $N$ satisfies $RP$. This, in turn, implies that Conjecture \ref{conj:dual} holds when $k=3$ and $M$ is a uniform matroid of rank greater than zero, and hence, by matroid duality, Conjecture \ref{conj:And} holds when $k=3$ and $M$ is a uniform matroid other than the free matroid. 
\section{Concluding Remarks}

\subsection{Third symmetric powers}
Theorem~\ref{thm:rank_zero} implies that Mason's rank conjecture is false for third symmetric powers of the free matroid. We can construct other counterexamples by taking the direct sum of any matroid which has a third symmetric power, with the free matroid. On the other hand,  Theorem~\ref{thm:dual_mason_k_3}  tells us that Mason's rank conjecture holds for third symmetric powers of uniform matroids, other than the free matroid. It may be true that there exist other families of matroids whose third symmetric powers satisfy  Mason's rank conjecture. It  is even conceivable that third symmetric powers  of every connected matroid satisfy Mason's rank conjecture. 

\subsection{Unique maximality of symmetric powers}
We close by considering a second question of Mason. He says in the introduction to his paper~\cite[Pages 519,520]{M}
\begin{quote}
    The outstanding questions are whether there is always a free-est matroid having the geometrical properties of a tensor product, exterior or symmetric powers, and even if there is not a single free-est such in each case, how can at least one maximally free matroid be constructed. 
\end{quote}

A more detailed version of the special case of this question for symmetric powers is given on \cite[Page 554]{M}. In the terminology of our paper, it  can be restated  as: 
\begin{quote}
    The outstanding questions for symmetric powers [of an arbitrary matroid $M$] are whether every $k$-th quasi-symmetric power  of $M$ which satisfies the Flat Property has 
    rank $\binom{r(M)+k-1}{k}$, and if so, whether there is  a unique maximal $k$-th quasi-power of $M$ with respect to the weak order of matroids.
\end{quote}

%\st{COMMENT: Mason posed this maximality question for QUASI-powers in page 554, but there is a more general question in page 514, which seems asking the maximality within the class of symmetric powers. The reason why I am pointing this is that the maximality question for the class of quasi-powers is known to be false. This was shown by Las Vergnas~\cite{LV81} for tensor quasi-products of a uniform matroid. His example should also give a counterexample for symmetric quasi-powers by looking at the direct sum of two uniform matroids. Our unpublished paper on $K_{1,t}$-matroids also gives a counterexample to the case when the underlying matroid is uniform.}

We have seen that the `rank part' of Mason's question on \cite[Page 554]{M} has a positive answer when $k=2$ and a negative answer for $k\geq 3$. The `unique maximal' part of this question is false even for $k=2$. This was shown by Las Vergnas~\cite{LV81} for quasi-products of  uniform matroids, and his result implies that the direct sum of two uniform matroids will not have a unique maximal second  symmetric quasi-power.
On the other hand, the version of Mason's unique maximality question given on 
~\cite[Pages 519,520]{M} is still wide open. 

\begin{question}[Mason's Maximality Question] Suppose $k\geq 2$ is an integer. Does every matroid which has a $k$-th symmetric power, have a unique maximal 
$k$-th symmetric power with respect to the weak order of matroids?
\end{question}

This question seems to be difficult even for the special case of second symmetric powers of the uniform matroid $U_n^d$. A plausible construction for a unique maximal second power $N$ of $U_n^d$ would be to choose a set $V$ of $n$ generic points $p_1,p_2,\ldots p_n$ in $\R^d$ and take $N$ to be the linear matroid on $\sym_2(V)$ represented by the symmetric tensor product of $p_i$ and $p_j$.
%A combination of  Lemma \ref{lem:weak_construction} and \cite[Lemma 6 and Theorem 3]{JTfields} implies  
We can use \cite[Lemmas 3.1, 3.2]{JTmax} to deduce that $N$ is indeed the unique maximal second power of $U_n^d$ when $d\in \{1,2\}$,
but we show in \cite[Theorem 5.28]{cruickshank2025rigidity} that this is  false when $d\geq 7$.

%\st{The negative part is okay, but I am not sure about the positive side. 
%There are two points I want to share: 
%\begin{itemize}
 %   \item For a general matroid, Lemma 3.2 is valid only for constructing matroids in Lovasz's class. As far as I checked before, the corresponding statement for Anderson's class works if $M$ is uniform. However, at least, we need a proof for this.
   % \item Even if an analogue of Lemma 3.2 is valid for Andrson's class, there is still subtlety. 
  %  The combination of Lemma 3.2 (for Anderson's class) and \cite[Theorem 3]{JTfields} shows that 
 %   $({\cal B}\vee N)^*$ is unique maximal in Anderson's class, where $N$ is the generic rigidity matroid of an appropriate dimension.
 %   It seems to me that $({\cal B}\vee N)^*$ is distinct from the geometric symmetric-tensor-product matroid.
 %   The dual of the geometric symmetric-tensor-product matroid is the symmetric-%matrix-completion matroid, which seems weaker than ${\cal B}\vee N$...
%\end{itemize}
%}

\section*{Acknowledgements}

This work was supported by the Japan Science and Technology Agency (JST) as part of Adopting Sustainable Partnerships for Innovative Research Ecosystem (ASPIRE), Grant Number JPMJAP2520.

\bibliographystyle{amsplain}
\bibliography{product_ref}

\end{document}